\documentclass[11pt]{article}

\usepackage{jcd}
\usemedgeometry
\bluehyperref
\usepackage[numbers]{natbib}

\usepackage{macros}
\usepackage{overpic}

\begin{document}

\title{Distribution free M-estimation}
\author{Felipe Areces ~~~~~~ John Duchi\thanks{Partially supported by
    the Office of Naval Research Grant N00014-22-1-2669.
    This work began as
    part of the ``Algorithmic Stability'' workshop
    hosted by the American Institute of Mathematics (AIM), May 2025.} \\
  Stanford University}
\date{November 2026}
\maketitle

\vspace{-1cm}

\begin{abstract}
  The basic question of delineating those statistical problems that
  are solvable without making any assumptions on the underlying
  data distribution has long animated statistics and learning theory.
  This paper characterizes when a convex M-estimation or
  stochastic optimization problem is solvable in such an assumption-free
  setting, providing a precise dividing line between solvable and unsolvable
  problems.
  The conditions we identify show, perhaps surprisingly, that Lipschitz
  continuity of the loss being minimized is not necessary for distribution
  free minimization, and they are also distinct from classical characterizations
  of
  learnability in machine learning.
\end{abstract}


\section{Introduction}

Consider a general convex M-estimation problem, where for a population
distribution $P$ on a set $\mc{Z}$ and a convex loss $\loss_z(\theta)$
measuring the performance of a parameter $\theta$ on example $z$,
we wish to minimize the population loss
\begin{equation}
  \poploss_P(\theta) \defeq \E_P[\loss_Z(\theta)]
  = \int \loss_z(\theta) dP(z)
  \label{eqn:pop-loss}
\end{equation}
over $\theta$ belonging to a (known) convex parameter space $\Theta$.
We are interested in truly distribution free minimization of this population
loss, meaning that we would like to be able to (asymptotically) minimize
$\poploss_P$ given i.i.d.\ observations $Z_i$ drawn from $P$ without making
\emph{any} assumptions on the distribution $P$.
By this, we mean we know essentially the bare minimum from a statistical
perspective: only (i) the loss $\loss$, (ii) the set $\mc{Z}$, and
(iii) the parameter space $\Theta$.
Our main contribution will be to delineate those situations in which
minimizing the loss~\eqref{eqn:pop-loss} is possible from those in which it
is not.

To do so, we require a bit more formality.
Let $\interior\Theta$ be the interior of the set $\Theta$.
For each $z \in \mc{Z}$, the
loss $\loss_z : \R^d \to \R \cup \{+\infty\}$ is closed convex and
proper, and to avoid trivialities, we assume that $\loss_z(\theta) < \infty$
for all $\theta \in \interior \Theta$ and is measurable in $z$.
%
%
Defining the minimal loss $\poploss_P\opt(\Theta) \defeq
\inf_{\theta \in \Theta} \poploss_P(\theta)$ and letting $\mc{P}(\mc{Z})$ be
the collection of (Borel) probability measures on $\mc{Z}$,
we give conditions under which
the \emph{minimax optimization risk} for the loss $\loss$,
\begin{equation}
  \minimax_n(\loss, \mc{Z}, \Theta)
  \defeq \inf_{\what{\theta}_n}
  \sup_{P \in \mc{P}(\mc{Z})}
  \E_{P^n}
  \left[\poploss_P(\what{\theta}_n(Z_1^n))
    - \poploss_P\opt(\Theta)\right],
  \label{eqn:minimax-opt}
\end{equation}
approaches 0 as $n$ grows---meaning M-estimation (optimization)
is possible---or, conversely, is bounded away from 0.
The definition~\eqref{eqn:minimax-opt}
takes the infimum over measurable
estimators $\what{\theta}_n : \mc{Z}^n \to \Theta$ and the supremum
over distributions $P$ generating $Z_1^n = (Z_1, \ldots, Z_n) \simiid P$.
Crucially, in the minimax risk~\eqref{eqn:minimax-opt}, the
loss $\loss$ and sets $\mc{Z}$ and $\Theta$ are fixed,
so that any lower bound holds not because we may construct a worst-case
function, but because of properties the instance at hand
actually enforces.
%
%
This makes clear the sense in which
we consider an estimator to be distribution free: it must
achieve small excess risk $\poploss_P(\what{\theta}) - \poploss_P\opt$
uniformly over sampling distributions $P$, for a given triplet $(\loss, \mc{Z}, \Theta)$.
(We refine definition~\eqref{eqn:minimax-opt} to address
integrability questions when we give formal results.)

If the loss $\loss_z$ is Lipschitz for each $z \in \mc{Z}$, with the
same Lipschitz constant $\lipconst$, then well-known convergence guarantees
for stochastic gradient algorithms~\citep{NemirovskiJuLaSh09}
show that
\begin{equation*}
  \minimax_n(\loss, \mc{Z}, \Theta)
  \le \frac{\lipconst \diam(\Theta)}{\sqrt{n}},
\end{equation*}
where 
$\diam(\Theta)\defeq \sup_{\theta,\theta'\in \Theta}\ltwo{\theta-\theta'}$.
These guarantees are sharp in that
there exist sets $\mc{Z}$ and
$\lipconst$-Lipschitz convex losses $\loss$ for which
$\minimax_n(\loss, \mc{Z}, \Theta) \ge c \frac{\lipconst
  \diam(\Theta)}{\sqrt{n}}$, where $c > 0$ is a numerical
constant~\citep{RaginskyRa11, AgarwalBaRaWa12, DuchiJoMc13, Duchi18}.
In contrast to these ``classic'' lower bounds, which find worst-case
losses and sets $\mc{Z}$ to demonstrate tightness of
convergence results, we provide a precise dividing
line, based on the
properties of the particular loss $\loss$ and its behavior on the
space $\mc{Z} \times \Theta$, for the separation
\begin{equation*}
  \inf_n \minimax_n(\loss, \mc{Z}, \Theta) > 0
  ~~ \mbox{versus} ~~
  \lim_n \minimax_n(\loss, \mc{Z}, \Theta) = 0.
\end{equation*}
(Of course, in the generality we consider here, there is no hope of getting
a convergence rate, as we will show in the sequel.)

We will give the dividing lines
separating these two cases both when $\Theta$ is compact and
when it is unbounded.
In the compact case, our main results show that the following condition
very nearly provides this division:

\conditionbox{\label{cond:compact-case}
  For each compact subset $\Theta_0 \subset \interior \Theta$,
  the functions $\loss_z(\cdot)$ restricted to $\Theta_0$
  are uniformly Lipschitz:
  there exists $\lipconst = \lipconst(\Theta_0) < \infty$ such that
  for each $z \in \mc{Z}$,
  the function $\loss_z(\cdot)$ is $\lipconst$-Lipschitz continuous
  on $\Theta_0$.
}
\noindent
If the condition holds, then $\lim_n \minimax_n(\loss,
\mc{Z}, \Theta) = 0$: there exist algorithms that can solve the
minimization problem with vanishing risk uniformly over all $P$.
If it fails, then excepting some trivialities
about achievable minimizers that we elucidate,
no estimation or optimization procedure can achieve excess
risk tending to zero uniformly over distributions $P$ on $\mc{Z}$.

Questions of truly distribution free inference
date back at least to \citeauthor{BahadurSa56}'s 1956 work on the
impossibility of nonparametric inference of a mean~\citep{BahadurSa56}.
A thread of such work stretches through today; for example, \citet{Donoho88}
shows that two-sided confidence statements, i.e., of the form ``with 95\%
confidence, the underlying distribution $P$ has KL-divergence from the
uniform between 1 and 3'' are generally impossible.
More recently, \emph{conformal prediction} methods perform predictive
inference without making any assumptions on the underlying
distribution~\cite[e.g.][]{VovkGaSh05, Vovk13}.
At the most basic level, conformal prediction methods leverage the ability
to robustly estimate quantiles to provide (marginal) confidence statements
on the predictions of even completely black-box models~\citep{Lei14,LeiWa14,
  LeiGSRiTiWa18, BarberCaRaTi21a}.
This renewed focus on assumption free inference motivates the questions we
investigate in M-estimation.

The line of work on \emph{universal consistency} in nonparametric
regression, beginning with
\citet{CoverHa67}, asks related questions to ours, with a slightly
different focus.
There, the data $z =
(x, y)$ (with $y \in \R$ for regression
or $y \in \{0, 1\}$ for classification),
and one wishes to estimate
conditional expectation $f\opt(x) \defeq \E[Y \mid X = x]$.
An estimator $\what{f}_{P_n}$, denoting that
it is a function of the empirical distribution $P_n$
of $(X_i, Y_i)_{i=1}^n \simiid P$, is
\emph{consistent} if $\E[|\what{f}_{P_n}(X_{n+1}) - f\opt(X_{n+1})|] \to 0$ as
$n \to \infty$ whenever $\E[|Y|] < \infty$.
\citet{CoverHa67} show that $k$-nearest neighbor classifiers, where $k =
k_n \to \infty$ but $k_n / n \to 0$, are consistent for classification
problems, and \citet{Stone77} extends these results
to regression.
In distinction from our setting, however, these results are
pointwise (not uniform in the distribution $P$), as uniform
guarantees are impossible without making additional
assumptions on the regression function $f\opt$.
See also the books~\citep{DevroyeGyLu96, GyorfiKoKrWa02}.

In theoretical machine learning, researchers studying
\emph{learnability} investigate
problems similar to ours~\citep{Vapnik98, ShalevShSrSr10}.
One considers prediction problems with data $z = (x, y) \in \mc{X}
\times \mc{Y}$, and seeks a hypothesis $h : \mc{X} \to \mc{Y}$
to minimize $\poploss_P(h) \defeq
\E_P[\loss(h(X), Y)]$, where $\loss$ is a particular loss.
A problem is \emph{learnable} for the class $\mc{H}$ if
\begin{equation}
  \label{eqn:learnability}
  \sup_P \E_P\left[\poploss_P(\what{h}_{P_n}) - \inf_{h \in \mc{H}}
    \poploss_P(h)\right] \to 0,
\end{equation}
where $\what{h}_{P_n}$ is a function of the empirical distribution
$P_n$;
this convergence coincides with convergence of the minimax
risk~\eqref{eqn:minimax-opt}.
Early work considered particular losses
$\loss$;
in the case of binary classification, where $\mc{Y} = \{\pm 1\}$ and
$\loss(\hat{y}, y) = \indic{\hat{y} \neq y}$,
learnability~\eqref{eqn:minimax-opt} holds if and only if
$\mc{H}$ has finite VC-dimension,
which in turn coincides with the uniform convergence
that $\sup_{h \in \mc{H}} |\poploss_{P_n}(h) - \poploss_P(h)| \to 0$
(see~\cite[Ch.~19]{AnthonyBa99} or~\cite{AlonBeCeHa97}).
%
%
\citet{ShalevShSrSr10} generalize and extend these results to show that when
the loss $\loss(h(x), y)$ is uniformly bounded,
a class $\mc{H}$ is learnable~\eqref{eqn:learnability}
if and only if there
exist
leave-one-out stable procedures that
asymptotically perform empirical risk minimization,
i.e., $\E[\poploss_{P_n}(\what{h}_{P_n}) - \inf_{h \in \mc{H}}
  \poploss_{P_n}(h)] \to 0$ uniformly in $P$.

Condition~\ref{cond:compact-case} gives distinct conditions from this
prior work. 
We dedicate the remainder of the paper to demonstrating
(and extending) the ways in which
Condition~\ref{cond:compact-case} describes when M-estimation is possible
and impossible.
Section~\ref{sec:minimax} provides the main results, with
Section~\ref{sec:minimax-compact} addressing the case when $\Theta$ is a
compact set.
%
In Section~\ref{sec:unbounded-case}, we present
Theorem~\ref{theorem:unbounded-case}, which precisely characterizes when the
minimax risk for M-estimation and optimization can tend to zero when
$\Theta$ may not be compact
(providing a Condition~\ref{cond:unbounded-case} extending
\ref{cond:compact-case} for this case).
Perhaps surprisingly, given the known upper and lower bounds
using Lipschitz constants above,
these results make clear that global Lipschitz continuity of
the losses $\loss_z$ is neither sufficient nor necessary for distribution-free
minimization to be possible.

Our characterization of solvability is inherently finite dimensional,
so in Section~\ref{section:infinite-dimensional} we argue that 
our proposed techniques do not directly extend to infinite dimensions.
We consider the case where the parameter space is a closed and bounded 
subset of a separable Hilbert space and provide examples where 
natural analogues of Condition~\ref{cond:compact-case} fail but 
distribution free M-estimation is still possible.
We also identify a restricted linear symmetric setting 
where we can derive lower bounds in terms of the linear 
diameter of $\Theta\times \mc{Z}$.

The negative results in Theorem~\ref{theorem:unbounded-case}
apply to traditional cases of ``robust'' estimation,
such as the absolute loss $\loss_z(\theta) = |\theta - z|$,
corresponding to estimating a median.
Thus, we briefly consider the problem of finding stationary points of
$\poploss_P$ instead of points for which the excess loss $\poploss_P(\theta)
- \poploss_P\opt$ is small in Section~\ref{sec:stationary-points},
which relates more closely to questions of quantile prediction.
We do not characterize the separations in minimax risk, instead giving a
basic convergence result that applies to both differentiable and
non-differentiable functions that shows there are necessarily differences
between M-estimation, as we have defined it, and
obtaining stationary points.
Our discussion (see Section~\ref{sec:discussion}) builds out of these
differences, providing open questions and identifying alternatives to
the choice~\eqref{eqn:minimax-opt} of optimality criterion.

\paragraph{Notation}
We use $\pointmass_z$ to denote a point mass at the point $z$.
For a convex function $f$, $\partial f(\theta_0)$ denotes the
subdifferential of $f$ at $\theta_0$, meaning those vectors $g$ for which
$f(\theta) \ge f(\theta_0) + \<g, \theta - \theta_0\>$ for all $\theta$.
We let $\ball_2 = \{v \mid \ltwo{v} \le 1\}$ be the $\ell_2$ ball (in
a dimension that is clear from context), while $\sphere^{d-1} = \{v \in \R^d
\mid \ltwo{v} = 1\}$ denotes the sphere.
For a set $C$, $\interior C$ denotes its interior, 
$\diam(C)\defeq\sup_{x,y\in C}\ltwo{x-y}$ denotes its diameter, 
and $\dist(x,C)\defeq\inf_{y\in C}\ltwo{x-y}$ denotes the distance 
between a point and a set.
For a convex function $f$, we let 
$\dom f:=\{x:f(x)<\infty\}$.
We let $\Rup = \R \cup \{+\infty\}$ and $\Rdown = \R \cup \{-\infty\}$
denote the extended up (or down) reals.
For any set $A \subset \mc{Z}$, we let $\mc{P}(A)$ denote the collection of
Borel probability distributions on $A$.


\section{Minimax bounds on the risk}
\label{sec:minimax}

To make rigorous that Condition~\ref{cond:compact-case} neatly divides
solvable versus unsolvable problems, we require a few preliminary
definitions to avoid some minutiae, then provide the formal
minimax
lower bounds in Sections~\ref{sec:minimax-compact}
and~\ref{sec:unbounded-case}.


\subsection{Restricting to achievable minimizers}

We shall restrict analysis in the minimax lower bounds to what we term
\emph{achievable minimizers}, roughly meaning those points in $\Theta$ that
could plausibly minimize $\poploss_P(\theta) = \E_P[\loss_Z(\theta)]$ for
some $P$ on $\mc{Z}$.
This restricts little; in Appendix~\ref{sec:achievable-minimizers},
where we enumerate properties of and provide technical detail
on the achievable minimizers, we also see that
there are points in the set of achievable minimizers
with smaller risk than those outside the set.
%
We also warn the reader that this section simply helps us avoid
pathologies; it might be skipped on a first reading.

\subsubsection{Achievable minimizers in the one-dimensional case}

In the one-dimensional case, the intuition is simple.
Assuming temporarily that
$\loss$ is differentiable, we only consider points
$\theta$ for which there exist $z^+$ and $z^-$
satisfying $\loss_{z^+}'(\theta) > 0$ and $\loss_{z^-}'(\theta) < 0$.
Then immediately, $\theta$ minimizes $\poploss_P$ for a distribution
$P = p \pointmass_{z^+} + (1 - p) \pointmass_{z^-}$ for some
appropriate $p \in (0, 1)$: we cannot eliminate the
possibility that $\theta$ minimizes $\poploss_P$ for some $P$ without
observing data.
Because derivatives of convex functions are monotonic,
if $\loss_z'(\theta) \le 0$ for all $z \in \mc{Z}$, then for
any distribution $P$, there \emph{necessarily} exists a minimizer
$\theta\opt(P)$ of $\poploss_P(\theta) = \E_P[\loss_Z(\theta)]$ with
$\theta\opt(P) \ge \theta$, and similarly when $\loss_z'(\theta) \ge 0$
for all $z$.
We consequently define the minimal
and maximal plausible values
\begin{equation}
  \label{eqn:extreme-minimizers}
  \begin{split}
    \thetamin & = \thetamin(\loss, \mc{Z}, \Theta)
    \defeq \inf\left\{\theta \in \Theta
    \mid \mbox{there~exists}~ z \in \mc{Z}
    ~ \mbox{s.t.} ~
    \loss_z'(\theta) > 0 \right\} \\
    \thetamax & = \thetamax(\loss, \mc{Z}, \Theta)
    \defeq \sup\left\{\theta \in \Theta
    \mid \mbox{there~exists}~ z \in \mc{Z}
    ~ \mbox{s.t.} ~
    \loss_z'(\theta) < 0 \right\}
  \end{split}
\end{equation}
where if the set defining $\thetamin$ is empty we take the right
endpoint $\thetamin = \sup\{\theta \in \Theta\}$ and similarly
$\thetamax = \inf\{\theta \in \Theta\}$ in the complementary case.
(We take extrema over points for which the $\loss_z'(\theta)$ exists;
Appendix~\ref{sec:achievable-minimizers}
shows the differentiability is immaterial.)
It is possible that $\thetamax \le \thetamin$, but this trivializes
the problem, as any $\theta\opt \in [\thetamax, \thetamin]$
minimizes $\poploss_P$
simultaneously for all distributions $P$; see
Lemma~\ref{lemma:achievable-minimizers} in
Appendix~\ref{sec:achievable-minimizers}.
Thus, define
\begin{equation}
  \label{eqn:one-dim-achievable-minimizers}
  \achievable \defeq [\thetamin, \thetamax],
\end{equation}
passing tacitly to the extended reals when $\Theta$ is unbounded to
allow sets of the form $[a, \infty]$ or $[-\infty, b]$
(cf.~\cite[Appendix A.2]{HiriartUrrutyLe93}).
Intuitively, because $\poploss_P'(\theta) \ge 0$ for $\theta >
\thetamax$ and $\poploss_P'(\theta) \le 0$ for $\theta < \thetamin$
for \emph{all} $P$, restricting analysis to the set
$\achievable$ changes nothing.

\subsubsection{Achievable minimizers in the multi-dimensional case}

When $\Theta \subset \R^d$ for general $d \ge 1$, we require a technical
extension to the condition~\eqref{eqn:extreme-minimizers}:
\begin{definition}
  \label{definition:directable}
  A
  set $C \subset \Theta$ is \emph{directable} if for each $\epsilon > 0$,
  there is a compact convex $C_\epsilon$
  satisfying $C \subset C_\epsilon \subset \Theta$
  and a collection of points $\{z_i\}_{i = 1}^k \subset \mc{Z}$
  such that
  \begin{enumerate}[(i)]
  \item $\sup_{\theta \in C_\epsilon} \dist(\theta, C) \le \epsilon$, and
  \item there exists $\delta > 0$ for which
    \begin{equation*}
      \bigcup_{Q \in \mc{P}(\{z_i\}_{i=1}^k)}
      \bigcup_{\theta \in C_\epsilon} \partial
      \left(\poploss_Q + \convexindic{\Theta}\right)(\theta)
      \supset \delta \ball_2.
    \end{equation*}
  \end{enumerate}
\end{definition}
\noindent
Implicit in the Definition~\ref{definition:directable}
is that $\partial (\poploss_Q + \convexindic{\Theta})$ is non-empty.
In the one-dimensional case, for each $\theta \in \interior \Theta$, the
point $\{\theta\}$ is directable whenever there are $z^+, z^-$ such that
$\loss_{z^+}'(\theta) > 0$ and $\loss_{z^-}'(\theta) < 0$, so that this
generalizes the sets defining the extreme plausible
values~\eqref{eqn:extreme-minimizers}.
Intuitively, a directable set means that in an arbitrarily small
neighborhood of the set, we can ``place'' a subgradient in any direction
$u \in \sphere^{d-1}$: there exist $\theta$ and $Q$ for which
$\delta u \in \partial (\poploss_Q + \convexindic{\Theta})(\theta)$.

Before defining the analogue of the achievable
set~\eqref{eqn:one-dim-achievable-minimizers}, we give a few
sufficient characterizations for directability to show that such sets are
a typical phenomenon.
The first
condition is that
any time a loss attains its minimizers on a compact set,
the set of minimizers is directable.
See Figure~\ref{fig:directable-set} for an illustration.
\begin{lemma}
  \label{lemma:compact-minimizers-directable}
  Let $Q$ be finitely supported and assume
  $\theta\opt(Q) \defeq \argmin_{\theta \in \Theta} \poploss_Q(\theta)$ is
  compact.
  Then
  $\theta\opt(Q)$ is directable.
\end{lemma}
\noindent
This lemma is an immediate
consequence of Proposition~\ref{proposition:strong-solution-placing}
to come in Section~\ref{sec:solution-placing-conjugates}.

\begin{figure}
  \begin{center}
    \begin{tabular}{cc}
      \includegraphics[width=.6\columnwidth]{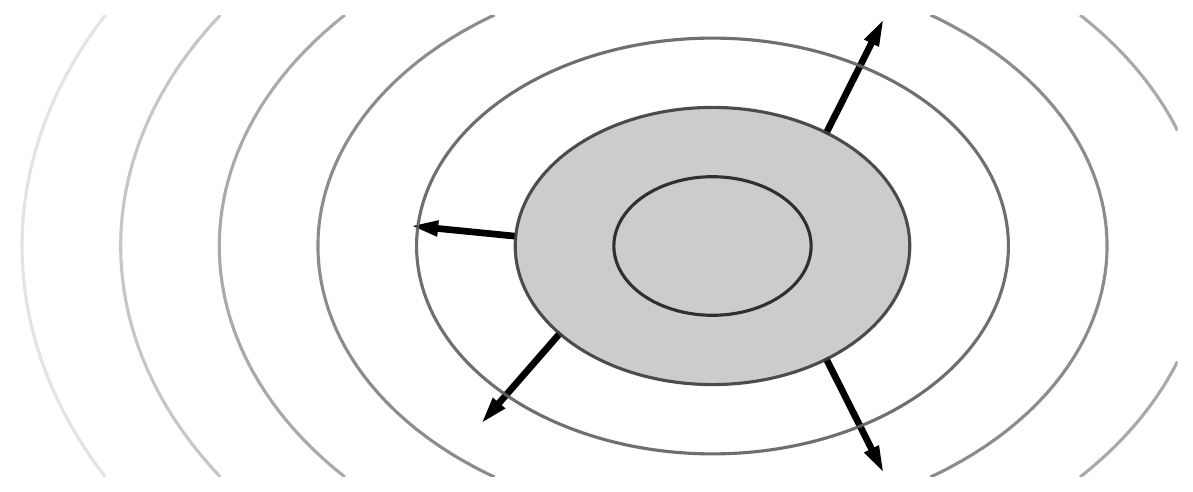} &
      \hspace{-.7cm}
      \begin{minipage}{.42\columnwidth}
        \vspace{-4cm}
        \caption{\label{fig:directable-set}
          A directable set constructed as the sublevel set
          of the function
          $\loss(\theta) = \half \theta_1^2 + \theta_2^2$.
          The gradients point out of the sublevel set $\{\theta
          \mid \half \theta_1^2 + \theta_2^2 \le c\}$ and have norm
          $\ltwo{\nabla \loss(\theta)}$ bounded away from 0.
        }
      \end{minipage}
    \end{tabular}
  \end{center}
\end{figure}

We can also more directly analogize the
condition~\eqref{eqn:extreme-minimizers}.
Recall the \emph{directional
derivative}~\cite[Def.~VI.1.1.1]{HiriartUrrutyLe93} of a convex function $f$
at the point $\theta \in \dom f$ in the direction $v$,
\begin{equation*}
  \deriv f(\theta; v) \defeq
  \lim_{t \downarrow 0} \frac{f(\theta + tv) - f(\theta)}{t}
  = \inf_{t > 0} \frac{f(\theta + tv) - f(\theta)}{t}
  \stackrel{(\star)}{=} \sup_{g \in \partial f(\theta)}
  \<g, v\>
\end{equation*}
where the limit always exists,
$\deriv f(\theta; v)$ is convex and positively homogeneous in $v$, and the
equality~$(\star)$ holds if $\theta \in \interior \dom f$.
If $f$ is differentiable at $\theta$, then $\deriv f(\theta; v)
= \<\nabla f(\theta), v\>$.
To make the analogy with condition~\eqref{eqn:extreme-minimizers}
precise, we make the following definition.
\begin{definition}
  \label{definition:achievable-point}
  A point $\theta$ is \emph{achievable} if
  for each $v \in \sphere^{d-1}$,
  \begin{equation}
    \label{eqn:positive-derivative}
    \mbox{there~exists~} z \in \mc{Z}
    ~~ \mbox{with} ~~
    \deriv \loss_z(\theta; v) > 0.
  \end{equation}
\end{definition}
%
\noindent
Achievable points are also directable
(see Appendix~\ref{sec:proof-positive-derivatives-directable} for proof):
%
\begin{lemma}
  \label{lemma:positive-derivatives-directable}
  If $\theta \in \interior \Theta$ is
  achievable~\eqref{eqn:positive-derivative}, then $\{\theta\}$ is
  directable (Definition~\ref{definition:directable}).
  In the Definition~\ref{definition:directable},
  we may take the set $C_\epsilon = \{\theta\}$ for all $\epsilon > 0$.
\end{lemma}
%

Typical losses in statistics and machine learning satisfy
these achievability conditions; an example may help.

\begin{example}[Generalized linear models (GLMs)]
  \label{example:glms}
  Consider a GLM for predicting $y \in \mc{Y} \subset
  \R$ from $x \in \mc{X} \subset \R^d$ with parametric model $p_\theta(y
  \mid x) = \exp(y \<x, \theta\> - A(\theta \mid x)) h(y)$, where $h$ is a
  carrier and $A$ is the log-partition function~\cite{Brown86,
    WainwrightJo08}.
  Then because
  \begin{equation*}
    \nabla A(\theta \mid x) = \E_\theta[Y \mid X = x] \cdot x
  \end{equation*}
  (where $\E_\theta$ denotes expectation under the model $p_\theta$), the
  log loss $\loss_{x,y}(\theta) = -y\<x, \theta\> + A(\theta \mid x)$ is
  $\mc{C}^\infty$ and satisfies $\nabla \loss_{x,y}(\theta) = x(\E_\theta[Y
    \mid x] - y)$.
  As $\E_\theta[Y \mid X = x] \in \interior \conv(\mc{Y})$, so long as
  $\mc{X}$ contains \emph{any} linearly independent set of $d$ points,
  inequality~\eqref{eqn:positive-derivative} holds for all $\theta \in
  \R^d$.
\end{example}

\newcommand{\unconstrained}{\mc{U}}

We move to the definition of the achievable set in the
general $d$-dimensional case, which we perform by considering
the outer construction of closed convex sets as intersections of
half-spaces.
To that end, for $v \in \R^d$ and $t \in \R$, define the
half space $H_{v,t} \defeq \{\theta \mid \<v, \theta\> \le t \ltwo{v}\}$,
and call the pair $(v, t) \in \R^d \times \R$
\emph{unconstraining} if
\begin{equation*}
  H_{v,t} \cap C \neq \emptyset
  ~~ \mbox{for~all~directable}~ C \subset \Theta.
\end{equation*}
That is, $H_{v,t}$ has non-trivial intersection with all
directable sets $C$.
Call the collection of such pairs $\unconstrained$.
When $\Theta$ is compact,
any unconstraining half space $H_{v,t}$ satisfies
\begin{equation}
  \label{eqn:halfspace-unconstraining}
  \poploss_Q\opt(H_{v, t} \cap \Theta)
  = \poploss_Q\opt(\Theta)
\end{equation}
for all $Q \in \Pdiscrete$, because $\argmin_{\theta \in \Theta}
\poploss_Q(\theta)$ is a non-empty compact convex set and
Lemma~\ref{lemma:compact-minimizers-directable} applies.
In fact, equality~\eqref{eqn:halfspace-unconstraining} extends to any
(Borel)
probability measure $P$ for which the loss $\loss_z$ is integrable; see
Lemma~\ref{lemma:discrete-or-borel} in
Appendix~\ref{sec:multi-dim-achievable}.
More generally, whenever $Q$ is such that $\argmin_{\theta \in \Theta}
\poploss_Q(\theta)$ exists and is compact,
the equality~\eqref{eqn:halfspace-unconstraining} holds for
all $(v, t) \in \unconstrained$.

Noting that $v = \zeros$ allows $H_{v,t} = \R^d$,
we may then define the achievable set
\begin{equation}
  \label{eqn:achievable-minimizers}
  \achievable \defeq
  \bigcap_{(v, t) \in \unconstrained} H_{v,t}.
\end{equation}
Clearly, $\achievable \subset \Theta$ as $\Theta$ is closed convex,
and
in one dimension, the half-spaces $H_{1, t} =
\openleft{-\infty}{t}$ and $H_{-1, t} = \openright{t}{\infty}$
show that
definitions~\eqref{eqn:one-dim-achievable-minimizers}
and~\eqref{eqn:achievable-minimizers}
coincide;
see also Lemma~\ref{lemma:achievable-as-all-argmins} in
Section~\ref{sec:achievable-minimizers}.

The next example highlights why, even in statistical or other optimization
scenarios, we may wish to consider settings in which the achievable
points~\eqref{eqn:positive-derivative} are not dense.

\begin{example}[Robust losses for location]
  Consider estimating a location via a robust loss $\loss_z(\theta) =
  \norm{\theta - z}$ for $\Theta = \R^d$, where $z \in \mc{Z} \subset \R^d$
  and $\norm{\cdot}$ is a norm on $\R^d$.
  Even if $\mc{Z}$ is discrete (e.g., taking values with integer
  coordinates), if $\mc{Z}$ covers enough space that $\conv(\mc{Z}) = \R^d$,
  then because $\theta = z$ uniquely minimizes $\loss_z(\theta)$ for each
  $z$, the set of unconstraining
  half-spaces in the definition~\eqref{eqn:achievable-minimizers}
  consists of $\{H_{\zeros,t}\}_{t \in \R}$, so $\achievable = \R^d$.
\end{example}

We therefore make the following standing assumption without further comment
and with essentially no loss of generality except that we could replace
$\Theta$ in each theorem with $\achievable$.
\begin{assumption}
  \label{assumption:achievable}
  The achievable set coincides with
  $\Theta$, that is, $\Theta = \achievable$.
\end{assumption}

\subsection{The minimax bounds: compact domains}
\label{sec:minimax-compact}

With these administrivia out of the way, we
move to delineating when
completely distribution free minimization is possible
in the case that $\Theta$ is a compact convex set.
We will prove somewhat more careful results than bounds
on the general minimax risk~\eqref{eqn:minimax-opt}.
For many of our lower bounds, we need only consider
discrete distributions, so letting the set $\Pdiscrete(\mc{Z})$
be the collection of finitely supported probability distributions
on $\mc{Z}$,
define
\begin{equation*}
  \minimaxlow_n(\loss, \mc{Z}, \Theta)
  \defeq \inf_{\what{\theta}_n}
  \sup_{P \in \Pdiscrete(\mc{Z})}
  \E_{P^n}\left[\poploss_P(\what{\theta}_n(Z_1^n))
    - \poploss_P\opt(\Theta)\right].
\end{equation*}
We also slightly refine the risk~\eqref{eqn:minimax-opt}
to make the results sensible.
Let $\mc{P}_\loss(\mc{Z})$ be the set of Borel probability
measures for which $\poploss_P(\theta) = \E_P[\loss_Z(\theta)]$ is
\emph{well-defined}, meaning that
\begin{equation}
  \label{eqn:well-defined}
  \E_P\left[|\loss_Z(\theta)|\right] < \infty
  ~ \mbox{for}~ \theta \in \interior \Theta
  ~~ \mbox{and} ~~
  \E_P\left[\hinge{-\loss_Z(\theta)}\right] < \infty
  ~~ \mbox{for~} \theta \in \Theta.
\end{equation}
Thus we may have $\poploss_P(\theta) = +\infty$, but $\poploss_P(\theta) >
-\infty$ and $\inf_{\theta \in \Theta} \poploss_P(\theta) < \infty$, and
$\poploss_P$ is convex, proper, and lower
semicontinuous~\cite[Ch.~7.2.4]{ShapiroDeRu14}.
Then we use the refinement
\begin{equation*}
  \minimax_n(\loss, \mc{Z}, \Theta)
  \defeq \inf_{\what{\theta}_n}
  \sup_{P \in \mc{P}_\loss(\mc{Z})}
  \E_{P^n}\left[\poploss_P(\what{\theta}_n(Z_1^n))
    - \poploss_P\opt(\Theta)\right]
\end{equation*}
of the minimax risk~\eqref{eqn:minimax-opt}.
Clearly $\minimax_n \ge \minimaxlow_n$.

We have our first two main theorems.
Theorem~\ref{theorem:minimax-lower} does not require compactness of $\Theta$.
\begin{theorem}
  \label{theorem:minimax-lower}
  Let Assumption~\ref{assumption:achievable} hold.
  If Condition~\ref{cond:compact-case} fails,
  then
  \begin{equation*}
    \inf_n \minimaxlow_n(\loss, \mc{Z}, \Theta)
    > 0.
  \end{equation*}
\end{theorem}
\noindent
By considering estimators $\what{\theta}_n$ obtained by averaging $n$ steps
of the stochastic subgradient method over $\Theta_n$, where $\Theta_n
\subset \Theta$ are somewhat carefully chosen subsets (depending on $\mc{Z}$
and $\loss$) satisfying $\Theta_n \uparrow \Theta$, we can also show a
converse achievability result.
\begin{theorem}
  \label{theorem:minimizable}
  Let $\Theta$ be compact.
  If Condition~\ref{cond:compact-case} holds, then 
  \begin{equation*}
    \lim_n \minimax_n(\loss, \mc{Z}, \Theta) = 0.
  \end{equation*}
\end{theorem}
\noindent
This result provides no rate of convergence, which must
generally depend on the structure of $\loss$:
without extra assumptions on $\loss$, the rate may be arbitrarily
slow, even in one-dimensional cases with finite $\mc{Z}$,
as the next result shows.
\begin{proposition}
  \label{proposition:no-rate}
  Let $\mc{Z}=\{0,1\}$ and $\Theta=[0,1]$.
  Then there exists a numerical constant $c > 0$ such that
  for any continuous increasing rate function $r: \openright{1}{\infty}
  \to \R_+$, there exists a loss $\loss$ satisfying
  Condition~\ref{cond:compact-case} with
  Lipschitz constant $\lipconst([\epsilon, 1 - \epsilon])
  \le r^{-1}(r(1) / \epsilon)$
  such that
  \begin{equation*}
    \minimaxlow_n(\loss, \mc{Z}, \Theta) \ge
    c \, \frac{r(1)}{r(2n)}.
  \end{equation*}
\end{proposition}

We divide the proofs of the results into several parts,
which we defer, providing an
outline in Section~\ref{sec:proof-outline}.
%
We provide relevant background
on convexity in Section~\ref{sec:convex-analysis}.
We separate the arguments for Theorem~\ref{theorem:minimax-lower} in the
one-dimensional case that $\Theta \subset \R$, as the convex analytic
details simplify considerably, though the intuition for the constructions
remains consistent in the higher-dimensional cases; see
Section~\ref{sec:proof-bounded-lower-one-d}.
Section~\ref{sec:proof-compact-lower-general} contains the proof of
Theorem~\ref{theorem:minimax-lower} in the $d$-dimensional case.
We provide the proof of achievability (Theorem~\ref{theorem:minimizable})
in Section~\ref{sec:proof-minimizable},
showing its essential sharpness (Proposition~\ref{proposition:no-rate})
in Section~\ref{sec:proof-no-rate}.

Before continuing, however, we present two examples to help
delineate Theorems~\ref{theorem:minimax-lower} and~\ref{theorem:minimizable}.

\begin{example}[Log loss minimization]
  \label{example:log-loss}
  For $z \in \mc{Z} = \{0, 1\}$ and
  $\theta \in \Theta \defeq [0, 1]$,
  consider
  \begin{equation*}
    \loss_z(\theta) = z \log \frac{1}{\theta}
    + (1 - z) \log \frac{1}{1 - \theta}.
  \end{equation*}
  The loss $\loss_z(\theta)$ is Lipschitz in $\theta$
  over any compact subset interior to $[0, 1]$, though it
  fails to be Lipschitz at the boundaries.
  Nonetheless, $\loss$ satisfies Condition~\ref{cond:compact-case}, so
  Theorem~\ref{theorem:minimizable} shows that minimization of
  $\E[\loss_Z(\theta)]$ over $\theta \in [0, 1]$ is possible.

  The log loss for estimating multinomials similarly satisfies
  Condition~\ref{cond:compact-case}:
  for $z \in \mc{Z} = \{1, \ldots, k\}$ and
  $\Theta = \{\theta \in \R^{k-1}_+ \mid \<\ones, \theta\> \le 1\}$,
  the losses
  \begin{equation*}
    \loss_z(\theta) =
    \sum_{j = 1}^{k - 1} \indic{z = j}
    \log \frac{1}{\theta_j}
    + \indic{z = k} \log \frac{1}{1 - \<\ones, \theta\>}
  \end{equation*}
  are Lipschitz on any convex compact subsets interior to $\Theta$.
\end{example}

\begin{example}[Squared loss minimization]
  Consider the parameter space $\Theta = [0, 1]$ and sample space $\mc{Z} =
  \R$, and take the loss $\loss_z(\theta) = \half (\theta - z)^2$.
  Then $\sup_z |\loss_z'(\theta)| = \sup_z |\theta - z| = \infty$ for each
  $\theta$, and Theorem~\ref{theorem:minimax-lower} shows that distribution
  free minimization is impossible.
\end{example}

We also highlight that Condition~\ref{cond:compact-case} is related but 
not equivalent to standard uniform convergence.
However, it does admit an equivalent characterization in terms of uniform
convergence of anchored losses
$\bar\loss_z(\theta) \defeq \loss_z(\theta)-\inf_{\theta'\in \Theta}\loss_z(\theta')$. We provide a proof in Section~\ref{subsection:proof-uc-vs-c1}.
\begin{proposition}
  \label{proposition:uc-vs-c1}
  Let $\Theta$ be compact with nonempty interior, and consider the conditions
  \begin{equation}
    \label{eqn:uc-l}
    \text{For all compact subsets } \Theta_0  \subset \interior \Theta
    ~~~~
    \sup_{P\in \mc{P}_\loss(\mc{Z})}
    \E_P\left[\sup_{\theta\in \Theta_0}|\poploss_{P_n}(\theta)-\poploss_P(\theta)|\right]\to 0,
    \tag{UC}
  \end{equation}
  and
  \begin{equation}
    \label{eqn:uc-al}
    \text{For all compact subsets } \Theta_0  \subset \interior \Theta
    ~~~~
    \sup_{P\in \mc{P}_\loss(\mc{Z})}
    \E_P\left[\sup_{\theta\in \Theta_0}|(P_n-P)\bar{\loss}_z(\theta)|\right]\to 0.
    \tag{UC'}
  \end{equation}
  Condition~\ref{cond:compact-case} is equivalent to~\eqref{eqn:uc-al} 
  but not~\eqref{eqn:uc-l}.
\end{proposition}

\subsection{The minimax bounds: unbounded domains}
\label{sec:unbounded-case}

Now we consider the case that $\Theta$ is unbounded, and we continue to make
the standing Assumption~\ref{assumption:achievable} that $\Theta$ consists
of achievable minimizers~\eqref{eqn:achievable-minimizers}.
The limiting minimax behavior in this case
depends essentially on whether simultaneously for all finitely
supported distributions $Q$,
there exist compact sets attaining near minimizers.
In particular, recalling $\poploss_Q\opt(\Theta) = \inf_{\theta \in \Theta}
\poploss_Q(\theta)$, the relevant conditions become the following.

\conditionbox{\label{cond:unbounded-case}
  In addition to Condition~\ref{cond:compact-case},
  for all $\epsilon > 0$, there
  exists a compact $\Theta_0 \subset \Theta$ such that
  \begin{equation}
    \label{eqn:cannot-be-far-in-compact}
    \inf_{\theta \in \Theta_0} \left[\poploss_Q(\theta)
      - \poploss_Q\opt(\Theta)\right] \le \epsilon
  \end{equation}
  for all $Q \in \Pdiscrete(\mc{Z})$.
}
\noindent
In the one-dimensional case,
a simpler criterion is sufficient:
that for all $\epsilon > 0$, there exists
a compact $\Theta_0 \subset \Theta$ such that
\begin{equation}
  \label{eqn:cannot-be-far-in-interval}
  \inf_{\theta \in \Theta_0} \left[\loss_z(\theta)
    - \loss_z\opt(\Theta)\right] \le \epsilon
\end{equation}
for all $z \in \mc{Z}$.
(See Lemma~\ref{lemma:conditions-same-one-dim} in
Section~\ref{sec:proof-unbounded-lower}.)
While this simplification does not appear to extend to $d$-dimensional
cases, we have the following characterization.

\begin{theorem}
  \label{theorem:unbounded-case}
  Let Assumption~\ref{assumption:achievable} hold.
  If Condition~\ref{cond:unbounded-case} fails, then
  \begin{equation*}
    \inf_n \minimaxlow_n(\loss, \mc{Z}, \Theta) > 0.
  \end{equation*}
\end{theorem}

\noindent
We have the complementary result as well.
\begin{proposition}
  \label{proposition:unbounded-upper}
  If Condition~\ref{cond:unbounded-case} holds,
  then
  \begin{equation*}
    \lim_{n \to \infty} \minimax_n(\loss, \mc{Z}, \Theta) = 0.
  \end{equation*}
\end{proposition}

We provide the proof of the one-dimensional
case of Theorem~\ref{theorem:unbounded-case}
in
Section~\ref{sec:proof-unbounded-lower},
demonstrating the extension to the higher dimensional case
in Section~\ref{sec:proof-general-unbounded-lower}.
Given Theorem~\ref{theorem:minimizable},
Proposition~\ref{proposition:unbounded-upper} admits a relatively
simple proof once we demonstrate that Condition~\ref{cond:unbounded-case}
immediately extends from the collection $\Pdiscrete$ to
the well-defined distributions $\mc{P}_\loss$ defined by
the condition~\eqref{eqn:well-defined}, which
we provide in Section~\ref{sec:proof-unbounded-upper}.
Before giving the proofs, it is instructive to consider two examples,
highlighting the ways that finite Lipschitz constants are neither necessary
nor sufficient for minimization to be possible.

\begin{example}
  \label{example:exp-loss}
  Consider the sample space $\mc{Z} = \{-1, 1\}$ and exponential loss
  $\loss_z(\theta) = e^{z \theta}$.
  While $\loss$ is $\mc{C}^\infty$, none of its derivatives are
  Lipschitz on $\R$.
  Nonetheless, it is clear that for any $0 < \epsilon \le 1$,
  we have
  \begin{equation*}
    \inf_{|\theta|
      \le \log \frac{1}{\epsilon}} \loss_z(t)
    = \epsilon = \loss_z\opt + \epsilon.
  \end{equation*}
  Proposition~\ref{proposition:unbounded-upper}
  shows that $\minimax_n(\loss, \mc{Z}, \R) \to 0$ as
  $n \to \infty$.
\end{example}

On the other hand, the absolute loss is globally Lipschitz,
and yet the associated minimax risk over $\R$ is infinite.
In one dimension, $\deriv_+ f(\theta)$ and $\deriv_- f(\theta)$ denote the
right and left derivatives of a convex function $f$, respectively; 
see~\eqref{eqn:right-derivative} and~\eqref{eqn:left-derivative} for the 
formal definitions.

\begin{example}[Median and quantile estimation]
  \label{example:quantile-estimation}
  Consider the sample space $\mc{Z} = \R$ and, for a fixed $\alpha \in (0,
  1)$, the $\alpha$-quantile losses
  \begin{equation*}
    \loss_z(\theta)
    \defeq \alpha \hinge{\theta - z} + (1 - \alpha)
    \hinge{z - \theta}
    - \left(\alpha \hinge{-z} + (1 - \alpha) \hinge{z}\right),
  \end{equation*}
  where the subtraction guarantees
  that $\poploss_P(\theta)$ is well-defined~\eqref{eqn:well-defined}
  for all distributions
  $P$ on $\R$.
  %
  %
  Calculating left and right derivatives,
  \begin{align*}
    \deriv_+ \poploss_P(\theta)
    & = \alpha P(Z \le \theta)
    - (1 - \alpha) P(Z > \theta)
    = P(Z \le \theta) - (1 - \alpha) \\
    \deriv_- \poploss_P(\theta)
    & = \alpha P(Z < \theta)
    - (1 - \alpha) P(Z \ge \theta)
    = P(Z < \theta) - (1 - \alpha),
  \end{align*}
  so that if $\theta\opt$ is a $(1 - \alpha)$-quantile of $Z$,
  i.e., satisfies $P(Z < \theta\opt) \le 1 - \alpha
  \le P(Z \le \theta\opt)$,
  we obtain $\deriv_-\poploss_P(\theta\opt) \le 0
  \le \deriv_+ \poploss_P(\theta\opt)$.
  Nonetheless,
  Condition~\ref{cond:unbounded-case} fails:
  consider a compact interval $[t_0, t_1]$.
  As $\loss_z(\theta)$ attains its minimum at $\theta = z$,
  for $z \ge t_1$ we obtain
  \begin{equation*}
    \inf_{\theta \le t_1} \loss_z(\theta)
    - \loss_z\opt
    = (1 - \alpha) \hinge{z - t_1},
  \end{equation*}
  which clearly tends to $+\infty$ as $z \to \infty$,
  contradicting inequality~\eqref{eqn:cannot-be-far-in-interval}.

  In this case, a more direct argument
  yields that $\minimaxlow_n = +\infty$, which we include
  only for its simplicity.
  Let $z_0 < z_1$ be arbitrary and to be chosen,
  and for $0 < \delta \le \min\{\alpha, 1-\alpha\}$ (also to be chosen) 
  define the two-point distributions
  \begin{equation*}
    P_0 = (1 - \alpha + \delta) \pointmass_{z_0}
    + (\alpha - \delta) \pointmass_{z_1}
    ~~ \mbox{and} ~~
    P_1 = (1 - \alpha - \delta) \pointmass_{z_0}
    + (\alpha + \delta) \pointmass_{z_1}.
  \end{equation*}
  Then $P_0$ has unique $(1 - \alpha)$-quantile
  $\theta\opt_0 = z_0$, while
  $P_1$ has unique $(1 - \alpha)$-quantile $\theta\opt_1 = z_1$.
  Recalling the bound~\eqref{eqn:minimax-to-testing},
  a quick calculation yields
  $\dopt(\poploss_{P_0}, \poploss_{P_1})
  \ge \frac{\delta}{2} |z_1 - z_0|$,
  so
  \begin{equation*}
    \minimaxlow_n(\loss, \R, \R)
    \ge \frac{\delta}{4} |z_1 - z_0|
    \left(1 - \tvnorm{P_0^n - P_1^n}\right).
  \end{equation*}
  Set $\delta = \min\{\alpha,1-\alpha,\frac{1}{4n}\}$, so
  $\tvnorm{P_0^n - P_1^n} \le n \tvnorm{P_0 - P_1}
  \le 2n \delta \le \half$.
  Then taking $|z_1 - z_0| \to \infty$ gives $\minimaxlow_n = +\infty$.
\end{example}

\noindent
In Section~\ref{sec:stationary-points}, we revisit this example in the
context of finding stationary points.

\subsubsection{Applicability of Conditions~\ref{cond:compact-case} and~\ref{cond:unbounded-case}}

Conditions~\ref{cond:compact-case} and~\ref{cond:unbounded-case}
may be hard to check in full generality, but they hold for several
natural estimation problems, even when global Lipschitzness fails.

\begin{example}
  Consider the task where we have a random variable
  $Y\in \mc{Y}=\{1,\cdots,k\}$ and we want to predict the distribution
  of $Y$ from data.
  Letting $\Delta_{k}$ be the $(k-1)$-dimensional simplex, the
  standard approach to this problem is to derive an estimate
  $\hat{p}\in \Delta_{k}$ by minimizing a regular strictly proper loss
  $\tilde{\loss}(p,y)$ \cite{GneitingRa07}.
  We often generate $p$ via a prediction map or (inverse link)
  $q:\Theta \to \Delta_k$ from an underlying parameter space $\Theta$,
  which we take to be closed convex with non-empty interior.
  We now make the following assumptions:
  \begin{enumerate}[1.]
  \item The prediction map $q$ is continuous, satisfies
    $q(\interior \Theta)=\relint\Delta_k$, and for every compact
    subset $K\subset \relint \Delta_k$ the preimage $q^{-1}(K)$
    is a compact subset of $\interior \Theta$.
  \item The regular strictly proper losses $\tilde{\loss}(p,y)$ are
    convex in $p$ for each $y\in \mc{Y}$.
  \item\label{item:composition-convex} The composed losses
    \begin{equation*}
      \loss_y(\theta)=\tilde{\loss}(q(\theta),y),
    \end{equation*}
    are convex in $\theta$ for each $y\in\mc{Y}$.
  \end{enumerate}
  \noindent
  Typical choices include the log loss $\tilde{\loss}(p, y)
  = -\log p_y$, with link
  $q(\theta) = [e^{\theta_i} / \sum_j e^{\theta_j}]_{i=1}^k$
  where $\theta\in\mathbb{R}^{k-1}$ and we set $\theta_k:=0$, 
  which clearly need not be bounded over $p \in \Delta_k$.
  Nonetheless, such losses satisfy Conditions~\ref{cond:compact-case}
  and~\ref{cond:unbounded-case}.
  (See
  Appendix~\ref{subsec:proof-proper-losses} for a proof of this result).

  \begin{claim}
    \label{claim:proper-losses}
    The loss class defined above satisfies
    Conditions~\ref{cond:compact-case} and~\ref{cond:unbounded-case}.
  \end{claim}

  The final assumption~\ref{item:composition-convex} is typically a
  consequence of the construction of the natural prediction map $q(\theta)$.
  Indeed, standard results on proper losses~\cite{GneitingRa07}
  show that $\tilde{\loss}$ admits a Gneiting/Raftery/Savage representation,
  in that there exists a closed strictly convex negative entropy
  $\Omega:\Delta_k\to \R$ such that, for $p\in \relint\Delta_k$
  \begin{equation*}
    \tilde{\loss}(p,y)=-\Omega(p)-\langle \nabla \Omega(p), e_y-p\rangle,
  \end{equation*}
  where $\{e_i\}_{i=1}^{k}$ are the canonical basis vectors in $\R^k$
  and $\nabla\Omega(p)\in \partial \Omega(p)$.
  If the negative entropy is of Legendre type~\cite[Ch.~14]{Duchi26}
  on $\relint \Delta_k$, we may define the
  conjugate $\Omega^*$ and conjugate link $q:\R^k\to \Delta_k$ by
  \begin{equation*}
    \Omega^*(\theta)
    =
    \sup_{p\in\Delta_k}
    \{\langle \theta,p\rangle-\Omega(p)\}
    ~~ \mbox{and} ~~
    q(\theta) = \nabla \Omega^*(\theta)
    = \argmax_{p\in \Delta_k}\{\langle \theta, p\rangle -\Omega(p)\}
  \end{equation*}
  %
  \citet[Thm.~14.3.5]{Duchi26} shows that then
  \begin{equation*}
    \tilde{\loss}(q(\theta),y)
    =
    \Omega^*(\theta)-\langle \theta,e_y\rangle
    =
    \Omega^*(\theta)-\theta_y .
  \end{equation*}
  Therefore $\theta\mapsto \tilde{\loss}(q(\theta),y)$ is convex,
  since $\Omega^*$ is convex.
  The final convexity assumption~\ref{item:composition-convex}
  then follows from the
  structural assumption that $q$ is the conjugate
  link.
\end{example}

Our results also apply under reparametrizations.
Consider $h:\Xi\to\Theta$ with $h(\Xi)=\Theta$, and assume that
$h$ maps $\interior\Xi$ homeomorphically onto $\interior\Theta$
and that $\xi\mapsto\loss_z(h(\xi))$ remains convex.
When both formulations satisfy our standing assumptions and every
point in both domains is achievable in the sense of
Definition~\ref{definition:achievable-point} for its respective
loss family, we can also show that $h$ must map their interiors
homeomorphically.
We omit the separate topological argument, since our main results
do not rely on this observation.

Condition~\ref{cond:compact-case} holds for one formulation if and
only if it holds for the other.
To see this, suppose it holds for the original losses. Fix a compact
set $\Xi_0\subset\interior\Xi$ and choose $r>0$ small enough that
$\Xi_r\defeq \Xi_0+rB_2\subset\interior\Xi$.
Since $h$ maps the interiors homeomorphically, the set
$\Theta_r\defeq h(\Xi_r)$ is compact in $\interior\Theta$.
Writing $D_r:=\operatorname{diam}(\Theta_r)$,
Condition~\ref{cond:compact-case} gives, uniformly in $z$,
\begin{equation*}
  \sup_{\xi,\xi'\in\Xi_r}
  \left\vert\loss_z(h(\xi))-\loss_z(h(\xi'))\right\vert
  \le M(\Theta_r)D_r.
\end{equation*}
Applying the subgradient inequality from $\xi\in\Xi_0$ to
$\xi+rv\in\Xi_r$, for unit vectors $v$, bounds the norm of 
every subgradient of $\loss_z\circ h$ at $\xi$ by $M(\Theta_r)D_r/r$.
Thus the reparametrized losses are $M(\Theta_r)D_r/r$-Lipschitz
on $\Xi_0$, uniformly in $z$, establishing
Condition~\ref{cond:compact-case} for the reparametrized problem.
Applying the same argument to the interior inverse gives the
converse, although the Lipschitz constants may change.

We can also transfer Condition~\ref{cond:unbounded-case} when its
compact witness sets lie in the interior.
If sets $\Theta_\epsilon\subset\interior\Theta$ witness 
Condition~\ref{cond:unbounded-case}, the interior inverse of 
$h$ maps them to compact sets $\Xi_\epsilon\subset\interior\Xi$.
These sets satisfy the same excess loss bounds, since composition
preserves the loss values and $h(\Xi)=\Theta$ preserves the infima
over the full domains.
Together with the equivalence of Condition~\ref{cond:compact-case},
this gives Condition~\ref{cond:unbounded-case} for the reparametrized
problem.
Conversely, $h$ maps compact interior witnesses for the reparametrized
problem to compact interior witnesses for the original problem.

\subsection{Proof outlines}
\label{sec:proof-outline}

Each of our lower bounds follows a standard recipe of reducing optimization
to a problem of testing between two distributions~\cite{AgarwalBaRaWa12,
  Duchi18, Wainwright19}, but because the losses have no particular form
except that they cannot be uniformly Lipschitz, we must carefully
construct the distributions $P$ so that optimizing well indeed implies
testing between two statistically indistinguishable distributions.
For two convex functions, define the
optimization ``distance''
\begin{equation}
  \label{eqn:opt-distance}
  \dopt(f_0, f_1) \defeq \sup\left\{\delta \ge 0 \mid
  \begin{array}{ll} f_0(\theta) < f_0\opt + \delta
    ~~ \mbox{implies} ~~
    f_1(\theta) \ge f_1\opt + \delta & \mbox{and} \\
    f_1(\theta) < f_1\opt + \delta
    ~~ \mbox{implies} ~~
    f_0(\theta) \ge f_0\opt + \delta
  \end{array}
  \right\}.
\end{equation}
Then~\cite[Eq.~(5.2.3)]{Duchi18}
for any two (finitely supported) distributions $P_0$ and $P_1$ on $\mc{Z}$,
\begin{equation}
  \label{eqn:minimax-to-testing}
  \minimaxlow_n(\loss, \mc{Z}, \Theta)
  \ge \half \dopt(\poploss_{P_0}, \poploss_{P_1})
  \left(1 - \tvnorm{P_0^n - P_1^n}\right)
\end{equation}
Each of our lower bound proofs thus pursues competing goals: to construct
well-separated losses $\poploss_{P_0}$ and $\poploss_{P_1}$ while the
distributions $P_0, P_1$ generating them are too close to distinguish.
We give an overview of the constructions here, leaving
details to the proofs in Section~\ref{sec:main-proofs}.

Let us focus for now on the case that Condition~\ref{cond:compact-case}
fails, so that there are points $\theta_0 \in \interior \Theta$ with arbitrarily large
gradient magnitude $\lipconst_{\theta_0} \defeq \sup_{z \in \mc{Z}}
\sup_{g \in \partial \loss_z(\theta_0)} \ltwo{g}$.
Let $z_0 \in \mc{Z}$ and $g \in \partial \loss_{z_0}(\theta_0)$ be one of
these arbitrarily large gradients, defining the normalized version $v = g /
\ltwo{g}$.
Then because of the standing Assumption~\ref{assumption:achievable}, taking
$t = \<v, \theta_0\> + \alpha$ for some constant $\alpha > 0$, the
half-space $\{\theta \mid \<v, \theta\> \ge t\}$ contains a
directable set $C$ (e.g., a finitely supported distribution $Q$ defining a
compact set $\theta\opt(Q)$ of minimizers, as in
Lemma~\ref{lemma:compact-minimizers-directable} and
Fig.~\ref{fig:directable-set}).
Take this point $\theta_0$ and the direction $v$,
and let $Q$ be the finitely supported
distribution so that $\poploss_Q$ satisfies the inclusion in
Definition~\ref{definition:directable};
as Figure~\ref{fig:plant-hyperplane} illustrates,
we have effectively ``placed'' a desired
gradient at some distance from $\theta_0$.

\begin{figure}[t]
  \begin{tabular}{cc}
    \begin{overpic}[width=.7\columnwidth]{
        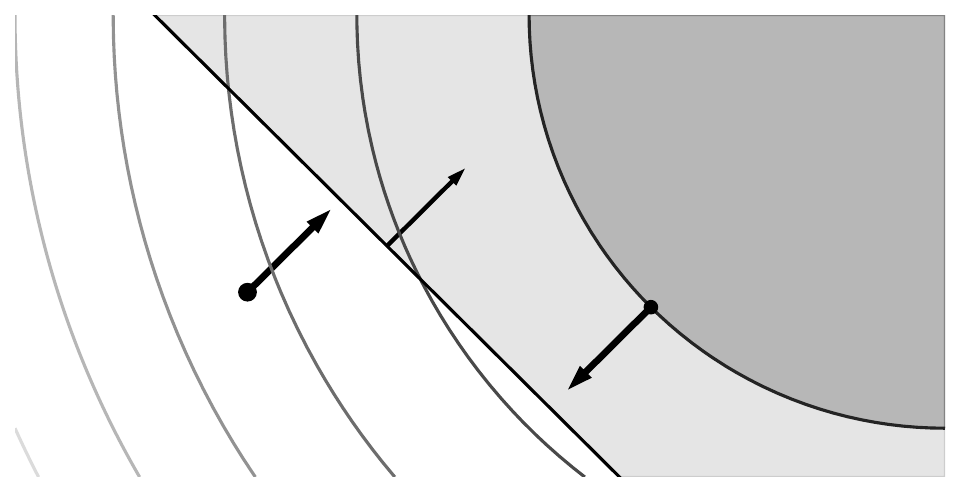}
      \put(22, 17){$\theta_0$}
      \put(30,29){$v$}
      \put(31.5,43){$\{\theta \mid \<v, \theta\> \ge t\}$}
      \put(62,10){$-v$}
      \put(48,30){$v$}
      \put(64,30){$\{\theta \mid \poploss_Q(\theta)
        \le \poploss_Q\opt + \epsilon\}$}
    \end{overpic} &
    \hspace{-1cm}
    \begin{minipage}{.35\columnwidth}
      \vspace{-5cm}
      \caption{\label{fig:plant-hyperplane} Constructing
        the lower bound.
        At the point $\theta_0$, $\loss_z(\theta_0)$
        has arbitrarily large
        gradient $g$ in the direction $v$.
        The sublevel set $\poploss_Q(\theta) \le \poploss_Q\opt
        + \epsilon$, contained in the halfspace $\{\theta \mid \<v, \theta\>
        \ge t\}$ strictly separated from $\theta_0$,
        has a point with gradient $\nabla \poploss_Q(\theta) = -v$.}
    \end{minipage}
  \end{tabular}
\end{figure}

To optimize $\poploss_Q$ to accuracy better than some constant $c$, one
must produce $\theta$ satisfying $\<v, \theta\> \ge t$,
as Figure~\ref{fig:plant-hyperplane} evidences.
On the other hand,
defining the mixture distribution
\begin{equation*}
  P_1 \defeq \left(1 - \frac{1}{n}\right) Q
  + \frac{1}{n} \pointmass_{z_0}
\end{equation*}
that ``hides'' the observation $z_0$ by giving it small probability,
to optimize the corresponding loss $\poploss_{P_1}(\theta) =
(1 - \frac{1}{n}) \poploss_Q(\theta) + \frac{1}{n} \loss_{z_0}(\theta)$ to
accuracy better than a constant, we must have $\<v, \theta\> < t$, as
otherwise, the loss $\loss_{z_0}(\theta)$ must grow arbitrarily.
In particular, we see that
\begin{equation*}
  \dopt(\poploss_Q, \poploss_{P_1}) \ge c
\end{equation*}
for a constant $c > 0$.
The inequality~\eqref{eqn:minimax-to-testing}
then shows that we must control
$\tvnorm{Q^n - P_1^n}$ when
$\tvnorm{Q - P_1} \le \frac{1}{n}$.
The next lemma (whose standard proof we provide for completeness in
Appendix~\ref{sec:proof-tv-and-hellinger}) allows
precise control over such mixtures.
\begin{lemma}
  \label{lemma:tv-and-hellinger}
  Let $P_0$ and $P_1$ be distributions
  satisfying $\tvnorm{P_0 - P_1} \le \gamma \le 1$. Then
  \begin{equation*}
    \tvnorm{P_0^n - P_1^n}
    \le \sqrt{1 - (1 - \gamma)^{2n}}.
  \end{equation*}
  In particular,
  if $P_{0,n}$ and $P_{1,n}$ are sequences of distributions with
  $\tvnorm{P_{0,n} - P_{1,n}} \le a/n$, where $a < \infty$,
  then
  \begin{equation*}
    \limsup_n \tvnorm{P_{0,n}^n - P_{1,n}^n}
    \le \sqrt{1 - e^{-2a}}.
  \end{equation*}
\end{lemma}

Because $\tvnorm{Q^n - P_1^n} \le \sqrt{1 - (1 - 1/n)^{2n}} \to \sqrt{1 -
  e^{-2}}$, inequality~\eqref{eqn:minimax-to-testing} then implies
\begin{equation*}
  \minimaxlow_n(\loss, \mc{Z}, \Theta) \ge
  \frac{c}{2} \left(1 - \sqrt{1 - e^{-2}}\right).
\end{equation*}
Section~\ref{sec:proof-minimax-lower} makes each of these steps rigorous
when Condition~\ref{cond:compact-case} fails.
In the case where $\Theta$ is unbounded and Condition~\ref{cond:compact-case} 
holds, we require a bit more care to generate the separated losses, which 
we obtain not by making an arbitrarily large gradient but instead by placing 
the directable set arbitrarily far away from any given $\theta_0$.
See Section~\ref{sec:proof-unbounded-lower} for the details.


\section{Proofs of the main results}
\label{sec:main-proofs}

We prove our main theorems here.
%

\subsection{Background on convex functions}
\label{sec:convex-analysis}

We begin with a few preliminary observations on convex functions, which will
be useful throughout the proofs of our results and other development; we
defer proofs to appendices.

\subsubsection{One-dimensional first-order behavior}

In the case that $\Theta \subset \R$, convex functions admit substantial
structure that we can
leverage~\cite[e.g.,][Ch.~1]{HiriartUrrutyLe93}, including extensions of
derivatives and subgradients to the extended real line $[-\infty, \infty]$.
Recalling the shorthand $\Rup = \R \cup \{+\infty\}$, let $f : \R
\to \Rup$ be a (proper) closed convex function.
Then we may define the left and right
derivatives~\cite[Thm.~I.4.1.1]{HiriartUrrutyLe93}
\begin{subequations}
  \label{eqn:left-right-derivatives}
  \begin{align}
    \label{eqn:right-derivative}
    \deriv_+ f(\theta) & \defeq \lim_{t \downarrow 0} \frac{f(\theta + t) - f(\theta)}{t}
    = \inf_{t > 0} \frac{f(\theta + t) - f(\theta)}{t}, \\
    \deriv_- f(\theta) & \defeq \lim_{t \downarrow 0} \frac{f(\theta - t) - f(\theta)}{-t}
    = \sup_{t > 0} \frac{f(\theta - t) - f(\theta)}{-t}.
    \label{eqn:left-derivative}
  \end{align}
\end{subequations}
%
The following lemma~\cite[Ch.~I.4]{HiriartUrrutyLe93}
characterizes these quantities everywhere on $\dom f$
and relates $\deriv_{\pm}$ to the
subdifferential of $f$.
\begin{lemma}
  \label{lemma:properties-of-derivatives}
  Let $f : \R \to \Rup$ be closed convex.  Then $f$ is continuous on its
  domain $\dom f = \{\theta \in \R \mid f(\theta) < \infty\}$, and the following
  properties hold.
  \begin{enumerate}[(i)]
  \item \label{item:subdifferential-as-directional}
    If $\theta \in \interior \dom f$,
    $\partial f(\theta) = [\deriv_- f(\theta), \deriv_+ f(\theta)]$
    is a non-empty compact interval.
  \item \label{item:monotonicity-subdifferential} The subdifferential
    $\partial f$ and directional derivatives $\deriv_{\pm}$ are monotone:
    for $\theta_0 < \theta_1$,
    \begin{equation*}
      g_0 \in \partial f(\theta_0) ~ \mbox{and} ~
      g_1 \in \partial f(\theta_1)
      ~~ \mbox{implies} ~~
      g_0 \le g_1
    \end{equation*}
    and
    \begin{equation*}
      \deriv_- f(\theta_0) \le \deriv_+ f(\theta_0)
      \le \deriv_- f(\theta_1)
      \le \deriv_+ f(\theta_1).
    \end{equation*}
  \item \label{item:continuity-directionals} The directional derivatives are
    continuous in that for $\theta_0 \in \interior \dom f$,
    \begin{equation*}
      \lim_{\theta \uparrow \theta_0} \deriv_-f(\theta)
      = \lim_{\theta \uparrow \theta_0} \deriv_+ f(\theta)
      = \deriv_- f(\theta_0)
      ~~ \mbox{and} ~~
      \lim_{\theta \downarrow \theta_0} \deriv_+ f(\theta)
      = \lim_{\theta \downarrow \theta_0} \deriv_- f(\theta) = \deriv_+ f(\theta_0).
    \end{equation*}
  \item 
    If $\theta_0$ is the left endpoint of $\dom f$,
    $\deriv_+f(\theta_0) = \lim_{t \downarrow 0}
    \sup\{g \in \partial f(\theta_0 + t)\}$ and
    the defining equality~\eqref{eqn:right-derivative}
    holds in $\Rdown$.
    (Respectively, the right endpoint
    and \eqref{eqn:left-derivative} in $\Rup$).
  \item On each $[\theta_0, \theta_1] \subset \dom f$,
    $f$ is $\lipconst = \max\{-\deriv_-f(\theta_0),
    \deriv_+ f(\theta_1)\}$-Lipschitz on $[\theta_0, \theta_1]$.
  \end{enumerate}
\end{lemma}

A convex $f$ is differentiable almost everywhere on the interior of its
domain, and if $f$ is defined on an interval, it has
an integral form~\cite[Remark I.4.2.5]{HiriartUrrutyLe93}:
\begin{lemma}
  \label{lemma:integrability-convex}
  If $f$ is closed convex and $[\theta_0, \theta_1] \subset \dom f$, then
  $f(\theta_1) - f(\theta_0) = \int_{\theta_0}^{\theta_1} \deriv_+ f(t) dt =
  \int_{\theta_0}^{\theta_1} \deriv_- f(t) dt$, and for almost all $\theta
  \in (\theta_0, \theta_1)$, $\deriv_+ f(\theta) = f'(\theta) = \deriv_-
  f(\theta)$.
\end{lemma}

We may reformulate Condition~\ref{cond:compact-case}
more precisely in terms of the Lipschitz constants of the
losses $\loss_z$.
Define the signed \emph{local Lipschitz
constants} at each $\theta \in \interior \Theta$ by
\begin{equation}
  \label{eqn:signed-lip-consts}
  \lipconst_\theta^+ \defeq \sup\left\{g \in \partial \loss_z(\theta)
  \mid z \in \mc{Z}\right\}
  ~~ \mbox{and} ~~
  \lipconst_\theta^- \defeq \inf \left\{g \in \partial \loss_z(\theta)
  \mid z \in \mc{Z} \right\}
\end{equation}
The monotonicity properties of
subdifferentials more or less immediately
imply that the Lipschitz constants
$\lipconst_\theta^{\pm}$ are monotone (see
Appendix~\ref{sec:proof-lipschitz-consts-monotone}):
\begin{lemma}
  \label{lemma:lipschitz-consts-monotone}
  The function $\lipconst_\theta^-$ is non-decreasing in $\theta$,
  and $\lipconst_\theta^+$ is non-decreasing in $\theta$.
\end{lemma}
\noindent
In view of Lemma~\ref{lemma:lipschitz-consts-monotone},
the signed Lipschitz constants~\eqref{eqn:signed-lip-consts}
define the local Lipschitz constant
\begin{equation*}
  \lipconst_\theta = \max\{\lipconst_\theta^+, -\lipconst_\theta^-\}
  = \max\{|\lipconst_\theta^+|, |\lipconst_\theta^-|\},
\end{equation*}
and Condition~\ref{cond:compact-case} holds if and only if
$\lipconst_\theta < \infty$ on $\interior \Theta$.

\subsubsection{Multi-dimensional first-order behavior}
\label{sec:first-order-convex}

For $\Theta \subset \R^d$ with $d > 1$, convex functions admit similar
monotonicity and continuity properties of the subdifferential to the
one-dimensional case.
%
%
Let $f : \R^d \to \Rup$ be closed convex.
If $\theta \in \interior \dom f$, then the subdifferential $\partial
f(\theta)$ is a non-empty compact convex
set~\cite[Ch.~VI.1]{HiriartUrrutyLe93}.
Recalling the directional derivative $\deriv f(\theta; v) = \lim_{t
  \downarrow 0} \frac{f(\theta + tv) - f(\theta)}{t}$,
we define
\begin{equation*}
  \deriv_- f(\theta; v)
  \defeq \lim_{t \downarrow 0} \frac{f(\theta - t v) - f(\theta)}{-t}
  \stackrel{(\star)}{=} \inf_{g \in \partial f(\theta)} \<g, v\>,
\end{equation*}
where equality~$(\star)$ holds if $\theta \in
\interior \dom f$~\cite[Ch.~VI.2.3]{HiriartUrrutyLe93}.
We can analogize Lemmas~\ref{lemma:properties-of-derivatives}
and~\ref{lemma:integrability-convex} in the multi-dimensional case as
well~\cite[cf.][Ch.~VI]{HiriartUrrutyLe93}:
\begin{lemma}
  \label{lemma:subdifferential-properties}
  Let $f : \R^d \to \Rup$ be closed convex and proper.
  Then it is continuous on the interior of its domain
  $\dom f = \{\theta \in \R^d \mid f(\theta) < \infty\}$,
  and the following hold.
  \begin{enumerate}[(i)]
  \item On each compact set $\Theta_0 \subset \interior \dom f$,
    $f$ is Lipschitz
    with constant
    \begin{equation*}
      \lipconst(\Theta_0)
      \defeq \sup_{\theta \in \Theta_0} \sup_{g \in \partial f(\theta)}
      \ltwo{g}.
    \end{equation*}
  \item The subdifferentials are monotone in that whenever
    $\partial f(\theta_0)$ and $\partial f(\theta_1)$ are non-empty,
    \begin{equation*}
      g_i \in \partial f(\theta_i)
      ~~ \mbox{implies} ~~
      \<g_1 - g_0, \theta_1 - \theta_0\> \ge 0.
    \end{equation*}
  \item \label{item:outer-semictontinuity-subdiff}
    The subdifferential mapping is outer semicontinuous:
    if $\theta_0 \in \interior \dom f$, then for
    each $\epsilon > 0$ there exists $\delta > 0$ such that
    \begin{equation*}
      \ltwo{\theta - \theta_0} < \delta
      ~~ \mbox{implies} ~~
      \partial f(\theta) \subset \partial f(\theta_0) + \epsilon \ball_2.
    \end{equation*}
  \end{enumerate}
\end{lemma}

In view of Lemma~\ref{lemma:subdifferential-properties},
we can thus generally define the local Lipschitz
constants
\begin{equation}
  \label{eqn:main-lip-const}
  \lipconst_\theta \defeq
  \sup\left\{\ltwo{g} \mid z \in \mc{Z},
  g \in \partial \loss_z(\theta) \right\},
\end{equation}
which coincide with the one-dimensional case~\eqref{eqn:signed-lip-consts}.

\subsubsection{Minimizers, conjugacy, and coercivity}
\label{sec:solution-placing-conjugates}

Let $f : \R^d \to \Rup$ be closed convex.
Recall that $f$ is coercive if $\lim_{\norm{\theta} \to \infty} f(\theta) =
+\infty$, which implies that all sublevel sets $\{\theta \mid f(\theta)
\le c\}$ are convex and compact.
For closed convex $f$, coercivity relates to
``derivatives of $f$ at infinity,'' which in turn connect
to the existence of minimizing sets.
To develop this idea, fix $\theta_0 \in \dom f$, and
recall~\cite[Def.~IV.3.2.3]{HiriartUrrutyLe93} the recession function
\begin{equation*}
  f_\infty'(v)
  = \lim_{t \to \infty} \frac{1}{t}
  (f(\theta_0 + tv) - f(\theta_0))
  = \sup_{t > 0} \frac{f(\theta_0 + t v) - f(\theta_0)}{t},
\end{equation*}
which is positively homogeneous, closed convex in $v$,
and independent of the choice $\theta_0$.
Similarly, for a convex set $C$, any $\theta_0 \in C$ defines the
\emph{recession cone}~\cite[Def.~III.2.2.2]{HiriartUrrutyLe93}
\begin{equation*}
  C_\infty \defeq \{v \mid \theta_0 + tv \in C
  ~ \mbox{for}~ t \ge 0\},
\end{equation*}
which again is independent of the choice $\theta_0$.

Coercivity and the positivity of recession functions coincide for closed
convex $f$:
\begin{lemma}
  \label{lemma:coercive-to-directional-coercive}
  Let $f : \R^d \to \Rup$ be closed convex.
  Then $f$ is coercive if and only if
  $f_\infty'(v) > 0$ for all $v \neq 0$, that is,
  it is coercive on each line.
\end{lemma}
\begin{proof}
  If $f$ is coercive, it is clear that
  $f_\infty'(v) > 0$ for all $v \in \sphere^{d-1}$.
  For the converse, assume that $f_\infty'(v) > 0$ for all
  $v \in \sphere^{d-1}$.
  If $f$ were not coercive, then for some $c < \infty$
  the sublevel set $S \defeq \{\theta \mid f(\theta) \le c\}$
  must be unbounded, and as this set is closed convex,
  it must thus have a recession
  direction $v \in S_\infty$ with $v \neq 0$
  (cf.~\cite[Proposition III.2.2.3]{HiriartUrrutyLe93}).
  But then $f(\theta_0 + t v) \le c$ for all $t$
  implies $\sup_{t > 0} f(\theta_0 + t v) \le c$ and
  so $f_\infty'(v) \le 0$, a contradiction.
\end{proof}

Whenever $f\opt = \inf_\theta f(\theta) > -\infty$,
for $\epsilon \ge 0$ let
\begin{equation*}
  \sublevel_\epsilon(f) \defeq \left\{\theta \mid f(\theta)
  \le f\opt + \epsilon\right\}
\end{equation*}
denote its $\epsilon$-sub-optimality set, which is always
closed convex, and it is non-empty for $\epsilon > 0$.
Then Lemma~\ref{lemma:coercive-to-directional-coercive} immediately implies
the following equivalence result.
\begin{lemma}
  \label{lemma:coercive-iff-compact-argmin}
  Let $f : \R^d \to \Rup$ be proper closed convex.
  Then $\sublevel_0(f) = \argmin f$ is compact and non-empty
  if and only if
  $f$ is coercive if and only if
  $f_\infty'(v) > 0$ for each $v \neq 0$.
\end{lemma}

Optimality and subdifferentials connect via convex conjugates,
including in cases where $f$ may take on infinite
values~\cite[Ch.~X]{HiriartUrrutyLe93b}.
Recall the conjugate $f^*(v)
\defeq \sup_\theta \{\<v, \theta\> - f(\theta)\}$.
If $f : \R^d \to \Rup$ is closed convex, then
\begin{equation}
  \label{eqn:fenchel-duality}
  v \in \partial f(\theta) ~~ \mbox{if and only if} ~~
  \theta \in \partial f^*(v)
  ~~ \mbox{if and only if} ~~
  f(\theta) + f^*(v) = \<\theta, v\>,
\end{equation}
and so
\begin{equation*}
  \argmin_\theta \left\{f(\theta) - \<v, \theta\>\right\}
  = \partial f^*(v).
\end{equation*}
The equality~\eqref{eqn:fenchel-duality} extends the subdifferential to
arbitrary closed convex $f$, so we take it as the definition of $\partial
f$.
Recalling the support function $\sigma_C(v) \defeq \sup_{\theta \in C} \<v,
\theta\>$ of a set $C$, we also have~\cite[Thm.~13.3]{Rockafellar70} that
\begin{equation*}
  \sigma_{\dom f^*}(v) = f_\infty'(v),
\end{equation*}
and therefore $0 \in \interior \dom f^*$ if and only if
$f_\infty'(v) > 0$ for each $v \neq 0$
(cf.~\cite[Thm.~13.1]{Rockafellar70}).
Thus, when $f$ is coercive, or, equivalently,
the set of minimizers $\sublevel_0(f) \defeq \{\theta \mid f(\theta)
= f\opt\}$ is compact (Lemma~\ref{lemma:coercive-iff-compact-argmin}),
$f^*$ is finite in a neighborhood of 0, and
then $\partial f^*(v)$ is the usual subdifferential for
$v$ near 0.
The outer semi-continuity
of the sub-differential of finite convex functions
(Lemma~\ref{lemma:subdifferential-properties},
part~\eqref{item:outer-semictontinuity-subdiff}) implies that
for any $\epsilon > 0$ there therefore exists $\delta > 0$
such that
$\partial f^*(v) \subset \partial f^*(0) + \epsilon \ball_2$
if $\ltwo{v} \le \delta$.

As a consequence of these calculations,
whenever $f$ is coercive, for each $\epsilon > 0$ there exists
$\delta > 0$ such that
\begin{equation*}
  \bigcup_{\ltwo{v} \le \delta} \partial f^*(v)
  \subset \sublevel_0(f) + \epsilon \ball_2.
\end{equation*}
Inverting this inclusion, we obtain the following
proposition.
\begin{proposition}
  \label{proposition:strong-solution-placing}
  Let $f : \R^d \to \Rup$ be a closed convex function and
  assume that $\sublevel_0(f) = \argmin f$ is compact.
  Then for all $\epsilon > 0$, there exists $\delta > 0$
  such that
  \begin{equation*}
    \bigcup_{\theta \in \sublevel_0(f) + \epsilon \ball_2}
    \partial f(\theta) \supset \delta \ball_2.
  \end{equation*}
\end{proposition}

Nothing in Proposition~\ref{proposition:strong-solution-placing} assumes
that $\dom f = \R^d$, so that it allows for constraints.
Most saliently for us, let
$\convexindic{\Theta}(\theta) = 0$
for $\theta \in \Theta$ and $+\infty$ denote the convex indicator of
the set $\Theta$.
Then $\poploss_Q + \convexindic{\Theta}$ is closed convex for any finitely
supported $Q$.
In particular, Lemma~\ref{lemma:compact-minimizers-directable}
follows as an immediate consequence of
Proposition~\ref{proposition:strong-solution-placing},
where the conclusion
that $\partial (\poploss_Q + \convexindic{\Theta})(\theta)$
is non-empty is one of the results.


\subsection{Proof of Theorem~\ref{theorem:minimax-lower}}
\label{sec:proof-minimax-lower}

We
follow the outline in Section~\ref{sec:proof-outline}
to prove Theorem~\ref{theorem:minimax-lower} here.

\subsubsection{Proof of Theorem~\ref{theorem:minimax-lower}: the
  one-dimensional case}
\label{sec:proof-bounded-lower-one-d}


\paragraph{Constructing well-separated functions.}
Without loss of generality (by symmetry), consider the case that the lower
Lipschitz constant~\eqref{eqn:signed-lip-consts} satisfies
$\lipconst_\theta^- = -\infty$ for some $\theta \in \interior \Theta$.
Choose
$\theta_0 \in \interior \Theta$ and $\delta > 0$
to satisfy $\theta_0 < \theta_0 + \delta <
\theta$, so that $\lipconst^-_{\theta_0} = \lipconst^-_{\theta_0 + \delta} =
-\infty$ by Lemma~\ref{lemma:lipschitz-consts-monotone}.
We first ``place'' the gradients and minimizers
(recall Fig.~\ref{fig:plant-hyperplane}).
By Assumption~\ref{assumption:achievable} that minimizers are achievable,
there exists $z^+$ such that $\deriv_+ \loss_{z^+}(\theta_0) = \sup
\{\partial \loss_{z^+}(\theta_0)\} > 0$ and $\deriv_+\loss_{z^+}(\theta_0) \le
\deriv_+ \loss_{z^+}(\theta_0 + \delta) < \infty$.
We use the shorthands $a = \loss_{z^+}'(\theta_0) > 0$ and $b =
\loss_{z^+}'(\theta_0 + \delta) > 0$ for some strictly positive elements of
the subdifferentials $\partial \loss_{z^+}(\theta_0)$ and $\partial
\loss_{z^+}(\theta_0 + \delta)$, noting that we may treat them as
constants---they do not depend on the sample size $n$ or any other
distribution or sample-dependent quantities.

For any $\lipconst < \infty$, there exists $z_\lipconst$ such
that
\begin{equation*}
  \deriv_+ \loss_{z_\lipconst}(\theta_0 + \delta) \le -\lipconst.
\end{equation*}
Now, define the point distributions
\begin{equation*}
  P_0 \defeq \pointmass_{z^+}
  ~~ \mbox{and} ~~
  P_1 \defeq \left(1 - \frac{1}{n^2}\right)
  \pointmass_{z^+}
  + \frac{1}{n^2} \pointmass_{z_\lipconst}.
\end{equation*}
With these choices, observe that there is a subderivative
$\poploss_{P_0}'(\theta_0) \in \partial \poploss_{P_0}(\theta_0)$ satisfying
\begin{equation*}
  \poploss_{P_0}'(\theta_0)
  = \loss_{z^+}'(\theta_0) = a > 0.
\end{equation*}
Similarly, for $\epsilon = \frac{1}{n^2}$,
\begin{align*}
  \poploss'_{P_1}(\theta_0 + \delta)
  & = \loss_{z^+}'(\theta_0 + \delta)
  + \epsilon \left[\loss_{z_\lipconst}'(\theta_0 + \delta)
    - \loss_{z^+}'(\theta_0 + \delta)\right] \\
  & \le (1 - \epsilon) b - \epsilon \lipconst.
\end{align*}
As $\lipconst$ was arbitrary, we may choose
$\lipconst = n^3$, so that
$\poploss'_{P_1}(\theta_0 + \delta)
\le -n/2$ for large $n$.

From these calculations (and recall Fig.~\ref{fig:plant-hyperplane}), we see
that for all $t \ge 0$,
\begin{align*}
  \poploss_{P_0}(\theta_0 + t)
  - \poploss_{P_0}\opt
  \ge \poploss_{P_0}(\theta_0 + t)
  - \poploss_{P_0}(\theta_0)
  & \ge a t   ~~ \mbox{while} ~~
  \\
  \poploss_{P_1}(\theta_0 + \delta - t)
  - \poploss_{P_1}\opt
  \ge \poploss_{P_1}(\theta_0 + \delta - t)
  - \poploss_{P_1}(\theta_0 + \delta)
  & \ge \frac{n t}{2}.
\end{align*}
So at least one of the population losses
$\poploss_{P_0}$ or $\poploss_{P_1}$ must be
nontrivially larger than its infimum, and in particular for $n \ge 6a$,
\begin{equation}
  \label{eqn:optimization-separation}
  \dopt(\poploss_{P_0}, \poploss_{P_1})
  \ge \inf_t \max\left\{at, \frac{n}{2} (\delta - t)\right\}
  = \frac{a n \delta / 2}{a + n/2}
  = \frac{an \delta}{2 a + n}
  \ge \frac{3a\delta}{4}.
\end{equation}

\paragraph{The testing guarantee.}
With the optimization separation~\eqref{eqn:optimization-separation},
we can now prove the desired lower bound.
By Lemma~\ref{lemma:tv-and-hellinger},
for $\epsilon = \frac{1}{n^2}$,
$\tvnorm{P_0 - P_1} \le \epsilon$ and so
\begin{equation*}
  \tvnorm{P_0^n - P_1^n}
  \le \sqrt{1 - (1 - n^{-2})^{2n}}
  = \frac{(\sqrt{2} + o(1))}{\sqrt{n}}.
\end{equation*}
This gives the minimax lower bound
\begin{equation*}
  \minimaxlow_n(\loss, \mc{Z}, \Theta)
  \ge \frac{\dopt(\poploss_{P_0}, \poploss_{P_1})}{2}
  \left(1 - \tvnorm{P_0^n - P_1^n}\right)
  \ge \frac{a \delta}{4},
\end{equation*}
for $n \ge \max\{6a,18\}$. The resulting constant lower bound is also valid
for all $n$ by monotonicity of the minimax risk.
Of course, as we note in the beginning of the proof, $\delta > 0$ is fixed
and independent of $n$, as is $a = \loss_{z^+}'(\theta_0) > 0$.

\subsubsection{Proof of Theorem~\ref{theorem:minimax-lower}: the
  general case}
\label{sec:proof-compact-lower-general}

We follow the outline of the proof of the one-dimensional case and
that present in
Section~\ref{sec:proof-outline}.
We find two well-separated finitely
supported losses by ``placing'' gradients in directions
that most separate the functions according to
the optimization distance~\eqref{eqn:opt-distance}.
When Condition~\ref{cond:compact-case} fails, the
definition~\eqref{eqn:main-lip-const} of the Lipschitz constant
$\lipconst_\theta$ shows that there is some compact set $\Theta_0 \subset
\interior \Theta$ for which $\sup_{\theta \in \Theta_0} \lipconst_\theta =
+\infty$.
Taking a subsequence if necessary, we may find a convergent
sequence $\theta_n \to \theta_0 \in \Theta_0$ satisfying
$\lim_n \lipconst_{\theta_n} = +\infty$.
In particular, we have a point $\theta_0$
where
for each $\epsilon > 0$,
\begin{equation*}
  \sup_{z \in \mc{Z}}
  \sup_{\ltwo{\theta - \theta_0} < \epsilon}
  \sup_{g \in \partial \loss_z(\theta)} \ltwo{g}
  = +\infty.
\end{equation*}
Then for any $\epsilon > 0$ and $m < \infty$, we may find
$z^*_m$ and $\theta_m$ with $\ltwo{\theta_m -
  \theta_0} < \epsilon$ satisfying
\begin{equation*}
  \sup_{g \in \partial \loss_{z^*_m}(\theta_m)} \ltwo{g} = m
\end{equation*}
(if $\ltwo{g} > m$, we may simply increase $m \uparrow \ltwo{g}$, so
assuming equality is no loss of generality).
Because $\partial \loss_{z^*_m}(\theta_m)$ is compact, the supremum is
attained at some vector $g_m \in \partial \loss_{z^*_m}(\theta_m)$, and
we let $v_m \defeq g_m / \ltwo{g_m} \in \sphere^{d-1}$ be the direction,
and moving to a subsequence if necessary,
we may assume without loss of generality
that $v_m \to v \in \sphere^{d-1}$.

We use this direction and the directability assumptions in the
definition~\eqref{eqn:achievable-minimizers} of the achievable
set to construct losses whose gradients point in opposite directions,
exactly as in the one-dimensional case
and in the ``gradient placing'' construction
in Figure~\ref{fig:plant-hyperplane} and
Section~\ref{sec:proof-outline}.
The following technical lemma, whose proof we defer to
Section~\ref{sec:proof-i-can-find-directable}, provides the key.
\begin{lemma}
  \label{lemma:i-can-find-directable}
  Let Assumption~\ref{assumption:achievable}
  hold and $\theta_0 \in \interior \Theta$.
  Then there exists $\alpha_0 > 0$ such that for all
  $v \in \sphere^{d-1}$,
  there exists a convex compact set $K \subset \Theta$ such that
  $\<v, \theta_0\> + \alpha_0
  \le \inf_{\theta \in K} \<v, \theta\>$.
  Additionally,
  for this $K$ there exists $\beta_0 > 0$ and a finite set 
  $S\subset \mc{Z}$ such that
  for each $u \in \sphere^{d-1}$, there exists
  $Q\in \mc{P}(S)$ such that
  \begin{equation*}
    \beta u \in \partial (\poploss_Q + \convexindic{\Theta})(\theta)
    ~~ \mbox{for~some}~ \theta \in K
    ~ \mbox{and} ~ \beta \ge \beta_0.
  \end{equation*}
  The point $\theta$ may be taken of the form $\theta = \theta_0 + \alpha v
  + v_\perp$, where $\<v_\perp, v\> = 0$, $\ltwo{v_\perp} \le
  \frac{1}{\alpha_0}$, and $\alpha_0 \le \alpha \le \frac{1}{\alpha_0}$.
\end{lemma}

Lemma~\ref{lemma:i-can-find-directable}
guarantees that we can find a finitely supported
distribution $Q$ with
\begin{equation*}
  \theta_1 = \theta_0 + \alpha v + v_{\perp},
  ~~~ \mbox{and} ~~~
  -\beta v_m \in \partial (\poploss_Q + \convexindic{\Theta})(\theta_1),
\end{equation*}
where $\ltwo{v_{\perp}} \le \frac{1}{\alpha_0}$, $\<v_{\perp}, v\> = 0$,
and $\alpha \ge \alpha_0 > 0$, $\beta \ge \beta_0
> 0$.
Therefore
\begin{equation*}
  \poploss_Q(\theta) \ge \poploss_Q(\theta_1) -
  \beta \<v_m, \theta - \theta_1\>.
\end{equation*}
Assume that $m$ is large enough that
$\ltwo{v_m - v} \le \epsilon$,
so that 
$\<v_m, v\> = 1 - \half \ltwo{v_m - v}^2
\ge 1 - \frac{\epsilon^2}{2}$ while
$|\<v_m, v_\perp\>| = |\<v_m - v, v_\perp\>|
\le \epsilon / \alpha_0$.
Consider the point $\wb{\theta} = \half
(\theta_m + \theta_1)$ halfway between $\theta_1$ and $\theta_m$.
If $\<\theta, v_m\> \le \<\wb{\theta}, v_m\>$, we have
\begin{align*}
  \poploss_Q(\theta)
  \ge \poploss_Q(\theta_1) - \beta \<v_m, \theta - \theta_1\>
  & \ge \poploss_Q(\theta_1) - \beta \<v_m, \wb{\theta} - \theta_1\>
  = \poploss_Q(\theta_1)
  - \frac{\beta}{2} \<v_m, \theta_m - \theta_1\> \\
  & = \poploss_Q(\theta_1)
  - \frac{\beta}{2} \<v_m, \theta_m - \theta_0 - \alpha v - v_{\perp}\> \\
  & = \poploss_Q(\theta_1) + \frac{\beta \alpha}{2}
  \<v_m, v\>
  - \frac{\beta}{2} \<v_m, \theta_m - \theta_0 - v_\perp\> \\
  & > \poploss_Q(\theta_1) + \frac{\beta \alpha}{2}
  \left(1 - \frac{\epsilon^2}{2}\right)
  - \frac{\beta}{2} \left(\epsilon + \frac{\epsilon}{\alpha_0}\right).
\end{align*}
because $\ltwo{\theta_m - \theta_0} < \epsilon$
by assumption.
%
On the other hand, if $\<\theta, v_m\> \ge \<\wb{\theta}, v_m\>$,
then we use $\ltwo{g_m} = m$ to obtain
\begin{align*}
  \loss_{z^*_m}(\theta)
  \ge \loss_{z^*_m}(\theta_m) + m \<g_m / m,
  \theta - \theta_m\>
  & = \loss_{z^*_m}(\theta_m)
  + m \<v_m, \theta - \theta_m\> \\
  & \ge \loss_{z^*_m}(\theta_m)
  + m \<v_m, \wb{\theta} - \theta_m\> \\
  & = \loss_{z^*_m}(\theta_m)
  + \frac{m}{2}
  \<v_m, \theta_1 - \theta_m\> \\
  & = \loss_{z^*_m}(\theta_m)
  + \frac{m}{2}
  \left(\<v_m, \alpha v + v_{\perp}\> + \<v_m, \theta_0 - \theta_m\>\right) \\
  & \ge \loss_{z^*_m}(\theta_m)
  + \frac{m}{2}\left(\alpha\left(1 - \frac{\epsilon^2}{2}\right)
  - \frac{\epsilon}{\alpha_0}
  - \epsilon\right).
\end{align*}

With these choices, assume $\epsilon$ is small enough that $\alpha (1 -
\frac{\epsilon^2}{2}) - \frac{\epsilon}{\alpha_0} - \epsilon > \half
\alpha$, and define the two probability distributions
\begin{equation*}
  P_0 \defeq \left(1 - \frac{1}{n}\right) Q + \frac{1}{n} \pointmass_{z^*_m}
  ~~ \mbox{and} ~~
  P_1 \defeq Q.
\end{equation*}
Then
\begin{equation*}
  \poploss_{P_1}(\theta)
  \ge \poploss_{P_1}\opt + \frac{\beta \alpha}{4}
  ~~ \mbox{if} ~~
  \<v_m, \theta\> \le \<v_m, \wb{\theta}\>.
\end{equation*}
Consider the converse case that $\<v_m, \theta\> \ge \<v_m, \wb{\theta}\>$.
By Lemma~\ref{lemma:no-infinite-values}, which we discuss in our analysis 
of unbounded domains, we may assume that $\loss_z^\star(\Theta)>-\infty$
for every $z\in \mc{Z}$, since the minimax risk is infinite otherwise.
Because $Q$ is finitely supported it follows that
$\poploss_Q\opt(\Theta) > -\infty$.
So if $\<v_m, \theta\> \ge \<v_m, \wb{\theta}\>$, then
\begin{align*}
  \poploss_{P_0}(\theta)
  = \left(1 - \frac{1}{n}\right)
  \poploss_Q(\theta)
  + \frac{1}{n} \loss_{z^*_m}(\theta)
  & \ge \left(1 - \frac{1}{n}\right)
  \poploss_Q\opt
  + \frac{1}{n} \loss_{z^*_m}(\theta_m)
  + \frac{m \alpha}{4 n} \\
  & \ge \poploss_{P_0}(\theta_m)
  + \frac{m \alpha}{4n}
  + \left(1 - \frac{1}{n}\right)
  \left(\poploss_Q\opt - \poploss_Q(\theta_m)\right).
\end{align*}
By Lemma~\ref{lemma:i-can-find-directable}, the distributions $Q$, which may 
depend on $m$, are supported on a common finite set $S\subset \mc{Z}$.
Because of the shared support, the losses $\poploss_Q$ are uniformly Lipschitz 
in any compact neighborhood of $\theta_0$, and so 
$\poploss_Q\opt - \poploss_Q(\theta_m)$ is uniformly bounded for all $m$.
The choice of $m < \infty$ was otherwise arbitrary,
so we may take it large enough that
\begin{equation*}
  \poploss_{P_0}(\theta) \ge \poploss_{P_0}\opt + 1
  ~~ \mbox{if}~~
  \<v_m, \theta\> \ge \<v_m, \wb{\theta}\>.
\end{equation*}
Combining the two calculations, we see that for
any $\theta$,
\begin{align*}
  \poploss_{P_0}(\theta)
  + \poploss_{P_1}(\theta)
  - \poploss_{P_0}\opt - \poploss_{P_1}\opt
  \ge \min\left\{\frac{\beta \alpha}{4}, 1 \right\}.
\end{align*}
That is, for some (problem-dependent)
constant $c > 0$, independent of the sample size $n$,
we have constructed $P_0$ and $P_1$ so that
\begin{align*}
  \dopt(\poploss_{P_0}, \poploss_{P_1}) \ge c,
\end{align*}
while Lemma~\ref{lemma:tv-and-hellinger} yields
\begin{align*}
  \tvnorm{P_0^n - P_1^n}
  \le (1 + o_n(1)) \sqrt{1 - e^{-1}} \sqrt{1 + e^{-1}}
  \to \sqrt{1 - e^{-2}}.
\end{align*}
The reduction from optimization to testing~\eqref{eqn:minimax-to-testing}
then gives the theorem.

\subsubsection{Proof of Lemma~\ref{lemma:i-can-find-directable}}
\label{sec:proof-i-can-find-directable}

Let $\theta_0 \in \interior \Theta$.
Then for any $v \in \sphere^{d-1}$ and $t = \<v, \theta_0\>$, the sets
$H_{v,t} \cap \Theta$ and $H_{v,t}^c \cap \Theta$ are each convex sets with
interiors.

We first demonstrate the existence of an $\alpha_0 > 0$
in the lemma.
Take $\alpha_0 > 0$ small enough that $H_{v, t + \alpha_0} \cap \Theta$ and
$H_{v, t + \alpha_0}^c \cap \Theta$ have interiors for all $v \in
\sphere^{d-1}$, which is immediately possible because $\theta_0 \in
\interior \Theta$.
We wish to show that for each $v \in \sphere^{d-1}$
and $t = \<v, \theta_0\>$,
there is a directable set $C$ satisfying
$C \subset H_{v, t + \alpha_0}^c$.
Assume for the sake of contradiction that
for all $\alpha_0 > 0$, there exists $v \in \sphere^{d-1}$
with $t = \<v, \theta_0\>$ such that
for all directable sets $C$,
\begin{equation*}
  H_{v, t + \alpha_0} \cap C \neq \emptyset.
\end{equation*}
Then $(v, t + \alpha_0) \in \mc{U}$, that is, the pair $(v, t + \alpha_0)$
is unconstraining.
Because $H_{v, t + \alpha_0}^c \cap \interior \Theta \neq \emptyset$,
we have $H_{v,t + \alpha_0} \cap \Theta \subsetneq \Theta$,
contradicting Assumption~\ref{assumption:achievable}.



By the Definition~\ref{definition:directable} of directability, for all
$\epsilon > 0$ there exists a compact convex $C_\epsilon$ between $C$ and
$\Theta$, i.e., $C \subset C_\epsilon \subset \Theta$, with $\dist(\theta,
C) \le \epsilon$ for all $\theta \in C_\epsilon$, and a finite collection
$\{z_i\}_{i = 1}^k$ and $\beta_0 > 0$ such that
\begin{equation*}
  \bigcup_{Q \in \mc{P}(\{z_i\}_{i=1}^k)}
  \bigcup_{\theta \in C_\epsilon} \partial
  \left(\poploss_Q + \convexindic{\Theta}\right)(\theta)
  \supset \beta_0 \ball_2.
\end{equation*}
Because $\alpha_0 > 0$, we 
can certainly take $\epsilon$ small enough
that $C + \epsilon \ball_2
\subset H_{v, t + \alpha_0/2}^c$,
and taking $K = C_\epsilon$ proves the first statement of the lemma.
Every element of
$\Theta \cap H_{v,t + \alpha_0/2}^c$ has
the form $\theta_0 + \alpha v + v_\perp$ for some
$\alpha \ge \alpha_0 / 2$ and $\<v_\perp, v\> = 0$.
That we may take $\ltwo{v_\perp}$ bounded and assume $\alpha \le
\frac{1}{\alpha_0}$ follows because $C_\epsilon$ is compact, so that
decreasing $\alpha_0$ does not change the conclusions of the lemma.

\subsection{Proof of Theorem~\ref{theorem:minimizable}}
\label{sec:proof-minimizable}

If $\lipconst \defeq \sup_{\theta \in \Theta} \lipconst_\theta < \infty$,
the result is trivial: the output $\what{\theta}_n$, the average of $n$ steps
of the stochastic subgradient method, achieves risk
\begin{equation}
  \label{eqn:sgd-convergence}
  \E[\poploss_P(\what{\theta}_n) - \poploss_P\opt] \le O(1) \frac{\lipconst
    \diam(\Theta)}{\sqrt{n}}
\end{equation}
with an appropriate stepsize~\cite{NemirovskiJuLaSh09}.
We now handle the case that
the Lipschitz constants~\eqref{eqn:main-lip-const}
can explode on the boundaries, so that $\sup_{\theta \in \Theta}
\lipconst_\theta = \infty$, while Condition~\ref{cond:compact-case} holds.


%

\newcommand{\thetacent}{\theta_0}
\newcommand{\cent}{_{\theta_0}}

We first argue that no loss can decrease too much on a compact set.
Define the ``star'' interior of $\Theta$ for $\delta \in [0, 1]$ by
\begin{equation*}
  \Theta_{\delta}
  \defeq \left\{\thetacent + t (\theta - \thetacent)
  \mid 0 \le t \le 1 - \delta,
  \theta \in \Theta \right\}
  = \delta\theta_0 + (1 - \delta) \Theta
\end{equation*}
which satisfies $\interior \Theta \supset \Theta_\delta$
for $\delta > 0$, and $\Theta_\delta$ is compact convex.

\begin{lemma}
  \label{lemma:do-not-shrink}
  Let $\loss_z : \interior \Theta \to \R$ be convex
  and $\thetacent \in \interior
  \Theta$ have finite local Lipschitz constant $\lipconst\cent \defeq
  \sup_{z \in \mc{Z}} \sup_{g \in \partial
    \loss_z(\theta_0)} \ltwo{g} < \infty$.
  Then for any $\theta \in \Theta$, there exists $\theta_\delta \in
  \Theta_\delta$ such that
  \begin{align*}
    \loss_z(\theta_\delta) \le \loss_z(\theta) + \delta \lipconst\cent
    \diam(\Theta)
    ~~ \mbox{for~all~} z \in \mc{Z}.
  \end{align*}
\end{lemma}
\begin{proof}
  Observe that for any $\theta, \theta' \in \Theta$, if
  $\loss_z'(\theta) \in \partial \loss_z(\theta)$, then
  \begin{align*}
    \loss_z(\theta') \ge \loss_z(\theta)
    + \<\loss_z'(\theta), \theta' - \theta\>.
  \end{align*}
  Fix $\theta \in \Theta$ and let $\theta_t = \thetacent + t(\theta -
  \theta_0)$, which for $0 \le t < 1$ satisfies $\theta_t \in \interior
  \Theta$ and
  \begin{align*}
    \loss_z(\theta) \ge \loss_z(\theta_t) + \<\loss_z'(\theta_t),
    \theta - \theta_t\>
    & = \loss_z(\theta_t) + (1 - t) \cdot
    \<\loss_z'(\theta_0 + t(\theta - \theta_0)),
    \theta - \theta_0\> \\
    & \ge \loss_z(\theta_t) + (1 - t) \cdot
    \<\loss_z'(\theta_0),
    \theta - \theta_0\>
  \end{align*}
  by the monotonicity properties of the subdifferential.
  Observe that for any $\theta \not\in \Theta_\delta$, the
  segment $\theta_0 + t (\theta - \theta_0) \in \Theta_\delta$ for
  $0 \le t \le 1 - \delta$, and in particular, at $t = 1 - \delta$ the
  point $\theta_t = \theta_0 + t(\theta - \theta_0)$ satisfies
  \begin{align*}
    \loss_z(\theta) \ge \loss_z(\theta_t)
    + \delta \<\loss_z'(\theta_0), \theta - \theta_0\>
    \ge \loss_z(\theta_t) - \delta \lipconst\cent \diam(\Theta).
  \end{align*}
  Rearrange to obtain the lemma.
\end{proof}

Lemma~\ref{lemma:do-not-shrink} makes apparent that the Lipschitz
constant~\eqref{eqn:main-lip-const} guarantees that for any distribution $P$
on $\mc{Z}$,
\begin{equation*}
  \inf_{\theta \in \Theta_\delta}
  \poploss_P(\theta)
  \le \inf_{\theta \in \Theta} \poploss_P(\theta)
  + \delta \lipconst\cent \diam(\Theta).
\end{equation*}
In particular, letting
$\lipconst(\Theta_\delta) = \sup_{\theta \in \Theta_\delta}
\lipconst_\theta$ and performing the stochastic gradient method for $n$ steps
on the set $\Theta_\delta$ gives
\begin{equation*}
  \E\left[\poploss_P(\what{\theta}_n) - \poploss_P\opt\right]
  \le \delta \lipconst\cent \diam(\Theta)
  + \frac{\lipconst(\Theta_\delta) \diam(\Theta)}{\sqrt{n}},
\end{equation*}
via the bound~\eqref{eqn:sgd-convergence}.
Take $\delta = \delta_n \to 0$ slowly enough that the right side tends to
zero.

\subsection{Proof of Proposition~\ref{proposition:no-rate}}
\label{sec:proof-no-rate}

Without loss of generality we assume $r(1)=1$, as
otherwise we apply the same argument to the normalized rate function
$\bar{r}(x)=r(x)/r(1)$.
The rate function has a continuous increasing inverse $r^{-1} :
\openright{1}{\infty} \to \openright{1}{\infty}$, and we define the
losses $\loss_z : [0, 1] \to \Rup$ for $z \in \{0, 1\}$ by
\begin{equation*}
  \ell_z(\theta) \defeq \theta + z \int_{\theta}^{1} r^{-1}(1/t) dt,
\end{equation*}
where we define $\ell_1(0) \defeq \int_0^1 r^{-1}(1/t)dt$, which might 
be $+\infty$.
The convexity of $\loss_z$ follows immediately, as
\begin{align*}
  \ell_z'(\theta) = 1 - z r^{-1}(1/\theta)
  ~~ \mbox{and} ~~
  \ell_z''(\theta) =\frac{z}{\theta^2}\left(r^{-1}\right)'(1/\theta)\geq 0
\end{align*}
(where we tacitly use that $(r^{-1})'$ exists almost everywhere).
For $\delta\in (0,1/2)$ let us now define distributions
\begin{equation*}
  P_0 \defeq \pointmass_{0}
  ~~ \mbox{and} ~~
  P_\delta \defeq \left(1 -\delta\right)
  \pointmass_{0}
  + \delta \pointmass_{1}.
\end{equation*}
Under $P_0$ we have $L_{P_0}(\theta) = \theta$, which satisfies $L_{P_0}(0)
= \inf_{\theta \in [0, 1]} \poploss_{P_0}(\theta) = 0$, while under
$P_\delta$
\begin{align*}
  L_{P_\delta}(\theta)&=\theta+\delta \int_{\theta}^{1}r^{-1}(1/x)dx,
\end{align*}
has minimizer $\theta_\delta=\frac{1}{r(1/\delta)}$.

We lower bound the optimization distance via
\begin{align*}
  \dopt(L_{P_0},L_{P_\delta})\geq \inf_{\theta \in [0,1]}
  \half \left(L_{P_0}(\theta)+L_{P_\delta}(\theta)- 0 -
  L_{P_\delta}(\theta_\delta)\right).
\end{align*}
Evidently $\theta_{\delta/2}$ minimizes
$L_{P_0}(\theta) + L_{P_\delta}(\theta)$, so that
\begin{align*}
  \dopt(L_{P_0},L_{P_\delta})&\geq \half\left(\frac{2}{r(2/\delta)}-\frac{1}{r(1/\delta)}+\delta\int_{1/r(2/\delta)}^{1/r(1/\delta)}r^{-1}(1/x)dx\right)\\
  &\geq \half\left(\frac{2}{r(2/\delta)}-\frac{1}{r(1/\delta)}+\delta\left(\frac{1}{r(1/\delta)}-\frac{1}{r(2/\delta)}\right)r^{-1}(r(1/\delta))\right)\\
  &= \frac{1}{2r(2/\delta)},
\end{align*}
where we have again used that $r^{-1}$ is increasing.
Choosing $\delta=\frac{1}{n}$ and applying Lemma~\ref{lemma:tv-and-hellinger},
we obtain
\begin{equation*}
	\minimaxlow_n(\loss, \mc{Z}, \Theta) \ge
	\frac{1}{4r(2n)} \left(1 - \sqrt{1 - e^{-2}}\right).
\end{equation*}
Set the numerical constant $c = \frac{1 - \sqrt{1 - e^{-2}}}{4}$.




\subsection{Proof of Theorem~\ref{theorem:unbounded-case}:
  the one-dimensional case}
\label{sec:proof-unbounded-lower}

We begin with a few preliminary results before
moving to the proof proper of Theorem~\ref{theorem:unbounded-case}
when $\Theta \subset \R$.
Throughout, we assume Condition~\ref{cond:compact-case} holds,
because otherwise, Theorem~\ref{theorem:minimax-lower} gives the result.
First, we show the equivalence between
the conditions~\eqref{eqn:cannot-be-far-in-compact}
and~\eqref{eqn:cannot-be-far-in-interval}.
\begin{lemma}
  \label{lemma:conditions-same-one-dim}
  Assume that $\inf \Theta = -\infty$.
  Then the following two statements are equivalent:
  \begin{enumerate}[(i)]
  \item \label{item:distribution-to-pointwise-minimization}
    For all $\epsilon > 0$, there exists
    $t > -\infty$ such that
    $\inf_{\theta \ge t} [\poploss_Q(\theta) - \poploss_Q\opt]
    \le \epsilon$ for all $Q \in \Pdiscrete(\mc{Z})$.
  \item \label{item:compactness-minimization}
    For all $\epsilon > 0$, there exists
    $t > -\infty$ such that
    $\inf_{\theta \ge t}
    \left[\loss_z(\theta) - \loss_z\opt\right] \le \epsilon$
    for all $z \in \mc{Z}$.
  \end{enumerate}
\end{lemma}
\begin{proof}
  Clearly \eqref{item:distribution-to-pointwise-minimization} implies
  \eqref{item:compactness-minimization}.
  For the converse, let $t_0 > -\infty$
  satisfy $\inf_{\theta \ge t_0} \loss_z(\theta) \le \loss_z\opt
  + \epsilon$, and consider any distribution $P$ for which $\poploss_P$ is
  well-defined~\eqref{eqn:well-defined}.
  Let $\theta\opt(P) \in \argmin_\theta \poploss_P(\theta)$, where we
  tacitly let $\theta\opt(P)$ be in the extended reals $[-\infty, \infty]$.
  If $\theta\opt = \theta\opt(P) < t_0$, then
  \begin{equation*}
    \E_P[\loss_Z(t_0)]
    = \poploss_P(\theta\opt) +
    \E_P\left[\loss_Z(t_0) - \loss_Z(\theta\opt)\right]  
  \end{equation*}
  so by assumption~\eqref{eqn:cannot-be-far-in-interval},
  the claim~\eqref{item:distribution-to-pointwise-minimization}
  holds:
  \begin{equation*}
    \loss_z(t_0) - \loss_z(\theta\opt)
    \le 
    \begin{cases}
      \loss_z(t_0) - \loss_z\opt \le \epsilon & \mbox{if}~
      \deriv_+\loss_z(t_0) \ge 0 \\
      0 & \mbox{if}~ \deriv_+\loss_z(t_0) < 0.
    \end{cases}
  \end{equation*}
  This gives $\E_P[\loss_Z(t_0)]
  \le \poploss_P(\theta\opt) + \epsilon$ as desired.
\end{proof}


\subsubsection{Eliminating infinite losses}

We begin the proof of Theorem~\ref{theorem:unbounded-case} (when $d = 1$)
by
eliminating issues that arise if losses can take on infinite (or
asymptotically infinite) values, showing that $\minimaxlow_n =
+\infty$ in these cases.

\begin{lemma}
  \label{lemma:no-infinite-one-d}
  If there exists $z \in \mc{Z}$ such that
  $\loss_z\opt(\Theta) = -\infty$, then
  $\minimaxlow_n(\loss, \mc{Z}, \Theta) = +\infty$.
\end{lemma}
\begin{proof}
  Let $z$ be such that $\loss_z\opt(\Theta) = -\infty$,
  and take any $\theta_0 \in \interior \Theta$ at which
  $\loss_z$ is differentiable (this occurs for a.e.\ $\theta_0$).
  Because $\loss_z$ is proper, w.l.o.g., we may assume $\lim_{t
    \downarrow -\infty} \loss_z(t) = -\infty$, and so
  $\loss_z'(\theta_0) > 0$ for all $\theta_0$.
  Now, let $\theta_1 > \theta_0$, and find $z_1$ such that
  $\deriv_+\loss_{z_1}(\theta_1) < 0$, which exists
  by Assumption~\ref{assumption:achievable};
  let $\beta$ solve
  $-\beta \loss_z'(\theta_0) \in \partial \loss_{z_1}(\theta_1)$,
  so that $\beta > 0$.
  Then
  \begin{align*}
    \loss_{z_1}(\theta) + \beta \loss_{z}(\theta)
    & \ge \loss_{z_1}(\theta_1)
    - \beta\loss_z'(\theta_0) (\theta - \theta_1)
    + \beta \loss_z(\theta_0) + \beta \loss_z'(\theta_0)(\theta - \theta_0)\\
    & = \loss_{z_1}(\theta_1)
    + \beta \loss_z(\theta_0)
    + \beta \loss_z'(\theta_0)(\theta_1 - \theta_0)
    > \loss_{z_1}(\theta_1) + \beta \loss_z(\theta_0).
  \end{align*}
  Thus
  \begin{equation*}
    \inf_\theta \left[\loss_{z_1}(\theta) + \beta \loss_z(\theta)\right]
    > \loss_{z_1}(\theta_1) + \beta \loss_z(\theta_0) > -\infty,
  \end{equation*}
  and so defining the probability distributions
  \begin{equation*}
    P_0 \defeq \pointmass_z
    ~~ \mbox{and} ~~
    P_1 \defeq \frac{1}{1 + \beta} \pointmass_{z_1}
    + \frac{\beta}{1 + \beta} \pointmass_z
  \end{equation*}
  we obtain
  \begin{equation*}
    \dopt(\poploss_{P_0}, \poploss_{P_1}) = +\infty.
  \end{equation*}
  Using the shorthand $\epsilon = 1 - \frac{1}{1 + \beta}
  = \frac{\beta}{1 + \beta}$,
  we have $\tvnorm{P_0 - P_1}
  \le 1 - \epsilon$, and
  Lemma~\ref{lemma:tv-and-hellinger}
  implies $\tvnorm{P_0^n - P_1^n}
  \le \sqrt{1 - \epsilon^{2n}} < 1$.
  Inequality~\eqref{eqn:minimax-to-testing} shows
  that $\minimaxlow_n(\loss, \mc{Z}, \Theta) = +\infty$.
\end{proof}

\begin{lemma}
  \label{lemma:interior-theta-bounded-gaps-1d}
  Let Condition~\ref{cond:compact-case} hold.
  If for some $\theta_0 \in \interior \Theta$ we have
  \begin{equation*}
    \sup_{z \in \mc{Z}}
    \left[\loss_z(\theta_0) - \loss_z\opt(\Theta)\right] = \infty,
  \end{equation*}
  then
  $\minimaxlow_n(\loss, \mc{Z}, \Theta) = +\infty$.
\end{lemma}
\noindent
The proof of this result in the general case is no more
complex than in the one-dimensional case, so we
simply refer to Lemma~\ref{lemma:interior-theta-bounded-gaps}.

\subsubsection{Planting solutions}
\label{sec:one-d-plant-unbounded}

Having shown by Lemma~\ref{lemma:interior-theta-bounded-gaps-1d}
that unless for all
$\theta \in \interior \Theta$,
\begin{equation}
  \sup_{z \in \mc{Z}} [\loss_z(\theta) - \loss_z\opt(\Theta)] < \infty,
  \label{eqn:no-infinite-one-d}
\end{equation}
the minimax risk is infinite, we now consider the case that
Condition~\ref{cond:unbounded-case} fails while
inequality~\eqref{eqn:no-infinite-one-d} holds.
As a consequence of inequality~\eqref{eqn:no-infinite-one-d},
all loss functions $\loss_z$
necessarily grow (or at least cannot decrease) as the parameter
$|\theta| \to \infty$:
\begin{lemma}
  \label{lemma:small-derivatives-eventually}
  Let inequality~\eqref{eqn:no-infinite-one-d} hold
  for all $\theta \in \interior \Theta$.
  Then
  \begin{equation*}
    \lim_{\theta \downarrow -\infty} \sup_{z \in \mc{Z}}
    \deriv_+ \loss_z(\theta) \le 0
    ~~ \mbox{and} ~~
    \lim_{\theta \uparrow \infty} \inf_{z \in \mc{Z}}
    \deriv_- \loss_z(\theta) \ge 0.
  \end{equation*}
\end{lemma}
\begin{proof}
  We prove the left claim;
  the right is similar.
  Assume for the sake of contradiction $\sup_{z \in \mc{Z}} \deriv_+
  \loss_z(\theta) \ge a > 0$ for all $\theta \in \R$.
  Then fixing $\theta_0 \in \interior \Theta$,
  let $\theta = \theta_0 - t$ for some (arbitrarily large) $t > 0$.
  Choose $z$ such that $\deriv_+ \loss_z(\theta) \ge a/2$.
  Then
  $\loss_z(\theta_0) \ge \loss_z(\theta) + \frac{at}{2}$, and
  taking $t \uparrow \infty$ would contradict that
  $\sup_z [\loss_z(\theta_0) - \loss_z\opt(\Theta)] < \infty$.
\end{proof}

\begin{figure}
  \begin{center}
    \begin{tabular}{cc}
      \begin{overpic}[width=.6\columnwidth]{
          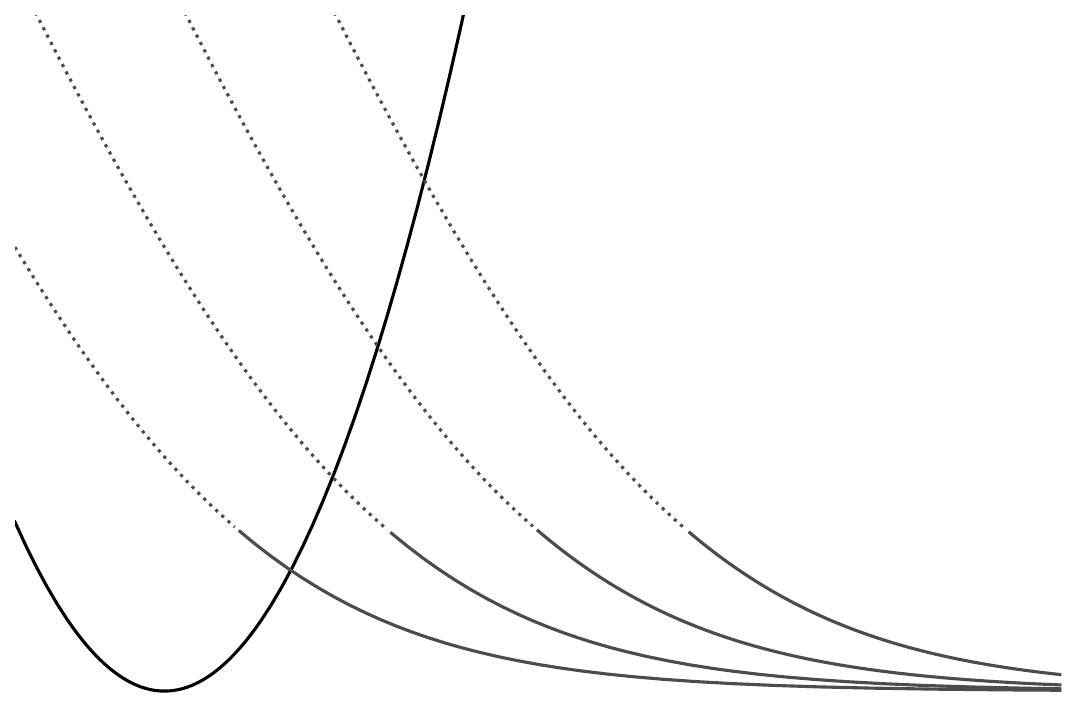}
        \put(22,4){$\theta_0$}
        \put(1,20){$\loss_{z_0}$}
        \put(30, 12.5){$\loss_{z_1}$}
        \put(43.5, 13){$\loss_{z_2}$}
        \put(57, 14){$\loss_{z_3}$}
        \put(70, 15){$\loss_{z_4}$}
      \end{overpic} &
      \begin{minipage}{.4\columnwidth}
        \vspace{-6cm}
        \caption{\label{fig:place-solutions-right}
          Choosing losses whose minima always are to the right.
          The loss $\loss_{z_0}$ has fixed positive derivative at
          $\theta_0$, while
          $\loss_{z_1}$, $\loss_{z_2}$, $\ldots$ cannot be minimized
          except by taking $\theta \gg \theta_0$.}
      \end{minipage}
    \end{tabular}
  \end{center}
\end{figure}

We leverage this lemma and the assumption~\eqref{eqn:no-infinite-one-d}
to place solutions, as in our outline in Section~\ref{sec:proof-outline}.
Assume without loss of generality that $\Theta$ is unbounded above,
and as Condition~\ref{cond:unbounded-case} fails,
we see that there exists $c > 0$ such that
for any sequence $\theta_n$, there exists $z_n \in \mc{Z}$ for which
\begin{equation}
  \label{eqn:big-ol-growth}
  \loss_{z_n}(\theta_n + n^2) - \loss_{z_n}\opt \ge c
\end{equation}
(where we have used Lemma~\ref{lemma:conditions-same-one-dim}).
Figure~\ref{fig:place-solutions-right} captures the basic idea: these
loss functions $\loss_{z_n}$ ``move right'' so that even for
$\theta_n \to \infty$, we always have $\loss_{z_n}(\theta_n) -
\loss_{z_n}\opt(\Theta) \ge c$, but the minimum
of the loss $\loss_{z_0}$ is static.
We then make it so testing between the case that the loss is
$\loss_{z_n}$ or incorporates a bit of $\loss_{z_0}$ is hard.

By the assumption that $\Theta$ is achievable,
there exists a constant $g > 0$ and $z_0$ and $\theta_0$
such that
\begin{equation*}
  \deriv_- \loss_{z_0}(\theta_0) \ge g.
\end{equation*}
Because $\theta
\mapsto \deriv_- \loss_z(\theta)$ is non-decreasing,
for all
$\theta_0 \le \theta_n \to \infty$
Lemma~\ref{lemma:small-derivatives-eventually}
shows that for any $\delta > 0$,
we may take $N$ large enough that
$\theta_n \ge N$ implies
\begin{equation*}
  \deriv_- \loss_{z_0}(\theta_n)
  \ge \deriv_- \loss_{z_0}(\theta_0) \ge g
  ~~ \mbox{and} ~~
  \inf_{\theta \ge \theta_n} \inf_{z \in \mc{Z}}
  \deriv_- \loss_z(\theta) \ge - \delta,
\end{equation*}
while the failure of Condition~\ref{cond:unbounded-case}
guarantees there exists $z_n \in \mc{Z}$ for which
$\deriv_- \loss_{z_n}(\theta_n) < 0$ and
inequality~\eqref{eqn:big-ol-growth} holds.

Define the probability distributions
\begin{equation*}
  P_0 \defeq \pointmass_{z_n}
  ~~ \mbox{and} ~~
  P_1 \defeq \left(1 - \frac{1}{n}\right) \pointmass_{z_n}
  + \frac{1}{n} \pointmass_{z_0}.
\end{equation*}
By inspection (and ignoring the measure-zero sets of non-differentiable
points, which are unimportant for this argument), we see that
\begin{align*}
  \poploss_{P_1}'(\theta) \ge -\left(1 - \frac{1}{n}\right)
  \delta + \frac{g}{n}
  \ge \frac{g}{2n}
  ~~ \mbox{for}~~ \theta \ge \theta_n,
\end{align*}
and so $\theta \ge \theta_n + n^2$ implies
\begin{equation*}
  \poploss_{P_1}(\theta)
  = \poploss_{P_1}(\theta_n)
  + \int_{\theta_n}^\theta \poploss_{P_1}'(t) dt
  \ge \poploss_{P_1}(\theta_n)
  + \int_0^{n^2} \frac{g}{2n} dt = \poploss_{P_1}(\theta_n) + \frac{g n}{2}.
\end{equation*}
Similarly, $\theta \le \theta_n + n^2$ implies
$\poploss_{P_0}(\theta) \ge \poploss_{P_0}\opt(\Theta) + c$
by assumption~\eqref{eqn:big-ol-growth}.
We therefore obtain
\begin{equation*}
  \dopt(\poploss_{P_0}, \poploss_{P_1}) \ge \min\left\{c,
  \frac{g n}{2}\right\}
  \ge c
\end{equation*}
for all large $n$.
The minimax bound~\eqref{eqn:minimax-to-testing}
thus implies
\begin{equation*}
  \minimaxlow(\loss, \mc{Z}, \Theta)
  \ge \frac{c}{2} \left(1 - \tvnorm{P_0^n - P_1^n}\right).
\end{equation*}
Applying Lemma~\ref{lemma:tv-and-hellinger} shows that $\tvnorm{P_0^n
  - P_1^n} \le (1 + o(1)) \sqrt{1 - e^{-2}}$, giving the theorem in
the one-dimensional case.

\subsection{Proof of Theorem~\ref{theorem:unbounded-case}:
  the $d$-dimensional case}
\label{sec:proof-general-unbounded-lower}

We first perform a reduction to simplify
our calculations
by modifying Condition~\ref{cond:unbounded-case} to a similar
condition that will turn out to be equivalent, but which leverages halfspaces
to more easily take advantage of convexity.
To state the condition, let
\begin{equation*}
  H_{v, t} \defeq \left\{\theta \mid \<\theta, v\> \le t \ltwo{v} \right\}
\end{equation*}
be the closed halfspace indexed by the direction $v / \ltwo{v}$
(where if $v = 0$ then obviously $H_{v, t} = \R^d$).
Consider the following
alternative version of Condition~\ref{cond:unbounded-case}:
\makeatletter
\renewcommand{\theourcondition}{%
  C.\@arabic{\numexpr\value{ourcondition}-1}'%
}
\makeatother
\conditionbox{
  \label{cond:recession-unbounded}
  For all $\epsilon > 0$, there exists a $t < \infty$ such that
  \begin{equation*}
    \poploss_P\opt(H_{v,t} \cap \Theta)
    - \poploss_P\opt(\Theta) \le \epsilon
  \end{equation*}
  for each direction $v$ and all $P \in \Pdiscrete(\mc{Z})$.
}
\makeatletter
\renewcommand{\theourcondition}{%
  C.\@arabic{\numexpr\value{ourcondition}-1}%
}
\makeatother

Because any compact convex set $\Theta_0$ coincides with
the intersection of all closed
halfspaces containing it, Condition~\ref{cond:recession-unbounded} is weaker
than Condition~\ref{cond:unbounded-case}.
Convexity, however, means that Condition~\ref{cond:recession-unbounded}
implies~\ref{cond:unbounded-case}, making them equivalent:
%
\begin{lemma}
  \label{lemma:halfspace-minimizers-to-ball}
  Let $\poploss : \R^d \to \Rup$ be closed convex, $t < \infty$, $\epsilon >
  0$, and assume that
  \begin{equation*}
    \poploss\opt(H_{v,t} \cap \Theta) < \poploss\opt(\Theta) + \epsilon
  \end{equation*}
  for each $v \in \R^d$.
  Then
  \begin{equation*}
    \poploss\opt(t \ball_2 \cap \Theta) \le \poploss\opt(\Theta) + \epsilon.
  \end{equation*}
\end{lemma}
\begin{proof}
  Recall the sub-optimality sets
  $\sublevel_\epsilon(\poploss)
  \defeq \{\theta \in \Theta \mid \poploss(\theta)
  \le \poploss\opt(\Theta) + \epsilon\}$,
  which are closed convex.
  The condition that $\poploss\opt(H_{v,t} \cap \Theta) <
  \poploss\opt(\Theta) + \epsilon$ implies that $H_{v,t} \cap \Theta \cap
  \sublevel_\epsilon \neq \emptyset$ for each $v$.
  The set $\Theta_t
  \defeq \Theta \cap_v H_{v,t}$ is compact
  convex.
  So if $\sublevel_\epsilon \cap \Theta_t = \emptyset$,
  there necessarily exists
  a hyperplane strictly separating $\Theta_t$ from $\sublevel_\epsilon$, meaning
  a vector $u \in \sphere^{d-1}$ for which
  \begin{equation*}
    \inf_{\theta \in \sublevel_\epsilon} \<u, \theta\> > \sup_{\theta \in \Theta_t}
    \<u, \theta\>.
  \end{equation*}
  But we know by assumption that $H_{u,t} \cap \Theta \cap \sublevel_\epsilon
  \neq \emptyset$, which contradicts this inequality.
\end{proof}

Thus, if Condition~\ref{cond:recession-unbounded}
holds, then
for each $\epsilon > 0$, there exists a finite $t$
such that the compact set $\Theta_t \defeq t \ball_2 \cap \Theta \subset \Theta$
satisfies $\inf_{\theta \in \Theta_t} [\poploss_P(\theta) - \poploss_P\opt(\Theta)]
\le \epsilon$ for all $P \in \Pdiscrete(\mc{Z})$.
As a consequence, if Condition~\ref{cond:unbounded-case} fails
(equivalently, \ref{cond:recession-unbounded} fails),
we have the half-space analogue of the condition~\eqref{eqn:big-ol-growth}
in the one-dimensional case:
there exists a $c > 0$
such that for all $t < \infty$, there is some $v \in \sphere^{d-1}$ and a $Q
\in \Pdiscrete(\mc{Z})$ such that
\begin{equation}
  \label{eqn:converse-unbounded-condition}
  \inf_{\theta \in H_{v, t} \cap \Theta}
  \poploss_Q(\theta) \ge \poploss_Q\opt(\Theta) + c.
\end{equation}
This characterization will be central to our lower bounds, and for the
remainder of the proof, we assume that Condition~\ref{cond:compact-case} holds
(because if it fails, Theorem~\ref{theorem:minimax-lower} gives the
result).
The remainder of the proof mimics the strategy we use in the one-dimensional
case.

\subsubsection{Eliminating infinite losses}

As in Section~\ref{sec:proof-general-unbounded-lower},
we first eliminate cases in which the losses $\loss_z$ may tend
to $-\infty$ in any way.
We first eliminate the case that
$\loss_z\opt(\Theta) = -\infty$ for some $z \in \mc{Z}$ as a triviality,
as in Lemma~\ref{lemma:no-infinite-one-d}.

\begin{lemma}
  \label{lemma:no-infinite-values}
  If there exists $z \in \mc{Z}$ such that $\loss_z\opt(\Theta) = -\infty$,
  then
  \begin{equation*}
    \minimaxlow_n(\loss, \mc{Z}, \Theta) = +\infty.
  \end{equation*}
\end{lemma}
\begin{proof}
  Let $z$ be such that $\loss_z\opt(\Theta) = -\infty$,
  $\theta_0 \in \interior \Theta$, and use the shorthand
  $\nabla \loss_z(\theta_0) \in \partial \loss_z(\theta_0)$.
  We now place gradients (recall Fig.~\ref{fig:plant-hyperplane})
  to guarantee that a particular mixture distribution has finite loss.
  For $v = \nabla \loss_z(\theta_0) / \ltwo{\nabla \loss_z(\theta_0)}$,
  Lemma~\ref{lemma:i-can-find-directable} guarantees that
  there is a finitely supported distribution $Q$ and $\beta > 0$
  for which
  a point $\theta_1$ of the form
  $\theta_1 \defeq \theta_0 + t v + v_\perp$ for some $t \ge 0$
  and $\<v, v_\perp\> = 0$
  satisfies
  \begin{equation*}
    -\beta \nabla \loss_z(\theta_0) \in \partial (\poploss_Q + \convexindic{\Theta})(\theta_1).
  \end{equation*}
  We then obtain
  \begin{align*}
    \poploss_Q(\theta) + \beta \loss_z(\theta)
    & \ge \poploss_Q(\theta_1) + \<-\beta \nabla \loss_z(\theta_0),
    \theta - \theta_1\>
    + \beta \loss_z(\theta_0)
    + \beta \<\nabla \loss_z(\theta_0), \theta - \theta_0\> \\
    & = \poploss_Q(\theta_1)
    + t \beta \ltwo{\nabla \loss_z(\theta_0)}
    + \beta \loss_z(\theta_0)
  \end{align*}
  by the choice of $v$.
  In particular,
  \begin{equation*}
    \inf_{\theta\in \Theta} \poploss_Q(\theta) + \beta \loss_z(\theta)
    \ge \poploss_Q(\theta_1) + \beta \loss_z(\theta_0)
    > -\infty,
  \end{equation*}
  while $\poploss_Q\opt(\Theta) < \infty$ and $\loss_z\opt(\Theta) = -\infty$,
  and for any compact subset $\Theta_0 \subset \Theta$, properness of
  $\loss_z$ guarantees $\loss_z\opt(\Theta_0) > -\infty$.
  Defining $p = \frac{1}{1 + \beta}$, the distributions
  \begin{equation*}
    P_0 = \pointmass_z
    ~~ \mbox{and} ~~
    P_1 = \frac{1}{\beta + 1} Q + \frac{\beta}{\beta + 1} \pointmass_z
    = p Q + (1 - p) \pointmass_z
  \end{equation*}
  yield losses for which, evidently,
  \begin{equation*}
    \dopt(\poploss_{P_0}, \poploss_{P_1})
    = +\infty.
  \end{equation*}
  As in the proof of Lemma~\ref{lemma:no-infinite-one-d},
  taking $\epsilon = \frac{\beta}{\beta + 1}$
  yields $\tvnorm{P_0 - P_1} \le 1 - \epsilon$,
  so Lemma~\ref{lemma:tv-and-hellinger} implies
  $\tvnorm{P_0^n - P_1^n} \le \sqrt{1 - \epsilon^{2n}} < 1$.
  Thus $\minimaxlow_n = +\infty$.
\end{proof}

\begin{lemma}
  \label{lemma:interior-theta-bounded-gaps}
  Let Condition~\ref{cond:compact-case}
  hold.
  If for some $\theta_0 \in \interior \Theta$ we have
  \begin{equation*}
    \sup_{z \in \mc{Z}}
    \left[\loss_z(\theta_0) - \loss_z\opt(\Theta)\right] = \infty,
  \end{equation*}
  then
  $\minimaxlow_n(\loss, \mc{Z}, \Theta) = +\infty$.
\end{lemma}
\begin{proof}
  Lemma~\ref{lemma:no-infinite-values} allows us to assume
  that $\loss_z\opt(\Theta) > -\infty$
  for each $z \in \mc{Z}$ (otherwise, $\minimaxlow_n = +\infty$
  in any case).
  By Condition~\ref{cond:compact-case},
  there exists $\lipconst_0 < \infty$ such that $\partial
  \loss_z(\theta_0) \subset \lipconst_0 \ball_2$ for each $z \in
  \mc{Z}$.
  Lemma~\ref{lemma:i-can-find-directable}
  thus guarantees that there exist constants $\alpha_0, \beta_0 > 0$
  such that
  for any $z \in \mc{Z}$ and
  any $\nabla \loss_z(\theta_0) \in \partial \loss_z(\theta_0)$,
  there is a
  finitely supported distribution $Q$ and
  point of the form
  $\theta_1 = \theta_0 + t v + v_\perp$
  for the
  direction $v = \nabla \loss_z(\theta_0) / \ltwo{\nabla \loss_z(\theta_0)}$
  and some $\alpha_0 \le t \le \frac{1}{\alpha_0}$
  for which
  \begin{equation*}
    -\beta \nabla \loss_z(\theta_0) \in \partial (\poploss_Q + \convexindic{\Theta})(\theta_1)
    ~~~ \mbox{for~a}~ \beta \ge \beta_0.
  \end{equation*}
  (In Figure~\ref{fig:plant-hyperplane},
  this corresponds to the choice $\theta_1 = \theta_0 + t v + v_\perp$
  for $v = \nabla \loss_z(\theta_0) / \ltwo{\nabla \loss_z(\theta_0)}$,
  though in this case $\nabla \loss_z(\theta_0)$ is bounded
  for all $z$, while the value $\loss_z(\theta_0) - \loss_z\opt$
  may be arbitrary.)
  
  Choose $p = \frac{1}{1 + \beta/2}$,
  so that $\beta p - (1 - p) = \frac{\beta}{2 + \beta}$, and so
  \begin{align*}
    \lefteqn{p \poploss_Q(\theta) + (1 - p) \loss_z(\theta)} \\
    & \stackrel{(i)}{\ge} p \left[\poploss_Q(\theta_1)
      + \<-\beta \nabla \loss_z(\theta_0), \theta - \theta_1\>
      \right]
    + (1 - p)
    \left[\loss_z(\theta_0) + \<\nabla \loss_z(\theta_0), \theta - \theta_0\>
      \right] \\
    & \stackrel{(ii)}{=}
    p \poploss_Q(\theta_1)
    + (1 - p) \loss_z(\theta_0)
    + p \<-\beta \nabla \loss_z(\theta_0), \theta - \theta_0\>
    + t p \beta \ltwo{\nabla \loss_z(\theta_0)}
    + (1 - p) \<\nabla \loss_z(\theta_0), \theta - \theta_0\> \\
    & = p \poploss_Q(\theta_1)
    + (1 - p) \loss_z(\theta_0)
    + t p \beta \ltwo{\nabla \loss_z(\theta_0)}
    - \frac{\beta}{2 + \beta}
    \<\nabla \loss_z(\theta_0), \theta - \theta_0\>.
  \end{align*}
  Here, inequality~$(i)$ follows from the first-order convexity condition,
  while inequality~$(ii)$ uses the definition of $\theta_1$.

  Define the two point distributions
  \begin{equation*}
    P_0 = \pointmass_z ~~ \mbox{and} ~~
    P_1 = (1 - p) \pointmass_z + p Q.
  \end{equation*}
  Letting $\gamma = \frac{\beta}{2 + \beta}$ for shorthand,
  noting that $\gamma \ge \frac{\beta_0}{2 + \beta_0}$
  regardless of the choice of $z$,
  we find that for any $\nabla \poploss_Q(\theta_0) \in \partial
  \poploss_Q(\theta_0)$,
  the first order condition for convexity implies
  \begin{align*}
    \poploss_{P_1}(\theta)
    & \ge
    p \poploss_{Q}(\theta_1)
    + (1 - p) \loss_z(\theta_0)
    - \gamma \<\nabla \loss_z(\theta_0), \theta - \theta_0\> \\
    & \ge \poploss_{P_1}(\theta_0)
    + p \<\nabla \poploss_Q(\theta_0), \theta_1 - \theta_0\>
    - \gamma \<\nabla \loss_z(\theta_0), \theta - \theta_0\> \\
    & \ge \poploss_{P_1}(\theta_0)
    - \frac{2p \lipconst_0}{\alpha_0}
    - \gamma \<\nabla \loss_z(\theta_0), \theta - \theta_0\>,
  \end{align*}
  where the final inequality follows via Cauchy-Schwarz.
  
  We now demonstrate the separation in optimization
  distance~\eqref{eqn:opt-distance}.
  By the preceding display, for any $c > 0$
  \begin{equation*}
    \poploss_{P_1}(\theta)
    \le \poploss_{P_1}(\theta_0) + c
    ~~ \mbox{implies} ~~
    - \gamma \<\nabla \loss_z(\theta_0), \theta - \theta_0\>
    \le c + \frac{2 \lipconst_0}{\alpha_0}.
  \end{equation*}
  Because the choice of $z \in \mc{Z}$ was arbitrary,
  we may assume that for any $K < \infty$,
  we choose $z$ so that
  $\loss_z(\theta_0) - \loss_z\opt(\Theta) \ge K$.
  Then
  \begin{equation*}
    \loss_z(\theta) \le \loss_z\opt + K/2
    ~~ \mbox{implies} ~~
    \loss_z(\theta) \le \loss_z(\theta_0) - K/2
    ~~ \mbox{so} ~~
    -\frac{K}{2} \ge \<\nabla \loss_z(\theta_0), \theta - \theta_0\>.
  \end{equation*}
  %
  %
  Setting $c = \frac{\gamma K}{2} - \frac{2 \lipconst_0}{\alpha_0}$
  then yields
  that for \emph{any} large enough (but finite) $K$, we can choose
  distributions $P_0$ and $P_1$ on $\mc{Z}$ for which
  \begin{equation*}
    \dopt(\poploss_{P_1}, \poploss_{P_0})
    \ge \min\left\{\gamma, 1\right\} \cdot \frac{K}{4}
    \ge \min\left\{\frac{\beta_0}{2 + \beta_0}, 1 \right\}
    \cdot \frac{K}{4}
  \end{equation*}
  while
  \begin{equation*}
    \tvnorm{P_0 - P_1}
    \le p = \frac{1}{1 + \beta/2} \le \frac{2}{2 + \beta_0}.
  \end{equation*}
  Applying Lemma~\ref{lemma:tv-and-hellinger},
  $\tvnorm{P_0^n - P_1^n}
  \le \sqrt{1 - (1 - p)^{2n}} < 1$.
  Take $K$ arbitrarily large.
\end{proof}

\subsubsection{Planting solutions}

We now turn to analogies of the arguments
in Section~\ref{sec:one-d-plant-unbounded} that allow us to
show separation in population losses by planting
solutions appropriately.
The
first step in this development
is to show that, so long as
$\inf_z \loss_z\opt(\Theta) > -\infty$,
all loss functions must eventually
grow away from any $\theta_0 \in \interior \Theta$,
as in Lemma~\ref{lemma:small-derivatives-eventually}.
To state the result, recall
the left derivative
$\deriv_- f(\theta; v) \defeq
\inf\left\{\<g, v\> \mid g \in \partial f(\theta)
\right\}$
from Section~\ref{sec:first-order-convex}.

\begin{lemma}
  \label{lemma:derivatives-grow-eventually}
  Let Condition~\ref{cond:compact-case} hold
  and assume $\inf_n \minimaxlow_n(\loss, \mc{Z}, \Theta) < \infty$.
  Then for
  any compact set $\Theta_0 \subset \interior \Theta$,
  \begin{equation*}
    \liminf_{t \to \infty} \inf_{\ltwo{\theta} \ge t,
      \theta \in \Theta}
    \inf_{z \in \mc{Z}}
    \inf_{\theta_0 \in \Theta_0}
    \deriv_- \loss_z\left(\theta;
    \frac{\theta - \theta_0}{\ltwo{\theta - \theta_0}}\right) \ge 0.
  \end{equation*}
\end{lemma}
\begin{proof}
  Because it is the supremum of closed convex functions, the function
  \begin{equation*}
    \wb{\loss}(\theta) \defeq \sup_{z \in \mc{Z}} [\loss_z(\theta)
      - \loss_z(\Theta)\opt]
  \end{equation*}
  is closed convex.
  By Lemma~\ref{lemma:interior-theta-bounded-gaps}, $\wb{\loss}(\theta_0) <
  \infty$ for any $\theta_0 \in \interior \Theta$, and as any convex function
  is continuous on the interior of its
  domain~\cite[Ch.~V]{HiriartUrrutyLe93}, the
  function $\wb{\loss}$ is evidently
  continuous on $\Theta_0$.
  It therefore attains its infimum and supremum on $\Theta_0$,
  whence $\sup_{\theta_0 \in K} \wb{\loss}(\theta_0) < \infty$.
  
  For the sake of contradiction, assume that the limit infimum in the
  lemma statement is $-2c$, where $c > 0$.
  Then there is some finite $t > 0$,
  $\theta \in \Theta$ satisfying
  $\dist(\theta, \Theta_0) \ge t$,
  and $\theta_0 \in \Theta_0$ and $z \in \mc{Z}$ for which
  \begin{equation*}
    \deriv_- \loss_z(\theta; v) \le -c,
  \end{equation*}
  where we defined
  $v = (\theta -
  \theta_0) / \ltwo{\theta - \theta_0}$ for shorthand.
  The monotonicity of (sub)gradients then implies
  \begin{equation*}
    \deriv_- \loss_z\left(u \theta + (1 - u) \theta_0; v
    \right)
    \le \deriv_- \loss_z(\theta; v) \le -c
    ~~ \mbox{for~} u \in [0, 1].
  \end{equation*}

  For any convex $f$ and $\theta_0 \in \interior \dom f$,
  we have the integral form~\citep[Thm.~VI.2.3.4]{HiriartUrrutyLe93}
  \begin{equation*}
    f(\theta) = f(\theta_0) + \int_0^1 \deriv f(\theta_0 + t
    (\theta - \theta_0); \theta - \theta_0) dt
    = f(\theta_0) + \int_0^1 \deriv_- f(\theta_0 + t(\theta - \theta_0);
    \theta - \theta_0) dt.
  \end{equation*}
  Applying this to $\loss_z$, we see
  $\loss_z(\theta) = \loss_z(\theta_0)
  + \int_0^1 \deriv_- \loss_z (\theta_0 + u(\theta - \theta_0); v)
  du \cdot \ltwo{\theta_0 - \theta}
  \le \loss_z(\theta_0) - c t$.
  For $t$ large enough, this contradicts that
  $\sup_{\theta_0 \in \Theta_0}
  \sup_z [\loss_z(\theta_0) - \loss_z\opt] < \infty$.
\end{proof}

Because integration and directional
differentiation commute for closed convex functions~\cite{Bertsekas73},
we may replace the infimum over $z \in \mc{Z}$ with
an infimum over $P \in \Pdiscrete(\mc{Z})$, so that
\begin{equation}
  \label{eqn:derivatives-grow-eventually}
  \liminf_{t \to \infty}
  \inf_{\ltwo{\theta} \ge t, \theta \in \Theta}
  \inf_{P \in \Pdiscrete(\mc{Z})}
  \inf_{\theta_0 \in \Theta_0}
  \deriv_- \poploss_P\left(\theta;
  \frac{\theta - \theta_0}{\ltwo{\theta - \theta_0}}\right)
  \ge 0
\end{equation}
for any compact $\Theta_0 \subset \interior \Theta$.

Now we provide an analogue of
Lemma~\ref{lemma:i-can-find-directable} that allows us to
plant solutions more carefully, constructing a finitely supported
distribution $\wb{Q}$ such that $\poploss_{\wb{Q}} + \convexindic{\Theta}$ 
is coercive and its minimum belongs to a compact set.
This construction mimics the distinct point $\theta_0$ in the
one-dimensional case (Fig.~\ref{fig:place-solutions-right} and $\loss_{z_0}$
there), but requires more care.
\begin{lemma}
  \label{lemma:grow-outside-box}
  Let $\theta_0 \in \interior \Theta$.
  There exists a finitely supported distribution $\wb{Q}$ and values $\gamma
  > 0$ and $b < \infty$ such that
  \begin{equation*}
    \deriv \poploss_{\wb{Q}}\left(\theta;
    \frac{\theta - \theta_0}{\ltwo{\theta - \theta_0}}\right)
    \ge \gamma
  \end{equation*}
  whenever $\theta \in \Theta$, $\ltwo{\theta - \theta_0} \ge b$.
\end{lemma}
\begin{proof}
  Let $\mc{E} = \{e_i, -e_i\}_{i = 1}^d$ be the collection of
  standard basis vectors and their negations.
  Use Lemma~\ref{lemma:i-can-find-directable}
  to find a convex compact set
  $K \subset \Theta$ such that
  for each $u \in \mc{E}$, there exists
  a finitely supported distribution $Q_u$ and
  scalar $\alpha_u > 0$ such that
  \begin{equation*}
    \alpha_u u \in \partial (\poploss_{Q_u} + \convexindic{\Theta})(\theta_u)
    ~~ \mbox{for~some~} \theta_u \in K.
  \end{equation*}
  Thus
  for each $u \in \mc{E}$,
  \begin{equation*}
    \poploss_{Q_u}(\theta)
    - \poploss_{Q_u}(\theta_u) \ge \alpha_u \<\theta - \theta_u, u\>
  \end{equation*}

  By Lemma~\ref{lemma:no-infinite-values}, we see that w.l.o.g.\ we may
  assume $\poploss_{Q_u}\opt(\Theta) = 0$, and so we also have
  $\poploss_{Q_u}(\theta) - \poploss_{Q_u}(\theta_u) \ge
  -\poploss_{Q_u}(\theta_u)$.
  We obtain
  \begin{equation*}
    \frac{1}{2d} \sum_{u \in \mc{E}} (\poploss_{Q_u}(\theta)
    - \poploss_{Q_u}(\theta_u))
    \ge \frac{1}{2d}
    \sum_{u \in \mc{E}} \left\{\alpha_u \<\theta - \theta_u, u\> \vee
    - \poploss_{Q_u}(\theta_u)\right\}.
  \end{equation*}
  Letting $\alpha\subopt = \min_{u \in \mc{E}} \alpha_u > 0$,
  we see that for $\linf{\theta}$ sufficiently large,
  \begin{equation*}
    \frac{1}{2d} \sum_{u \in \mc{E}} (\poploss_{Q_u}(\theta)
    - \poploss_{Q_u}(\theta_u))
    \ge \frac{1}{4d} \alpha\subopt
    \inf_{\theta_1 \in K} \linf{\theta - \theta_1}
    - \max_{u \in \mc{E}} \poploss_{Q_u}(\theta_u) > 0.
  \end{equation*}
  In particular,
  $\poploss_{\wb{Q}}$ grows at least linearly outside of
  a neighborhood of $K$
  for $\wb{Q} = \frac{1}{2d} \sum_{u \in \mc{E}} Q_u$.
  As
  $\dist(\theta_0, K) < \infty$,
  this linear growth
  also applies far from $\theta_0$,
  which implies the lemma.
\end{proof}

\subsubsection{The optimization separation and testing lower bound}

\providecommand{\directionvec}{v_0} 

With Lemma~\ref{lemma:grow-outside-box} in hand, we
can now construct hard instances extending the construction
in the one-dimensional case (Fig.~\ref{fig:place-solutions-right}) to the
full-dimensional setting.
Let $\wb{Q}$ be the finitely supported distribution
Lemma~\ref{lemma:grow-outside-box} promises with the attendant growth
constant $\gamma > 0$, and let $0 < \delta \le 1$
be a value to be chosen.
Fix $\theta_0 \in \interior \Theta$, and for $\theta \in \R^d$, define
the direction $\directionvec(\theta) \defeq (\theta - \theta_0) /
\ltwo{\theta - \theta_0}$.
Because Condition~\ref{cond:recession-unbounded}
fails,
we may take $t < \infty$ large enough that
it satisfies the following four conditions:
\begin{enumerate}[(i)]
\item \label{item:gap-out-at-z}
  There exists $v$ with $\ltwo{v} = 1$ and $Q \in \Pdiscrete(\mc{Z})$
  for which
  \begin{equation*}
    \poploss\opt_Q(H_{v,t} \cap \Theta) \ge \poploss\opt_Q(\Theta) + c.
  \end{equation*}
\item \label{item:mixture-big-deriv}
  For any $\theta \in \Theta$
  satisfying $\<v, \theta\> \ge \frac{t}{2}$,
  \begin{equation*}
    \deriv \poploss_{\wb{Q}}\left(\theta; \directionvec(\theta)\right)
    \ge \gamma.
  \end{equation*}
\item \label{item:z-small-deriv}
  For any $\theta \in \Theta$
  satisfying $\<v, \theta\> \ge \frac{t}{2}$,
  \begin{equation*}
    \deriv \poploss_P\left(\theta; \directionvec(\theta)\right)
    \ge - \frac{\delta \gamma}{2n}
    ~~ \mbox{for~all~} P \in \Pdiscrete(\mc{Z}).
  \end{equation*}
\item \label{item:t-is-big} $t \ge \frac{4 c n}{\gamma \delta}$.
\end{enumerate}

The conditions may appear to be circular, as
parts~\eqref{item:mixture-big-deriv} and~\eqref{item:z-small-deriv}
repose on the direction $v$ in part~\eqref{item:gap-out-at-z}, but
this is not truly an issue.
We may satisfy the requirement~\eqref{item:mixture-big-deriv}
via Lemma~\ref{lemma:grow-outside-box},
so long as $t$ is large, because $\<v, \theta\> \ge \frac{t}{2}$
certainly implies $\ltwo{\theta} \ge \frac{t}{2}$.
To satisfy the requirement~\eqref{item:z-small-deriv},
we similarly use Lemma~\ref{lemma:derivatives-grow-eventually}
(actually, the remark~\eqref{eqn:derivatives-grow-eventually} following
the lemma).
Satisfying requirement~\eqref{item:t-is-big} is trivial;
we may always simply take $t$ larger in the other parts,
and the failure of Condition~\ref{cond:recession-unbounded}
is sufficient to guarantee~\eqref{item:gap-out-at-z}.

With these four requirements satisfied,
consider the two distributions
\begin{equation*}
  P_0 \defeq Q
  ~~ \mbox{and} ~~
  P_1 \defeq \left(1 - \frac{\delta}{n}\right)
  Q + \frac{\delta}{n} \wb{Q}.
\end{equation*}
Let $\theta_1$ be any vector satisfying
$\<v, \theta_1\> \ge t$, and
let $\theta_{1/2}$ be the point at which the line
segment $[u \theta_1 + (1 - u) \theta_0]$ crosses the
hyperplane $\<v, \theta\> = t/2$, that is,
at $u = \frac{t/2 - \<v, \theta_0\>}{\<v, \theta_1 - \theta_0\>}$.
Then $0 < u < 1$ for large enough $t$, and the directions 
$\directionvec(\theta_1) = \directionvec(\theta_{1/2}) = 
\directionvec(u \theta_{1/2} + (1 - u) \theta_1)$ for all 
$u \in [0, 1]$.
Using the shorthand $\theta_u = u \theta_{1/2}
+ (1 - u) \theta_1$ and computing
directional derivatives,
we see that
$u \in [0, 1]$ implies
\begin{align*}
  \deriv \poploss_{P_1}(\theta_u; \directionvec(\theta_1))
  = \left(1 - \frac{\delta}{n}\right)
  \deriv\poploss_Q(\theta_u; \directionvec(\theta_1))
  + \frac{\delta}{n} \deriv \poploss_{\wb{Q}}(\theta_u; \directionvec(\theta_1))
  \ge \frac{\delta \gamma}{n} - \left(1 - \frac{\delta}{n}\right)
  \frac{\delta \gamma}{2n},
\end{align*}
where the inequality follows from the
assumptions~\eqref{item:mixture-big-deriv}
and~\eqref{item:z-small-deriv} on $t$.
Thus
\begin{align*}
  \poploss_{P_1}(\theta_1)
  & = \poploss_{P_1}(\theta_{1/2})
  + \ltwo{\theta_1 - \theta_{1/2}}
  \int_0^1 \deriv \poploss_{P_1}\left((1 - u) \theta_{1/2}
  + u \theta_1; \directionvec(\theta_1)\right) du \\
  & \ge \poploss_{P_1}(\theta_{1/2})
  + \ltwo{\theta_1 - \theta_{1/2}}
  \frac{\delta \gamma}{2n} \\
  & \ge \poploss_{P_1}(\theta_{1/2})
  + \frac{t}{2} \cdot \frac{\delta \gamma}{2n}
  \ge \poploss_{P_1}(\theta_{1/2})
  + c,
\end{align*}
where the final step uses the choice~\eqref{item:t-is-big} that $t \ge
\frac{4 c n}{\delta \gamma}$.

Performing a similar calculation, we see that if $\<v, \theta_1\> \le t$,
then
\begin{equation*}
  \poploss_{P_0}(\theta) = \poploss_Q(\theta) \ge \poploss_Q\opt(\Theta) + c
  = \poploss_{P_0}\opt(\Theta) + c
\end{equation*}
by assumption~\eqref{item:gap-out-at-z}.
Taking the contrapositive, we see that
if $\poploss_{P_1}(\theta) < \poploss_{P_1}\opt(\Theta) + c$, then
necessarily $\<\theta, v\> < t$, implying that
$\poploss_{P_0}(\theta) \ge \poploss_{P_0}\opt(\Theta) + c$,
while if $\poploss_{P_0}(\theta) < \poploss_{P_0}\opt(\Theta) + c$,
then $\<\theta, v\> \ge t$ and
$\poploss_{P_1}(\theta) \ge \poploss_{P_1}\opt(\Theta) + c$.
That is,
\begin{equation*}
  \dopt(\poploss_{P_0}, \poploss_{P_1}) \ge c.
\end{equation*}

Computing the variation distance between the probability
distributions is straightforward: we observe that
\begin{equation*}
  \tvnorm{P_0^n - P_1^n}
  \le \sqrt{1 - (1 - \delta/n)^n}
  \sqrt{1 + (1 - \delta/n)^n}
  \to \sqrt{1 - e^{-2\delta}}.
\end{equation*}
Thus, for any $\delta > 0$ and large
enough $n$, we may choose $P_0$ and $P_1$ so that
\begin{align*}
  \minimaxlow_n(\loss, \mc{Z}, \Theta)
  \ge \frac{\dopt(\poploss_{P_0}, \poploss_{P_1})}{4}
  \left(1 - \sqrt{1 - e^{-2 \delta}}\right)
  \ge \frac{c}{4}
  \left(1 - \sqrt{1 - e^{-2 \delta}}\right).
\end{align*}
Take $\delta \downarrow 0$ and recognize that
$\minimaxlow_n$ is non-increasing in $n$.

\subsection{Proof of Proposition~\ref{proposition:unbounded-upper}}
\label{sec:proof-unbounded-upper}


As a first step, we must show that
Condition~\ref{cond:unbounded-case}
extends so that inequality~\eqref{eqn:cannot-be-far-in-compact}
actually holds for all
$P \in \mc{P}_\loss(\mc{Z})$,
where $\mc{P}_\loss(\mc{Z})$ denotes the
well-defined probabilities~\eqref{eqn:well-defined}.
\begin{lemma}
  \label{lemma:discrete-to-borel}
  Let Condition~\ref{cond:unbounded-case} hold.
  Then for all $\epsilon > 0$, there is a compact
  $\Theta_0 \subset \Theta$ such that
  \begin{equation}
    \label{eqn:borel-gaps}
    \sup_{Q \in \mc{P}_\loss(\mc{Z})}
    \poploss_Q\opt(\Theta_0) - \poploss_Q\opt(\Theta)
    \le \epsilon.
  \end{equation}
\end{lemma}
\noindent
See Appendix~\ref{sec:proof-discrete-to-borel} for a proof
of this result.

Now let $\epsilon > 0$ and $\Theta_0 \subset \Theta$ be
a compact convex set for
which $\inf_{\theta \in \Theta_0} [\poploss_Q(\theta) - \poploss_Q\opt(\Theta)] \le
\epsilon$ for all $Q \in \mc{P}_\loss(\mc{Z})$,
as Lemma~\ref{lemma:discrete-to-borel} promises.
Then Theorem~\ref{theorem:minimizable}
shows that
there exists an estimator $\what{\theta}$
taking values in $\Theta_0$
and some $N < \infty$ such that $n \ge N$ implies
\begin{equation*}
  \sup_{P \in \mc{P}_\loss}
  \E_{P^n}
  \left[\poploss_P(\what{\theta}(Z_1^n))
    - \poploss_P\opt(\Theta_0)\right]
  \le \epsilon.
\end{equation*}
Of course, $\poploss_P\opt(\Theta_0) \le \poploss_P\opt(\Theta) + \epsilon$,
so $\minimax_n(\loss, \mc{Z}, \Theta) \le 2 \epsilon$.

\section{The infinite-dimensional case}
\label{section:infinite-dimensional}

It is natural to ask whether the results in this paper extend naturally to
infinite-dimensional scenarios, such as nonparametric regression
or density estimation problems~\cite{Stone82, GyorfiKoKrWa02, Tsybakov09}.
They do not:
in this section, we highlight that the fundamental limits of M-estimation
problems in finite and infinite dimensional settings differ significantly.
We will argue that there exist examples where standard analogues of our
Conditions~\ref{cond:compact-case} and~\ref{cond:unbounded-case} fail to
characterize solvability, and we show some of the types of additional
assumptions one might impose to develop lower bounds.
We currently cannot completely characterize the (worst-case) solvability of
infinite-dimensional convex problems, but the results suggest that such a
characterization would require new and different techniques.

\subsection{Bounded domains in infinite dimensions}

Let $\hilbert$ be a separable infinite dimensional Hilbert space with norm
$\ltwo{\cdot}$ and inner product $\<\cdot, \cdot\>$ and $\Theta \subset
\hilbert$.
Let $\ball_{\hilbert} = \{x \in \hilbert \mid \ltwo{x} \le 1\}$ be the unit
ball and $\sphere = \{x \in \hilbert \mid \ltwo{x} = 1\}$ the sphere;
each of these is closed and bounded, but not compact.
Nonetheless, a separable Hilbert space 
has the Lindel\"{o}f property that every open cover of a set
has a countable subcover.
By the Banach-Alaoglu theorem, bounded closed convex subsets of
$\hilbert$ are weakly compact.

The \emph{quasi-relative interior} of a set $\Theta$ has several equivalent
definitions~\cite{BorweinLe92, BauschkeCo11}.
Letting $\cone C = \{t x \mid x \in C, t \ge 0\}$ be the cone
generated by $C$, the quasi-relative interior of $C$ is
\begin{equation*}
  \qri C \defeq \left\{x \in C \mid
  \cl \cone (C - x) = \cl \linspan (C - x) \right\},
\end{equation*}
that is, for which $\cl \cone (C - x)$ is a subspace.
As $\hilbert$ is separable, then any closed convex set $C$ has
$C = \cl \qri C$, and in particular, if $C$ is non-empty then $\qri C$
is non-empty (see~\cite[Thm.~2.19]{BorweinLe92} or~\cite[Fact
6.14]{BauschkeCo11}).
For simplicity we shall assume that for each $z \in \mc{Z}$, at each
$\theta_0 \in \qri \Theta$, the convex function $\loss_z$ satisfies
$\theta_0 \in \interior \dom \loss_z$, that is, $\loss_z$ is
continuous at $\theta_0$ and subdifferentiable
there~\cite[Prop.~16.17]{BauschkeCo11}.

A natural but false extension of Condition~\ref{cond:compact-case}
would be the following.

\newtheorem{falseconjecture}[theorem]{False conjecture}

\begin{falseconjecture}
  \label{conjecture:compact-case-inf}
  Let all $\theta\in \qri \Theta$ be achievable~\eqref{eqn:positive-derivative} 
  and $\Theta$ be closed and bounded in $\hilbert$. Then 
  \begin{equation*}
    \lim_n \minimax_n(\loss, \mc{Z}, \Theta) = 0
  \end{equation*}
  if and only if for each closed $\Theta_0 \subset \qri \Theta$, there
  exists a finite $\lipconst = \lipconst(\Theta_0)$ such that for each
  $z\in\mc{Z}$ the function $\loss_z$ is $\lipconst$-Lipschitz continuous on
  $\Theta_0$.
\end{falseconjecture}
\noindent We provide an example falsifying this conjecture in
the next section.
In Section~\ref{subsec:lin-lb-inf}, we show we can derive an 
infinite-dimensional lower bound under a related condition, but only in a 
restricted setting.

\subsection{Solvability with unbounded gradients}

Conjecture~\ref{conjecture:compact-case-inf} fails because we can find
examples where estimators provide vanishing excess risk
even when the conditions in the conjecture fail.
We illustrate one such example below.
Let $\hilbert=\ell_2$ and $\{e_i\}_{i=1}^{\infty}$ be the canonical
orthonormal basis, so that
$\theta_k\defeq \langle \theta, e_k\rangle$.
We also let $a_k\downarrow 0$ and $b_k\uparrow \infty$ be positive
sequences such that $\sum_k a_k^2 < \infty$ and define
\begin{align*}
  \Theta_{\text{box}}
  &\defeq
  \left\{
    \theta\in \ell_2: |\theta_k| \leq a_k
  \right\}.
\end{align*}
For a fixed base point $\theta_0\in\Theta_{\text{box}}$, define the
star-shrunken set
\begin{equation*}
  \Theta_\delta\defeq \delta\theta_0+(1-\delta)\Theta_{\text{box}},
\end{equation*}
as in the proof of Theorem~\ref{theorem:minimizable}.
We now choose $0<c<b_1$, and define the linear
losses $\loss_z(\theta) = \<z, \theta\>$
over the index set
\begin{equation*}
  \mc{Z}=c\ball_\hilbert \cup \{b_ke_k, k\geq 1\}.
\end{equation*}
Evidently, the points $z \in c\ball_\hilbert$ make it so that every
$\theta\in \qri\Theta$ is achievable~\eqref{eqn:positive-derivative}, even
with a uniform margin $c>0$, while the remaining terms ensure that for any
$\theta\in \qri\Theta$
\begin{equation*}
  \sup_{z\in \mc{Z}}\sup_{g\in \partial \loss_z(\theta)}\norm{g}
  =
  \sup_{z\in \mc{Z}} \norm{z}
  =
  \sup_{k\geq 1}b_k
  =
  \infty,
\end{equation*}
so the gradients are unbounded in the ambient norm.
Even more, there are cases in which the directional derivatives
\begin{equation*}
  \sup_{z\in \mc{Z}}\sup_{\theta\in \Theta_\delta}
  \langle \theta-\theta_0, \nabla \loss_z(\theta)\rangle=+\infty,
\end{equation*}
and the problem is still solvable for any $P\in \mc{P}_\loss(\mc{Z})$.

\begin{proposition}
  \label{proposition:solve-regimes}
  For the triplet $(\Theta_{\text{box}},\loss,\mc{Z})$ if $a_k=2^{-k}$
  and $b_k=k2^{k}$ then there exists
  $\theta_0\in \qri\Theta_{\text{box}}$ such that for all
  $\delta\leq \frac{1}{2}$
  \begin{equation*}
    \sup_{z\in \mc{Z}}\sup_{\theta\in \Theta_\delta}
    \langle \theta-\theta_0, \nabla \loss_z(\theta)\rangle=+\infty.
  \end{equation*}
  Nonetheless, $\minimax_n(\loss, \mc{Z},  \Theta) \to 0$ as $n \to \infty$.
\end{proposition}
\noindent
See Section~\ref{subsec:proof-solve-regimes} for a proof.
%

\subsection{Lower bounds for linear functions}
\label{subsec:lin-lb-inf}

We can construct lower bounds for some infinite-dimensional problems
in restricted settings.
Assume that $\loss_z(\theta) = \<z, \theta\>$ and that $\Theta$ is a
closed and norm-bounded symmetric convex set, while $\mc{Z}$ is also a
symmetric set with elements $z \in \ell_2(\N)$ or $z \in \R^d$ for
some finite $d$.
Let
$\mc{P}_\loss(\mc{Z}) = \{P \mid \inf_{\theta \in \Theta} \E_P[\<\theta, Z\>] > -\infty\}$
be the set of distributions for which
$\poploss_P\opt(\Theta) > -\infty$, so that minimization over $\Theta$
is well-defined.
When $\Theta$ is norm-bounded and $\mc{Z} \subset \ell_2(\N)$, then
$\inf_{\theta \in \Theta} \<\theta, z\> \ge -\sup_{\theta \in \Theta} \ltwo{\theta} \ltwo{z} > -\infty$,
so that $\mc{P}_\ell(\mc{Z})$ includes all finitely supported
distributions.
\begin{proposition}
  Let the conditions above hold and let
  $\<\Theta, \mc{Z}\> \defeq \sup_{\theta \in \Theta, z \in \mc{Z}} \<\theta, z\>$
  be the ``linear diameter'' of $\Theta \times \mc{Z}$.
  For linear losses,
  \begin{align*}
    \minimax_n
    \defeq \inf_{\what{\theta}_n}
    \sup_{P \in \mc{P}_\loss(\mc{Z})}
    \E_{P^n}\left[\poploss_P(\what{\theta}_n)
      - \poploss_P\opt(\Theta)\right]
    \gtrsim \frac{1}{\sqrt{n}} \<\Theta, \mc{Z}\>.
  \end{align*}
\end{proposition}
\noindent This shows the importance of the asymmetry condition in
Proposition~\ref{proposition:solve-regimes}.

\begin{proof}
  First assume that
  $D = \<\Theta, \mc{Z}\> \defeq \sup_{\theta \in \Theta} \sup_{z \in \mc{Z}} \<z, \theta\> < \infty$,
  so that $\mc{P}_\loss$ consists of all Borel distributions on
  $\mc{Z}$.
  Let $\theta_0, z_0$ be points satisfying $\<\theta_0, z_0\> = D$ (or
  at least arbitrarily close to that).
  Then the points $-\theta_0 \in \Theta$ and $-z_0 \in \mc{Z}$ by
  symmetry, and so the distributions
  \begin{align*}
    P_{\pm 1} = \frac{1 \pm \delta}{2} \pointmass_{z_0}
    + \frac{1 \mp \delta}{2} \pointmass_{-z_0}
  \end{align*}
  satisfy (for small $\delta > 0$)
  \begin{align*}
    \dkl{P_1}{P_{-1}} = O(1) \delta^2
  \end{align*}
  while
  \begin{align*}
    \poploss_{P_1}(\theta) = \delta \<\theta, z_0\>
    ~~ \mbox{and} ~~
    \poploss_{P_{-1}}(\theta) = -\delta \<\theta, z_0\>.
  \end{align*}
  So we see if $\poploss_{P_1}(\theta) \le 0$ then
  $\poploss_{P_{-1}}(\theta) \ge 0$ and vice versa, so that
  $\dopt(\poploss_{P_1}, \poploss_{P_{-1}}) \gtrsim \delta D$.
  Take $\delta \asymp \frac{1}{\sqrt{n}}$ and apply Le Cam's two-point
  method.

  When $D = +\infty$, the lower bound requires an approximation
  argument.
  Because $\Theta$ is closed and norm-bounded, for any
  $z \in \ell_2(\N)$, we have
  $\inf_{\theta \in \Theta} \<z, \theta\> > -\infty$.
  Thus there exists a sequence $z_n \in \mc{Z}$ with
  $D_n \defeq \sup_{\theta \in \Theta} \<\theta, z_n\> < \infty$ and
  $\inf_{\theta \in \Theta} \<\theta, z_n\> \to -\infty$.
  Take $z_n$ in this sequence and observe that a two-point lower bound
  gives $\minimax_n \gtrsim \frac{1}{\sqrt{n}} D_n$.
  Since $D = +\infty$, for any fixed $n$ and any $C>0$ we can choose
  $z_n$ such that $D_n\geq C\sqrt{n}$.
  Letting $C\uparrow \infty$ yields the stated result.
\end{proof}

\subsection{Proof of Proposition~\ref{proposition:solve-regimes}}
\label{subsec:proof-solve-regimes}

Letting our base point be $\theta_0=0\in \Theta_{\text{box}}$ we will
first show that the inner product is unbounded.
For each $k\geq 1$ we take $x^{(k)}=a_ke_k\in \Theta_{\text{box}}$ and
\begin{equation*}
  \theta^{(k)}=\theta_0+(1-\delta)x^{(k)}=(1-\delta)x^{(k)}\in \Theta_\delta.
\end{equation*}
We now consider losses associated with the atoms $z=b_ke_k$ at these
$\theta^{(k)}$ where
\begin{equation*}
  \langle \theta^{(k)}-\theta_0, \nabla \loss_{z}(\theta^{(k)})\rangle
  =
  (1-\delta)\langle a_ke_k, b_ke_k\rangle
  =
  (1-\delta)k,
\end{equation*}
and taking $k\uparrow \infty$
\begin{equation*}
  \sup_{z\in \mc{Z}}\sup_{\theta\in \Theta_\delta}
  \langle \theta-\theta_0, \nabla \loss_z(\theta)\rangle=+\infty.
\end{equation*}

We will now show that there is an alternative estimator
$\hat\theta_{n}$ that works.
For a distribution $P\in\mc{P}_\loss(\mc{Z})$ let
\begin{align*}
  p_k&\defeq P(z=b_ke_k)
  &
  \mu&\defeq \E[z \mathbf{1}\{z\in c\ball_{\hilbert}\}]
  &
  s_k&= p_kb_k + \mu_k.
\end{align*}
The distribution is now in $\mc{P}_\loss(\mc{Z})$ if
$\E_P[|\loss_{Z}(\theta)|]<\infty$ for
$\theta\in \qri \Theta_{\text{box}}$ and
$\E_P[[-\loss_{Z}(\theta)]_+]<\infty$ for
$\theta\in \Theta_{\text{box}}$, or equivalently
$ \sum b_k a_k p_k < \infty$ since the contribution of the other
indices is bounded.
For any such $P$
\begin{equation*}
  \poploss_{P}(\theta)=\sum s_k\theta_k.
\end{equation*}
This objective is separable and the optimal $\theta^\star$ satisfies
\begin{equation*}
  \theta^\star_k = -a_k\sign(s_k).
\end{equation*}
We now propose to estimate $s_k$ directly via
\begin{align*}
  \hat{p}_k&\defeq \frac{1}{n}\sum_{i=1}^n \mathbf{1}\{z_i=b_ke_k\}
  &
  \hat{\mu}&\defeq \frac{1}{n}\sum_{i=1}^nz_i \mathbf{1}\{z_i\in c\ball_{\hilbert}\}
  &
  \hat{s}_k&\defeq \hat{p}_kb_k + \hat{\mu}_k,
\end{align*}
and construct the estimator truncated at level $m(n)\uparrow \infty$
\begin{equation*}
  \langle \hat\theta_{n}, e_k\rangle =
  \begin{cases}
    -a_k\sign(\hat{p}_kb_k + \hat{\mu}_k) & \text{if } k\leq m(n)\\
    -a_k & \text{otherwise.}
  \end{cases}
\end{equation*}
For $k\leq m(n)$ if $\sign(s_k)\neq \sign(\hat{s}_k)$ then
$|s_k|\leq |s_k-\hat{s}_k|$ and
\begin{equation*}
  s_k\hat{\theta}_{k}-s_k\theta_{k}^\star
  \leq
  2a_k|s_k|
  \leq
  2a_k|s_k-\hat{s}_k|.
\end{equation*}
Similarly, for $k> m(n)$ if $\sign(s_k)=-1$ then $|s_k|\leq c$, where
this $c$ comes from our definition of $\mc{Z}$, and
\begin{equation*}
  s_k\hat{\theta}_{k}-s_k\theta_{k}^\star\leq 2a_kc.
\end{equation*}
Therefore, for any $P\in \mc{P}_\loss(\mc{Z})$ we have
\begin{equation*}
  \poploss_P(\hat\theta_{n})-\poploss^\star_P(\Theta)
  \leq
  2\sum_{k\leq m(n)}a_k|s_k-\hat{s}_k|
  +
  2c\sum_{k > m(n)}a_k.
\end{equation*}
We now know that
\begin{equation*}
  \E[|s_k-\hat{s}_k|]\leq \frac{b_k+c}{\sqrt{n}},
\end{equation*}
and thus
\begin{equation*}
  \E[\poploss_P(\hat\theta_{n})-\poploss^\star_P(\Theta)]
  \leq
  \frac{2}{\sqrt{n}}\sum_{k\leq m(n)}a_kb_k
  +
  \frac{2c}{\sqrt{n}}
  +
  c2^{-m(n)+1}.
\end{equation*}
We now use the fact that $a_kb_k=k$ to conclude that
\begin{equation*}
  \E[\poploss_P(\hat\theta_{n})-\poploss^\star_P(\Theta)]
  \leq
  \frac{2m(n)^{2}}{\sqrt{n}}
  +
  \frac{2c}{\sqrt{n}}
  +
  c2^{-m(n)+1},
\end{equation*}
choosing $m(n)\uparrow \infty$ slowly enough that
$\frac{m(n)^{2}}{\sqrt{n}}\to 0$ then yields
\begin{equation*}
  \lim_{n\to\infty}\sup_{P\in\mc{P}_\loss(\mc{Z})}
  \E[\poploss_{P}(\hat{\theta}_{n})-\poploss_{P}^\star(\Theta)]=0.
\end{equation*}


\section{Attaining stationary points}
\label{sec:stationary-points}

In many M-estimation problems, instead of minimizing the loss $\poploss_P$,
we instead seek (near) stationary points
of $\poploss_P$ over $\Theta = \R$.
This has many motivations, including in non-convex
problems~\cite{Nesterov12b, LeePaPiSiJoRe19, ArjevaniCaDuFoSrWo23, BubeckMi24},
where attaining strong minimization guarantees is impossible, as well as
recent work that connects ``gradient equilibrium'' to various desiderata in
prediction problems~\cite{AngelopoulosJoTi25}, including online sequence
calibration~\cite{FosterVo98, FosterHa21} and conformal
prediction~\cite{GibbsCa21,AngelopoulosCaTi23}.
The simplest motivating example comes by revisiting
Example~\ref{example:quantile-estimation} on quantile estimation, where
Theorem~\ref{theorem:unbounded-case} shows that minimizing the loss without
any assumptions is impossible.


\begin{example}[Quantile estimation, Example~\ref{example:quantile-estimation}
    revisited]
  \label{example:revisit-quantile}
  For the losses $\loss_\alpha(\theta, z)
  = \alpha \hinge{\theta - z} + (1 - \alpha) \hinge{z - \theta}$,
  we see that if
  \begin{equation*}
    -\epsilon \le \deriv_- \poploss_P(\theta)
    ~~ \mbox{and} ~~
    \deriv_+ \poploss_P(\theta) \le \epsilon,
  \end{equation*}
  then evidently
  \begin{equation*}
    1 - \alpha - \epsilon \le P(Z < \theta)
    ~~ \mbox{and} ~~
    P(Z \le \theta) \le 1 - \alpha + \epsilon,
  \end{equation*}
  so that finding a nearly stationary point---one with
  small left and right derivatives---implies an accurate
  quantile estimate.
  As we shall see in the sequel, any empirical
  minimizer $\what{\theta}_n = \argmin_\theta \poploss_{P_n}(\theta)$
  (i.e., empirical quantile) satisfies
  $\hinge{\max\{-\deriv_+ \poploss_P(\what{\theta}_n),
  \deriv_- \poploss_P(\what{\theta}_n)\}} \le O(1) / \sqrt{n}$ with
  high probability,
  so that
  \begin{equation*}
    P(Z_{n + 1} \le \what{\theta}_n) \ge 1 - \alpha - \frac{O(1)}{\sqrt{n}}
    ~~ \mbox{and} ~~
    P(Z_{n + 1} < \what{\theta}_n) \le 1 - \alpha + \frac{O(1)}{\sqrt{n}},
  \end{equation*}
  where $Z_{n + 1}$ is an independent
  draw from $P$.
  This contrasts strongly with the impossibility result
  for loss \emph{minimization} in Example~\ref{example:quantile-estimation}.
\end{example}

In this section, we expand this example to
consider 1-dimensional Lipschitz convex losses.
We view these results as essentially preliminaries and
(hopefully) starting points for future work.
We will not provide fundamental lower bounds, instead treating these results
as evidence (see also \citeauthor{AngelopoulosJoTi25}'s
discussion~\citep{AngelopoulosJoTi25}) that there are significant
differences between finding (near) stationary points and even asymptotically
minimizing convex losses.

\subsection{The definition of stationarity}

When the space $\Theta$ is unbounded, Condition~\ref{cond:unbounded-case}
and Theorem~\ref{theorem:unbounded-case} show
that ``typically''
$\inf_n \minimaxlow_n(\loss, \mc{Z}, \Theta) > 0$, meaning that
distribution free estimation is impossible. 
Accordingly,
we modify the goals to seek \emph{stationary
points} of $\poploss_P$.
If $\poploss$ were differentiable and had minimizers, a natural
metric for stationarity would be the
derivative magnitude
\begin{equation*}
  \errorstationary_\poploss(\theta) \defeq |\poploss'(\theta)|.
\end{equation*}
This is not quite satisfactory, because (as in the quantile
example~\ref{example:revisit-quantile}) we frequently seek
stationary points of non-differentiable losses, and additionally, it is
possible that $\poploss$ has no (finite) minimizer.
Accordingly, we give an expanded definition for error,
which coincides when $\poploss$ is differentiable and has attained minimizers,
but allows moving beyond these cases.
As a first attempt to address non-differentiability, we might
mimic Example~\ref{example:revisit-quantile} to define
\begin{equation*}
  \errorstationary_\poploss(\theta) \defeq
  \hinge{\max\left\{
    - \deriv_+\poploss(\theta),
    \deriv_-\poploss(\theta)\right\}}.
\end{equation*}
For example, when $\poploss(\theta) = |\theta|$, then $\theta_0 = \theta_1 =
0$, and $\deriv_+ \poploss(0) = 1$ while $\deriv_-\poploss(0) = -1$, and
$\errorstationary_\poploss(\theta) = 0$ if and only if $\theta = 0$; more
generally, if $\theta\opt$ minimizes $\poploss$, then we know that $\deriv_-
\poploss(\theta\opt) \le 0 \le \deriv_+ \poploss(\theta\opt)$ and
$\errorstationary_\poploss(\theta\opt) = 0$.
When minimizers of $\poploss$ may not
exist, we incorporate a correction subtracting off the minimal
possible subgradient magnitude, defining
\begin{equation}
  \label{eqn:stationary-error}
    \errorstationary_\poploss(\theta)
    \defeq \hinge{
      \max\{-\deriv_+ \poploss(\theta),
      \deriv_- \poploss(\theta)\}}
    - \inf_\theta \inf_{g \in \partial \poploss(\theta)}
    \{|g|\}.
\end{equation}
In the case that $\poploss$ is differentiable and has a minimizer, we simply
have $\errorstationary_\poploss(\theta) = |\poploss'(\theta)|$.
Figure~\ref{fig:stationary-def} shows the help of the correction
term~\eqref{eqn:stationary-error}.

\begin{figure}[ht]
  \begin{center}
    \begin{tabular}{cc}
      \hspace{-.5cm}
      \begin{minipage}{.4\columnwidth}
        \vspace{-3.5cm}
        \caption{\label{fig:stationary-def}
          The function $\poploss$ decreases
          out to $\infty$, but points to the right
          are nearly stationary, approaching
          the asymptotically shallowest slope
          $\inf_\theta |\poploss'(\theta)|$.}
      \end{minipage}
      \hspace{-.5cm}
      &
      \begin{overpic}[width=.59\columnwidth]{
          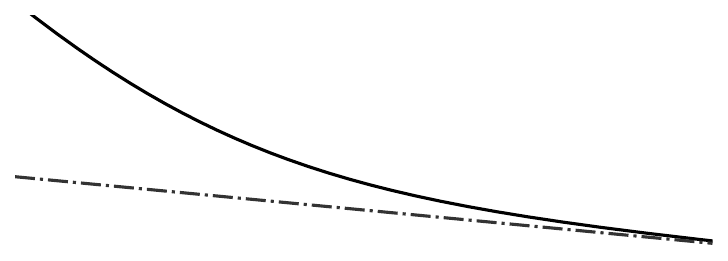}
        \put(20,24.5){$\poploss$}
      \end{overpic}
    \end{tabular}
  \end{center}
\end{figure}

A few properties of $\errorstationary_\poploss$ show
that it provides a natural measure of non-stationarity.
\begin{lemma}
  Let $\poploss : \R \to \Rup$ be closed
  convex.
  The stationarity error~\eqref{eqn:stationary-error} satisfies
  \begin{enumerate}[(i)]
  \item \label{item:stationary-is-stationary}
    If $0 \in \partial \poploss(\theta)$, then
    $\errorstationary_\poploss(\theta) = 0$.
  \item \label{item:nonstationary-is-nonstationary}
    If $0 \not \in \partial \poploss(\theta)$
    and $\argmin \poploss$ is non-empty,
    then $\errorstationary_\poploss(\theta) > 0$.
  \item If $\argmin \poploss$ is empty, then
    either $\deriv_+ \poploss(\theta) < 0$ and
    $\deriv_- \poploss(\theta) < 0$ for all
    $\theta$ or $\deriv_- \poploss(\theta) > 0$
    and $\deriv_+ \poploss(\theta) > 0$ for all
    $\theta$. 
    In either case,
    the choice $s = -\sign(\deriv_- \poploss(\theta))$ yields
    \begin{equation*}
      \lim_{\theta \to s \cdot \infty} \errorstationary_\poploss(\theta) = 0.
    \end{equation*}
  \end{enumerate}
\end{lemma}
\begin{proof}
  Part~\eqref{item:stationary-is-stationary} is immediate.
  %
  %
  For claim~\eqref{item:nonstationary-is-nonstationary}, note that
  $\partial \poploss(\theta) = [a, b]$, where $a$ and $b$ satisfy
  either $-\infty < a \le b < \infty$, $a = -\infty$, or $b =
  +\infty$.
  Considering each of the cases, it becomes apparent that we need only
  consider the case that $a > 0$ (as the $b < 0$ case is similar), so
  w.l.o.g.\ we assume $a > 0$, yielding $\deriv_- \poploss(\theta) = \inf
  \partial \poploss(\theta) > 0$, so $\errorstationary_\poploss(\theta) >
  0$.

  For the final claim, note that if $\argmin \poploss = \emptyset$,
  then necessarily $\poploss$ is strictly monotone (as otherwise, it
  would have a stationary point and hence a minimizer), yielding
  the one-sidedness of $\deriv_+ \poploss$ or $\deriv_- \poploss$.
  Because
  $\deriv_+ \poploss$ and $\deriv_-
  \poploss$ are monotone non-decreasing,
  then
  considering the case that $\deriv_+ \poploss < 0$,
  the ordering in part~\eqref{item:continuity-directionals}
  of Lemma~\ref{lemma:properties-of-derivatives} gives
  that $\deriv_- \poploss < 0$ as well, and
  \begin{equation*}
    \lim_{\theta \to \infty} \deriv_+ \poploss(\theta)
    = \lim_{\theta \to \infty} \deriv_- \poploss(\theta)
    = \sup_\theta \deriv_+ \poploss(\theta)
    = \sup_{g,\theta} \{g \in \partial \poploss(\theta)\}
    \le 0.
  \end{equation*}
  That is, $\lim_{\theta \to \infty} \errorstationary_\poploss(\theta)
  = \lim_{\theta \to \infty} -\deriv_+ \poploss(\theta)
  - \lim_{\theta \to \infty} (-\deriv_+ \poploss(\theta)) = 0$.
\end{proof}

\subsection{Achieving stationarity}

When the loss $\loss$ is always (globally) Lipschitz,
empirical risk minimizers are sufficient to achieve
stationarity according to the measure~\eqref{eqn:stationary-error},
even when losses are non-differentiable.

\begin{proposition}
  \label{proposition:stationary-lipschitz}
  Assume that $\loss_z(\cdot)$ is $\lipconst$-Lipschitz for all $z \in
  \mc{Z}$. Then any empirical minimizer $\what{\theta} \in \argmin_\theta
  \poploss_{P_n}(\theta)$ satisfies
  \begin{equation*}
    \P\left(\errorstationary_{\poploss_P}(\what{\theta}) > t \right)
    \le 2 \exp\left(-\frac{n t^2}{2 \lipconst^2}\right)
    ~~ \mbox{for~} t \ge 0.
  \end{equation*}
\end{proposition}

As an immediate consequence of
Proposition~\ref{proposition:stationary-lipschitz}, we see that
\begin{equation*}
  \E\left[\errorstationary_P(\what{\theta}_n)\right]
  \le 2 \int_0^\infty \exp\left(-\frac{n t^2}{2 \lipconst^2}\right) dt
  = \frac{2 \lipconst}{\sqrt{n}}
  \int_0^\infty e^{-u^2 / 2} du
  = \frac{\sqrt{2\pi} \lipconst}{\sqrt{n}}.
\end{equation*}
Let us consider this in the context of
Example~\ref{example:revisit-quantile}:
for the quantile loss,
\begin{equation*}
  \errorstationary_{\poploss_P}(\theta)
  = \hinge{\max\left\{(1 - \alpha) - P(Z \le \theta),
  P(Z < \theta) - (1 - \alpha)\right\}}.
\end{equation*}
So with probability at least $1 - 2 \exp(-t^2 / 2)$,
\begin{equation*}
  (1 - \alpha) - \frac{t}{\sqrt{n}}
  \le P(Z_{n + 1} \le \what{\theta}_n)
  ~~ \mbox{and} ~~
  P(Z_{n + 1} < \what{\theta}_n)
  \le (1 - \alpha) + \frac{t}{\sqrt{n}}.
\end{equation*}
In this particular case, the quantile satisfies a slightly
sharper guarantee~\cite[Proposition 2]{Duchi25},
but Proposition~\ref{proposition:stationary-lipschitz}
holds for general Lipschitz functions.

\subsection{Proof of Proposition~\ref{proposition:stationary-lipschitz}}
\label{sec:proof-stationary-lipschitz}

Let $\poploss = \poploss_P$ for shorthand when convenient.
We first give the argument assuming that $\argmin \poploss
= [\theta_0, \theta_1]$ is non-empty, then extend it to the case
that it is empty.
Fixing $t > 0$, consider the event that
$\errorstationary_\poploss(\what{\theta}) > t$.
For simplicity and without loss of generality,
let us assume that $\deriv_- \poploss (\what{\theta}) > t$.
Define
\begin{equation*}
  \theta_t \defeq \inf\left\{\theta \mid \deriv_- \poploss(\theta)
  \ge t \right\}
  = \sup\left\{\theta \mid \deriv_- \poploss(\theta) < t \right\},
\end{equation*}
where the equality follows from
Lemma~\ref{lemma:properties-of-derivatives},
part~\eqref{item:continuity-directionals}, as $\theta < \theta_t$ implies
$\deriv_- \poploss(\theta) < t$ and $\deriv_-\poploss(\theta_t) =
\lim_{\theta \uparrow \theta_t} \deriv_-\poploss(\theta) =
\lim_{\theta \uparrow \theta_t} \deriv_+ \poploss(\theta) \le t$.
We also observe that $\deriv_+ \poploss(\theta_t) =
\lim_{\theta \downarrow \theta_t} \deriv_+ \poploss(\theta) \ge t$,
again by Lemma~\ref{lemma:properties-of-derivatives}.

Combining the observation that
$\deriv_- \poploss(\what{\theta}) > t$ and
$\deriv_- \poploss(\theta_t) \le t$ imply
$\what{\theta} > \theta_t$
with the preceding derivations, we see that
$\what{\theta} \in \argmin \poploss_{P_n}$ implies
$\deriv_+ \poploss_{P_n}(\theta_t)
\le \deriv_- \poploss_{P_n}(\what{\theta}) \le 0$,
so
\begin{align*}
  \P\left(\deriv_- \poploss(\what{\theta}) > t\right)
  \le \P\left(\deriv_+ \poploss_{P_n}(\theta_t) \le 0
  \right)
  & = \P\left(\deriv_+ \poploss_{P_n}(\theta_t)
  - \deriv_+ \poploss_P(\theta_t) \le
  - \deriv_+ \poploss_P(\theta_t)
  \right) \\
  & \le \P\left(
  \deriv_+ \poploss_{P_n}(\theta_t)
  - \deriv_+ \poploss_P(\theta_t) \le
  - t \right)
  \le \exp\left(-\frac{n t^2}{2 \lipconst^2}\right),
\end{align*}
because $\deriv_+ \poploss_{P_n}(\theta)
= \frac{1}{n} \sum_{i = 1}^n \deriv_+ \loss_{Z_i}(\theta)$
(cf.~\cite{Bertsekas73}), where we have applied
a Hoeffding bound.
A completely parallel calculation gives that
\begin{equation*}
  \P\left(\deriv_+ \poploss(\what{\theta}) < -t\right)
  \le \exp\left(-\frac{n t^2}{2 \lipconst^2}\right)
\end{equation*}
for any $t > 0$.

Lastly, we address the case that $\argmin \poploss = \emptyset$.
Without loss of generality, assume that $\deriv_+ \poploss > 0$, so that
$\poploss$ is increasing, and $\deriv_- \poploss(-\infty) = \lim_{\theta
  \to -\infty} \deriv_- \poploss(\theta)$ satisfies $\deriv_-
\poploss(-\infty) = \inf_{\theta,g} \{g \in \partial
\poploss(\theta)\} \ge 0$
and $\errorstationary_\poploss(\theta)
= \deriv_- \poploss(\theta) - \deriv_- \poploss(-\infty)$.
If $\what{\theta} = -\infty$, then
$\errorstationary_\poploss(\what{\theta}) = 0$, so
we assume now that $\what{\theta}$ is finite.
Consider the event that $\errorstationary_\poploss(\what{\theta}) > t$,
i.e., $\deriv_- \poploss(\what{\theta}) > \deriv_-\poploss(-\infty) + t$.
As above, defining
\begin{equation*}
  \theta_t = \inf\left\{\theta \mid \deriv_- \poploss(\theta)
  - \deriv_- \poploss(-\infty) \ge t \right\},
\end{equation*}
we have $\deriv_+ \poploss(\theta_t) \ge t + \deriv_- \poploss(-\infty)
\ge t$,
and
$\what{\theta} > \theta_t$.
Then $\deriv_+ \poploss_{P_n}(\theta_t)
\le 
\deriv_- \poploss_{P_n}(\what{\theta}) \le 0$,
while
\begin{align*}
  \P\left(\deriv_- \poploss(\what{\theta}) >
  \deriv_- \poploss(-\infty) + t\right)
  & \le \P\left(\deriv_+ \poploss_{P_n}(\theta_t) \le 0\right) \\
  & = \P\left(\deriv_+ \poploss_{P_n}(\theta_t)
  - \deriv_+ \poploss_P(\theta_t)
  \le -\deriv_+ \poploss_P(\theta_t)\right)
  \le \exp\left(-\frac{n t^2}{2 \lipconst^2}\right).
\end{align*}


\section{Discussion}
\label{sec:discussion}

This paper has developed the conditions that provide a concrete line between
those cases in which distribution-free minimization of convex losses is
asymptotically possible and when, conversely, it is not.
Here, we situate some of the results relative to the statistical and
machine learning literatures while enumerating several open questions.

In classical (stochastic) optimization and machine learning, uniform
convergence guarantees~\cite{DevroyeGyLu96, Wainwright19, BartlettMe02} of
the form
\begin{equation*}
  \sup_{\theta \in \Theta}
  \left|\poploss_{P_n}(\theta) - \poploss_P(\theta)\right| \to 0
  ~~~
  \mbox{(in probability or otherwise)}
\end{equation*}
provide a frequent tool for
demonstrating the convergence of the empirical risk minimizer $\what{\theta}_n
= \argmin_{\theta \in \Theta} \poploss_{P_n}(\theta)$.
As we mention in the introduction, they are sometimes equivalent to the
convergence $\minimax_n \to 0$ of the minimax risk~\cite{AlonBeCeHa97,
  Vapnik98, AnthonyBa99, ShalevShSrSr10}.

Condition~\ref{cond:compact-case} shows that even when $\Theta$ is compact,
these are unnecessary for estimation, and we can even have $\sup_{\theta \in
  \Theta} |\poploss_{P_n}(\theta) - \poploss_P(\theta)| = \infty$ with
probability 1 while $\poploss_P(\what{\theta}_n) \to \poploss_P\opt$.
Indeed, in Example~\ref{example:log-loss}, let $p = P(Z = 1)$.
Then whenever $\what{p} = \frac{1}{n} \sum_{i = 1}^n \indic{Z_i = 1}
\neq p$,
\begin{equation*}
  \sup_{\theta \in [0, 1]}
  |\poploss_{P_n}(\theta) - \poploss_P(\theta)|
  =
  \sup_{\theta \in [0, 1]}
  \left|(\what{p}_n - p) \log \frac{\theta}{1 - \theta}
  \right|
  = \infty.
\end{equation*}
If $p$ is irrational this occurs at all $n$, while
Theorem~\ref{theorem:minimizable} shows that a restricted empirical risk
minimizer $\what{\theta}_n \in \argmin_{\theta \in \Theta_n}
\poploss_{P_n}(\theta)$, where the sets $\Theta_n \uparrow \Theta$ are
particularly chosen, achieves vanishing excess risk.
As stated in Proposition~\ref{proposition:uc-vs-c1} uniform convergence
on compact subsets of the domain is related to 
Condition~\ref{cond:compact-case}, but not equivalent unless we use 
the anchored loss formulation.
In the unbounded case, Example~\ref{example:exp-loss} makes clear that no
form of global Lipschitz continuity (and similarly, no uniform convergence)
is necessary for convergence of an estimator.
Nonetheless, it would be interesting to understand whether the existence of
stable estimators (say, that change little upon substituting a single
observation $Z_i$) is necessary and sufficient for $\loss$ to be
minimizable, as is the case for bounded functions in machine learning
problems~\cite{ShalevShSrSr10}.
It also remains unclear whether Condition~\ref{cond:compact-case} is
equivalent to the consistency of particular types of regularization.

Section~\ref{sec:stationary-points} implicitly raises the question of
whether our choice of minimax risk~\eqref{eqn:minimax-opt} via the excess
loss $\poploss_P(\theta) - \poploss_P\opt$ is the right one.
In some problems, stationary points may be more
interesting (and achievable) than loss minimizers.
In another vein, our characterizations are loss specific but do not address
questions of adapting to the underlying distribution $P$
(cf.~\cite{ChatterjeeDuLaZh16}): are instance-specific
guarantees, i.e., depending on $\loss$, $P$, and $\Theta$, possible?
Relatedly, if instead of the minimax risk~\eqref{eqn:minimax-opt}
we consider a game where we choose an estimator $\what{\theta}_n$
for each $n$, then nature chooses $P$ and we observe $Z_i \simiid P$,
we ask whether for some $\gamma \ge 0$ the risk measure
\begin{equation}
  \label{eqn:super-efficiency-maybe}
  \inf_{\what{\theta}}
  \sup_P \limsup_n n^\gamma \cdot
  \E_{P^n}\left[\poploss_P(\what{\theta}_n(Z_1^n))
    - \poploss_P\opt\right]
\end{equation}
converges.
Super-efficiency theory~\cite{VanDerVaart97} shows that
estimators may improve upon
classical statistical efficiency bounds only
at a measure-zero set of parameters
to provide insights into the limit~\eqref{eqn:super-efficiency-maybe}
in these cases;
extending it to general loss minimization problems would be a major
achievement.

We could instead ask for
convergence of $\what{\theta}_n$ to the solution
set $\Theta\opt(P) \defeq \argmin_{\theta \in \Theta} \poploss_P(\theta)$.
When $\Theta$ is compact, the
empirical risk minimizer $\what{\theta}_n$ satisfies $\dist(\what{\theta}_n,
\Theta\opt(P)) \cas 0$ for any $P$, but this convergence is
pointwise~\cite[Thm.~5.4]{ShapiroDeRu14}.
Even for Lipschitz functions on a compact set,
uniform guarantees are impossible:
consider minimizing $\loss_z(\theta) = |\theta - z|$,
and take the two distributions $P_0 = \frac{1 + \delta}{2}
\pointmass_0 + \frac{1 - \delta}{2} \pointmass_1$
and $P_1 = \frac{1 - \delta}{2} \pointmass_0 +
\frac{1 + \delta}{2} \pointmass_1$.
Then because
$\theta_0 = \argmin \poploss_{P_0} = 0$ and
$\theta_1 = \argmin \poploss_{P_1} = 1$,
an essentially standard minimax calculation~\cite[Ch.~14]{Wainwright19}
gives
\begin{equation*}
  \inf_{\what{\theta}_n}
  \max_{i \in \{0, 1\}}
  \E_{P_i^n} \left[\left|\what{\theta}_n(Z_1^n)
    - \theta_i\right|\right]
  \ge \half \left(1 - \tvnorm{P_0^n - P_1^n}\right)
  \ge \half \left(1 - \sqrt{\delta^2 (1 + o_\delta(1)) n}\right),
\end{equation*}
where the final inequality uses Pinsker's inequality
and $\dkl{P_0^n}{P_1^n} = n \dkl{P_0}{P_1}$
while $\dkl{P_0}{P_1} = \delta \log \frac{1 + \delta}{1 - \delta}
= 2 \delta^2 (1 + o(1))$ as $\delta \to 0$.
Take $\delta \ll 1/n$ to obtain
that the minimax parameter estimation
risk is as bad as predicting $\what{\theta} = 0$:
$\max_{i \in \{0, 1\}} \E_{P_i}[|\what{\theta}_n - \theta_i|]
\ge \half - o(1)$.
Any characterization of estimable parameters in a distribution free sense
requires care, as certainly Condition~\ref{cond:compact-case} is
insufficient.

We identify one more important open question.
First, we include Section~\ref{sec:stationary-points} only to contrast the
(in)achievability results present in the rest of the paper.
Providing a characterization of the losses for which
it is possible to attain stationary points in a
distribution free sense remains challenging.
Certainly, uniform Lipschitz continuity of the losses (as
Proposition~\ref{proposition:stationary-lipschitz} assumes) is
unnecessary: if $\loss_z(\cdot)$
has Lipschitz derivatives on compact subsets of $\Theta$,
then any estimator with
$\poploss_P(\what{\theta}_n) - \poploss_P\opt \to 0$ guarantees
that $\poploss_P'(\what{\theta}_n) \to 0$.
Indeed, if $\poploss_P(\theta) - \poploss_P\opt \le \epsilon$, then
because a function with $\lipconst$-Lipschitz derivative satisfies
$|\poploss_P(\theta + t) - \poploss_P(\theta)
- \poploss_P'(\theta) t| \le \lipconst t^2 / 2$,
taking $t = -\poploss_P'(\theta) / M$ yields
\begin{equation*}
  \poploss_P(\theta + t) \le \poploss_P(\theta)
  - \frac{\poploss'(\theta)^2}{2 \lipconst}.
\end{equation*}
That is, necessarily $\poploss'(\theta) \le \sqrt{2 \lipconst \epsilon}$.
So, for example, the exponential loss in Example~\ref{example:exp-loss}
admits estimators satisfying $\poploss_P'(\what{\theta}_n) \to 0$,
though it is certainly not Lipschitz.
We leave such extensions to future work.

\bibliography{bib}
\bibliographystyle{abbrvnat}

\appendix


\section{Properties of the set of achievable minimizers}
\label{sec:achievable-minimizers}

With the definitions of directional derivatives and associated calculus
rules in Section~\ref{sec:convex-analysis}, we can provide more perspective
on the achievable set~\eqref{eqn:achievable-minimizers}.

\subsection{The one-dimensional achievable set}

The cleanest results follow in the one-dimensional case, as the achievable
set~\eqref{eqn:one-dim-achievable-minimizers} admits many different
characterizations.
To do so, we first argue that so long as $\Theta$ has an interior,
the extreme points~\eqref{eqn:extreme-minimizers} satisfy
\begin{equation}
  \label{eqn:alt-extreme-minimizers}
  \thetamin
  = \inf\left\{\theta \in \Theta
  \mid \sup_{z \in \mc{Z}} \deriv_+ \loss_z(\theta) > 0
  \right\}
  ~~ \mbox{and} ~~
  \thetamax
  = \sup\left\{\theta \in \Theta
    \mid \inf_{z \in \mc{Z}}
    \deriv_- \loss_z(\theta) < 0
    \right\}.
\end{equation}
We show the left equality, as the right is similar.
Let $\theta \in \interior \Theta$ be any point for which
$\sup_{z \in \mc{Z}} \deriv_+ \loss_z(\theta) > 0$.
Then there exists $z \in \mc{Z}$ for which $\deriv_+ \loss_z(\theta)
> 0$, and because $\loss_z(\cdot)$ is a.e.\ differentiable on its
domain, for all $\epsilon > 0$ there exists $\theta' \in [\theta,
  \theta + \epsilon]$ for which $\loss_z'(\theta') \ge \deriv_+
\loss_z(\theta)$ by Lemmas~\ref{lemma:properties-of-derivatives}
and~\ref{lemma:integrability-convex},
so the claimed equality holds if
there exists $\theta \in \interior \Theta$ satisfying
$\deriv_+ \loss_z(\theta) > 0$ for some $z \in \mc{Z}$.
If, on the other hand, $\sup_{z \in \mc{Z}}\deriv_+ \loss_z(\theta)
\le 0$ for all $\theta \in \interior \Theta$, then the monotonicity of
$\deriv_+$ and $\loss'$ show that $\thetamin = \sup\{\theta \in
\Theta\}$ as desired.
We therefore use the characterization~\eqref{eqn:alt-extreme-minimizers}
for $\thetamin$ and $\thetamax$.

While the definition~\eqref{eqn:achievable-minimizers} of
$\achievable$ appears complex, the next two lemmas show how it
captures the sets of plausible minimizers of $\poploss_P(\theta)$ in
ways that depend only on the loss $\loss$ and potential data $\mc{Z}$,
irrespective of $P$.
In the lemmas, we recall that $\mc{P}_\loss$ denotes the
set~\eqref{eqn:well-defined} of Borel probabilities for which
$\poploss_P$ is well defined on $\theta \in \Theta$.
All finitely supported measures $P$ immediately belong to $\mc{P}_\loss$,
and moreover, minimizers of $\poploss_P$ in $\Theta$ are attained.
If $\Theta$ is compact, this is standard.
Otherwise, the extended reals address the issue, so that (e.g.) if $\sup
\Theta = +\infty$ and $\lim_{\theta \to \infty} \poploss(\theta) =
\inf_{\theta \in \Theta} \poploss(\theta)$, we say $\infty \in
\argmin_{\theta \in \Theta} \poploss(\theta)$.

\begin{lemma}
  \label{lemma:achievable-minimizers}
  Simultaneously for each probability distribution
  $P \in \mc{P}_\loss$, that is, for which $\poploss_P$ is well-defined
  on $\Theta$,
  the following hold:
  \begin{enumerate}[(i)]
  \item \label{item:singleton-achievable}
    If $\thetamin \ge \thetamax$, then
    $\thetamax \le \theta\opt \le \thetamin$ implies
    that $\theta\opt$ minimizes $\poploss_P(\theta)$ over $\theta \in \Theta$.
    In particular,
    all functions $\loss_z(\cdot)$ are constant and minimized
    on $[\thetamax, \thetamin]$.
  \end{enumerate}
  If $\achievable$ is non-empty, we additionally have
  \begin{enumerate}[(i)]
    \setcounter{enumi}{1}
  \item \label{item:minimizer-intersection}
    $\achievable$ intersects the
    minimizers of $\poploss_P(\theta)$, that is, $\argmin_{\theta \in
      \Theta} \E_P[\loss_Z(\theta)] \cap \achievable \neq \emptyset$.
  \item \label{item:project-back}
    For any point $\theta_0 \in \Theta$, the projection
    $\proj_{\achievable}(\theta_0) = \argmin_{\theta \in \achievable}\{(\theta
    - \theta_0)^2\}$ satisfies $\poploss_P(\proj_{\achievable}(\theta_0)) \le
    \poploss_P(\theta_0)$.
  \item \label{item:old-definition-of-achievable}
    If $\achievable$ is non-singleton, then
    \begin{equation}
      \achievable
      = \cl \left\{\theta \in \Theta \mid
      ~ \mbox{there~exist}~ z_0, z_1 \in \mc{Z}
      ~ \mbox{with}~
      \begin{array}{ll}
        \partial \loss_{z_0}(\theta) \cap \R_{> 0} \neq \emptyset & \mbox{and} \\
        \partial \loss_{z_1}(\theta) \cap \R_{< 0} \neq \emptyset
      \end{array}
      \right\}.
      \label{eqn:old-achievable-minimizers}
    \end{equation}
  \end{enumerate}
\end{lemma}
\noindent
We prove the lemma in Appendix~\ref{sec:proof-achievable}.

\newcommand{\allargmins}{\mc{A}}

We can write $\achievable$ more directly in terms of minimizers of
$\poploss_P$.
Let $\allargmins(\loss, \mc{Z}, \Theta)$ be the collection
of closed convex sets $C \subset [-\infty,\infty]$ for which for all
distributions $P$ on $\mc{Z}$,
\begin{equation*}
  C \cap \argmin_{\theta \in \Theta} \poploss_P(\theta) \neq \emptyset.
\end{equation*}
We then have the following result, whose proof we provide in
Appendix~\ref{sec:proof-achievable-as-all-argmins}, and which makes formal
that the definitions~\eqref{eqn:one-dim-achievable-minimizers}
and~\eqref{eqn:achievable-minimizers} are equivalent.
\begin{lemma}
  \label{lemma:achievable-as-all-argmins}
  Assume $\achievable$ is non-empty.
  Then viewed as a subset of the extended
  reals $[-\infty, \infty]$, it satisfies
  \begin{equation*}
    \achievable =
    \bigcap \left\{C \in \allargmins(\loss, \mc{Z}, \Theta)\right\}.
  \end{equation*}
\end{lemma}

That is, $\achievable$ is the smallest closed convex set $C$ satisfying $C
\cap \argmin_{\theta \in \Theta}\poploss_P(\theta) \neq \emptyset$ for all
$P$.  In summary, Lemmas~\ref{lemma:achievable-minimizers}
and~\ref{lemma:achievable-as-all-argmins} show that it is
no loss of generality to restrict attention to the
set~\eqref{eqn:achievable-minimizers} of achievable minimizers: completely
independently of the distribution $P$ on $\mc{Z}$, we can \emph{always}
improve the expected loss $\poploss_P$ by projecting a point onto
$\achievable$.

\subsection{The multi-dimensional achievable set}
\label{sec:multi-dim-achievable}

When $\Theta \subset \R^d$ for $d > 1$, the characterizations
of the achievable set~\eqref{eqn:achievable-minimizers}
are not quite so elegant as those for the one-dimensional case,
so instead we simply record a few results and extensions
to the set.
As a trivial remark, note that if $\theta_0 \not \in \achievable$, then for
each finitely supported distribution $Q$, there is $\theta \in \achievable$
for which $\poploss_Q(\theta) \le \poploss_Q(\theta_0)$, so in that sense,
there is no reason to include points in $\Theta \setminus \achievable$
in estimation.

As a preliminary result, we note that the condition that the halfspaces are
unconstraining~\eqref{eqn:halfspace-unconstraining} extends
to any probability measure so long as the loss
$\loss_z$ is appropriately integrable.
\begin{lemma}
  \label{lemma:discrete-or-borel}
  Let $\mc{P}_\loss$ contain the well-defined probability
  measures~\eqref{eqn:well-defined} and $\Theta$ be compact.
  Then equality~\eqref{eqn:halfspace-unconstraining} holds
  for all $P \in \mc{P}_\loss$.
\end{lemma}
\noindent
See Appendix~\ref{sec:proof-discrete-or-borel} for a proof.
In brief, we see that the restriction to finitely supported measures
is of no real import.

\newcommand{\eliminable}{E}

We also provide a few equivalent characterizations of achievability,
assuming the density of achievable points~\eqref{eqn:positive-derivative}.
For a vector $\theta_0$, define the set
of \emph{eliminatable points}
\begin{equation}
  \label{eqn:eliminable}
  \begin{split}
    E(\theta_0, \loss)
    & \defeq
    \left\{\theta \mid \deriv \loss_z(\theta_0; \theta - \theta_0)
    \ge 0 ~ \mbox{for~all~} z \in \mc{Z} \right\}\\
    & = \bigcap_{z \in \mc{Z}}
    \left\{\theta
    \mid \deriv \loss_z(\theta_0; \theta - \theta_0) \ge 0
    \right\},
  \end{split}
\end{equation}
which is a cone of points that necessarily
have loss at least that of $\theta_0$, as
\begin{equation*}
  \loss_z(\theta) \ge \loss_z(\theta_0) + \deriv
  \loss_z(\theta_0; \theta - \theta_0).
\end{equation*}
If a point $\theta \in E(\theta_0, \loss)$ for some $\theta_0 \neq \theta$,
then \emph{a priori} $\theta$ is sub-optimal (relative to $\theta_0$) for
any probability distribution $P$ on $\mc{Z}$, and $\poploss_P(\theta_0) \le
\poploss_P(\theta)$.
In fact, that a point $\theta_0$ eliminates no points in $\Theta$ is
equivalent to its achievability (see
Appendix~\ref{sec:proof-positive-derivative-eliminable}):
\begin{lemma}
  \label{lemma:positive-derivative-eliminable}
  If $\Theta \cap E(\theta_0, \loss) = \{\theta_0\}$ for
  some $\theta_0 \in \interior \Theta$, then
  inequality~\eqref{eqn:positive-derivative} holds at
  $\theta_0$.
  Conversely, if inequality~\eqref{eqn:positive-derivative} holds at
  $\theta \in \interior \Theta$, then
  $\theta \not \in E(\theta_0, \loss)$ for all $\theta_0 \in \interior \Theta$
  with $\theta_0 \neq \theta$.  
\end{lemma}
\noindent
As a consequence, if each $\theta \in \interior \Theta$ is
admissible~\eqref{eqn:positive-derivative}, then no points may be summarily
eliminated as sub-optimal from $\Theta$,
and $\achievable = \Theta$.

\newcommand{\dominating}{\mathsf{domin}_\loss}

An alternative to the characterizations above is to define
the set
\begin{equation*}
  \dominating(\theta_0) \defeq \left\{
  \theta \in \Theta \mid \loss_z(\theta) \le \loss_z(\theta_0)
  ~ \mbox{for~all~} z \in \mc{Z} \right\}
  = \bigcap_{z \in \mc{Z}}
  \left\{\theta \in \Theta \mid \loss_z(\theta) \le \loss_z(\theta_0)
  \right\}
\end{equation*}
of points that uniformly dominate $\theta_0$, meaning that
$\theta \in \dominating(\theta_0)$ achieves smaller loss on \emph{all}
examples $z \in \mc{Z}$.
The set $\dominating(\theta_0)$ is closed convex,
as it is the intersection of closed convex sets.
When the only point $\theta$ that uniformly dominates
$\theta_0$ is always $\theta_0$ itself, then
it turns out that all points in the interior of $\Theta$
are admissible~\eqref{eqn:positive-derivative}:
\begin{lemma}
  \label{lemma:dominating-and-achievable}
  Assume that $\dominating(\theta_0) = \{\theta_0\}$ for each $\theta_0 \in
  \interior \Theta$.
  Then for each $\theta \in \interior \Theta$
  and $v \in \sphere^{d-1}$, there exists $z \in \mc{Z}$ such that
  \begin{equation*}
    \deriv \loss_z(\theta; v)
    \ge \deriv_- \loss_z(\theta; v) > 0.
  \end{equation*}
  Conversely, if $\theta_0 \in \interior \Theta$ is
  achievable~\eqref{eqn:positive-derivative}, then $\dominating(\theta_0) =
  \{\theta_0\}$.
\end{lemma}
\noindent
See Appendix~\ref{sec:proof-dominating-and-achievable} for a proof.

Recalling Example~\ref{example:glms}, we see what we view as the
prototypical case:
the set of admissible points~\eqref{eqn:positive-derivative}
is dense in $\interior \Theta$.
When this occurs, the somewhat complicated seeming achievable
set~\eqref{eqn:achievable-minimizers} coincides with the closure of any of
the following sets:
\begin{enumerate}[(i)]
\item the admissible points
  $\theta \in \Theta$, i.e., those
  satisfying inequality~\eqref{eqn:positive-derivative}.
\item the points $\theta$ that are not
  eliminatable~\eqref{eqn:eliminable}, that is, $\theta \not \in E(\theta_0,
  \loss)$ for any $\theta_0 \in \interior \Theta$.
\item the set of un-dominated points $\theta \in \Theta$,
  i.e., those for which $\dominating(\theta) = \{\theta\}$.
\end{enumerate}

\section{TECHNICAL PROOFS}

\subsection{Proof of Lemma~\ref{lemma:positive-derivatives-directable}}
\label{sec:proof-positive-derivatives-directable}

We first perform a compactness argument when $\theta \in \interior \Theta$
that refines the positivity condition in 
Definition~\ref{definition:achievable-point}.
Because $\partial \loss_z(\theta)$ is compact for $\theta \in \interior
\Theta$, $v \mapsto \deriv \loss_z(\theta; v)$ is convex and bounded on
compacts, and so is continuous.
Thus, for each $v$, for the $z$ satisfying 
Definition~\ref{definition:achievable-point},
there is $\epsilon = \epsilon(v) > 0$ such that
\begin{equation*}
  \inf_{\ltwo{u - v} < \epsilon,
    \ltwo{u} = 1}
  \deriv \loss_z(\theta; u) > 0.
\end{equation*}
The sets $\{v + \epsilon(v) \interior \ball_2\} \cap \sphere^{d-1}$ form
an open cover of $\sphere^{d-1}$, from
which we can abstract a finite subcover,
and indexing by $v_1, \ldots, v_K$ and the associated
$z_1, \ldots, z_K$, we obtain
\begin{equation*}
  \inf_{v \in \sphere^{d-1}}
  \max_{i \le K} \deriv \loss_{z_i}(\theta; v)
  = \min_{j \le K}
  \inf_{\ltwo{u - v_j} < \epsilon}
  \max_{i \le K} \deriv \loss_{z_i}(\theta; u)
  \ge \min_{j \le K}
  \max_{i \le K}
  \inf_{\ltwo{u - v_j} < \epsilon} \deriv \loss_{z_i}(\theta; u)
  > 0.
\end{equation*}
Here $u$ satisfies $\ltwo{u}=1$ in both inner infima.
Written differently, we have the following
strengthening of Definition~\ref{definition:achievable-point}.
\begin{lemma}
  \label{lemma:strong-positive-derivative}
  Let $\theta \in \interior \Theta$
  satisfy  Definition~\ref{definition:achievable-point}.
  Then there exists a finite $k \in \N$, set $\{z_1, \ldots, z_k\}
  \subset \mc{Z}$, and $c = c(\theta) > 0$ such that for each $v \in
  \sphere^{d-1}$, there is $z_i$ for which
  \begin{equation}
    \label{eqn:strong-positive-derivative}
    \deriv \loss_{z_i} (\theta; v) \ge c.
  \end{equation}
\end{lemma}

The next result provides a theorem of alternatives,
in the vein of Farkas' lemma.
\begin{lemma}
  \label{lemma:bilinear-alternatives}
  Let $A \in \R^{m \times n}$ be a matrix with
  rows $a_1^T, \ldots, a_m^T$ and $u \in \R^n$ a
  vector. Then precisely one of the following two alternatives holds:
  \begin{enumerate}[(a)]
  \item \label{item:solve-convex-system}
    There exists
    $\lambda \in \R^m_+$, $\<\lambda, \ones\> = 1$ for which
    $A^T \lambda = u$
  \item \label{item:strict-system-inequality}
    There exists a vector
    $v$ for which $\<a_i - u, v\> > 0$ for each $i$.
  \end{enumerate}
\end{lemma}
\begin{proof}
  By homogeneity, part~\eqref{item:strict-system-inequality} is equivalent to
  there existing $v$, $\ltwo{v} \le 1$, for which
  $\<a_i - u, v\> > 0$.
  Let $\Lambda = \{\lambda \in \R^m_+, \<\lambda, \ones\> = 1\}$.
  Then
  \begin{equation*}
    \sup_{v \in \ball_2^n} \inf_{\lambda \in \Lambda}
    v^T(A^T \lambda - u) = \inf_{\lambda \in \Lambda} \sup_{v \in \ball_2^n}
    v^T (A^T \lambda - u)
  \end{equation*}
  by standard min-max duality with bilinear functions.
  So if condition~\eqref{item:solve-convex-system} holds for some
  $\lambda\subopt$, then obviously $\inf_{\lambda \in \Lambda} \sup_{v
    \in \ball_2} v^T (A^T \lambda - u) \le \sup_{v \in \ball_2} v^T
  (A^T \lambda\subopt - u) = 0$, while if
  condition~\eqref{item:strict-system-inequality} held, we would have
  $\inf_{\lambda \in \Lambda} v^T (A^T \lambda - u) = \min_i \<a_i -
  u, v\> > 0$, a contradiction.
  Conversely, if \eqref{item:solve-convex-system} fails, then
  $\sup_{v \in \ball_2^n} v^T (A^T \lambda - u)
  = \ltwo{A^T \lambda - u} > 0$, and
  as $\Lambda$ is compact,
  $\inf_{\lambda \in \Lambda} \ltwo{A^T \lambda - u} > 0$.
  The value of the game is thus positive, and
  so $\sup_{v \in \ball_2^n} \min_i \<a_i - u, v\> > 0$, i.e.,
  \eqref{item:strict-system-inequality} holds.
\end{proof}

Combining Lemma~\ref{lemma:strong-positive-derivative}
and Lemma~\ref{lemma:bilinear-alternatives}
will allow us to prove
Lemma~\ref{lemma:positive-derivatives-directable}.
Take the set $\{z_i\}_{i = 1}^k$ that
Lemma~\ref{lemma:strong-positive-derivative} guarantees,
so that there is a $c > 0$ such that
for each $v \in \sphere^{d-1}$ there exists $i$
and $g_i \in \partial \loss_{z_i}(\theta)$ satisfying
\begin{equation*}
  \<g_i, v\> \ge c.
\end{equation*}
We will show that for $0 \le t < c$, there is a solution to $\sum_{i = 1}^m
\lambda_i g_i = t u$ with $\lambda \succeq 0, \<\ones, \lambda\> = 1$, which
is evidently equivalent to Lemma~\ref{lemma:strong-positive-derivative}.
Take $t < c$, which in turn implies
\begin{equation*}
  \<g_i - t u, v\> \ge c - t > 0. 
\end{equation*}
In particular, for any $v \in \sphere^{d-1}$, there
exist indices $i$ and $j$ for which
\begin{equation*}
  \<g_i - t u, v\> > 0 
  ~~ \mbox{and} ~~
  \<g_j - t u, -v\> > 0 
  ~~ \mbox{i.e.} ~~
  \<g_j - tu, v\> < 0. 
\end{equation*}
That is, for any $v \in \sphere^{d-1}$,
\begin{equation*}
  \max_{i \le k} \<g_i - tu, v\> > 0 
  ~~ \mbox{and} ~~
  \min_{i \le k} \<g_i - tu, v\> < 0. 
\end{equation*}
Lemma~\ref{lemma:bilinear-alternatives} thus demonstrates that $\sum_{i =
  1}^m \lambda_i g_i = tu$ with $\lambda \succeq 0$, $\<\ones, \lambda\> =
1$ has a solution.

\subsection{Proof of Proposition~\ref{proposition:uc-vs-c1}}
\label{subsection:proof-uc-vs-c1}

We will show the slightly stronger result that 
\begin{enumerate}
  [(i)]
  \item \label{item:uc-l-is-bad} Condition~\eqref{eqn:uc-l} is not
        equivalent to \ref{cond:compact-case} since \label{c1-vs-ucl} \ref{cond:compact-case}
        $\nRightarrow$ \eqref{eqn:uc-l} but \eqref{eqn:uc-l}
        $\Rightarrow$ \ref{cond:compact-case}.

  \item \label{item:uc-al-is-good} Condition~\eqref{eqn:uc-al} is
        equivalent to \ref{cond:compact-case}.
\end{enumerate}

To prove~\eqref{item:uc-l-is-bad} we will first show that
\eqref{eqn:uc-l}~$\Rightarrow$ \ref{cond:compact-case}, and then provide a
counterexample illustrating that the reverse implication does not
hold.
For~\eqref{item:uc-al-is-good} we note that the proof of
\eqref{eqn:uc-l}~$\Rightarrow$ \ref{cond:compact-case} also shows that
\eqref{eqn:uc-al}~$\Rightarrow$ \ref{cond:compact-case}, so we will then only prove the
remaining implication \ref{cond:compact-case} $\Rightarrow$~\eqref{eqn:uc-al}.

\paragraph{\eqref{eqn:uc-l}~or~\eqref{eqn:uc-al}~$\Rightarrow$ \ref{cond:compact-case}.}
We will prove this result by contrapositive, so we assume that \ref{cond:compact-case}
fails.
Then there exists some compact subset
$\Theta_0\subset \interior \Theta$ such that
$\sup_{\theta\in \Theta_0}M_\theta=\infty$.
We may without loss of generality assume that $\Theta_0$ is convex as
if this is not the case we could replace it with its convex hull,
which is still contained in the interior of $\Theta$.
There must now exist sequences $z_k\in\mc{Z}$ and
$\theta_k, \theta_k'\in \Theta_0$, with $\theta_k\neq \theta_k'$ such
that
\begin{equation*}
    \frac{\loss_{z_k}(\theta'_k)-\loss_{z_k}(\theta_k)}
    {\ltwo{\theta'_k-\theta_k}}\to\infty,
\end{equation*}
where we have assumed without loss of generality that
$\loss_{z_k}(\theta'_k)\geq \loss_{z_k}(\theta_k)$.
We now pick $r>0$ small enough that
$\Theta_r=\Theta_0+r\ball_2\subset \interior \Theta$, and define
the growth direction
$u_k=\frac{\theta_k'-\theta_k}{\ltwo{\theta_k'-\theta_k}}$.
We now let $\bar{\theta}_k=\theta'_k+ru_k$ and note that by
monotonicity of the secant slopes
\begin{equation*}
    \frac{\loss_{z_k}(\bar{\theta}_k)-\loss_{z_k}(\theta_k')}
    {r}
    \geq
    \frac{\loss_{z_k}(\theta'_k)-\loss_{z_k}(\theta_k)}
    {\ltwo{\theta'_k-\theta_k}},
\end{equation*}
so that
$\loss_{z_k}(\bar{\theta}_k)-\inf_{\theta\in \Theta}\loss_{z_k}(\theta)\to \infty$.

In particular, pick any fixed reference $z_0\in \mc{Z}$ and for a
given $n\geq 2$ let $k$ be large enough that
$\bar{\loss}_{z_k}(\bar{\theta}_k)-\bar{\loss}_{z_0}(\bar{\theta}_k)\geq n$.
This is possible since
$\sup_{\theta\in \Theta_r}\bar{\loss}_{z_0}(\theta)<\infty$ since
$z_0$ is fixed.
Now consider the distributions
\begin{equation*}
    P_{(n)}=\frac{1}{n}\delta_{z_k}+\left(1-\frac{1}{n}\right)\delta_{z_0}.
\end{equation*}
Under these, with probability
$\left(1-\frac{1}{n}\right)^n\geq \frac{1}{4}$ we see no samples $z_k$
and thus
\begin{equation*}
    (P_{(n)}-P_n)\bar{\loss}_z(\bar{\theta}_k)\geq 1.
\end{equation*}
This now implies that~\eqref{eqn:uc-al} fails on $\Theta_r$ since
$P_{(n)}\in \mc{P}_\loss(\mc{Z})$.

To get the failure of~\eqref{eqn:uc-l} recall that
$\loss_{z_k}(\bar{\theta}_k)-\loss_{z_k}(\theta_k')\to\infty$ so
similarly for any fixed reference $z_0$
\begin{equation*}
    \loss_{z_k}(\bar{\theta}_k)-\loss_{z_k}(\theta_k')
    \leq
    |\loss_{z_k}(\bar{\theta}_k)-\loss_{z_0}(\bar{\theta}_k)|
    +
    |\loss_{z_k}(\theta_k')-\loss_{z_0}(\theta_k')|
    +
    |\loss_{z_0}(\bar{\theta}_k)-\loss_{z_0}(\theta_k')|.
\end{equation*}
The last term is bounded since
$\bar\theta_k,\theta'_k\in \Theta_r\subset \interior\Theta$ so this
implies that there exists a choice
$\tilde\theta_k\in \{\bar{\theta}_k,\theta_k'\}$
\begin{equation*}
    |\loss_{z_k}(\tilde{\theta}_k)-\loss_{z_0}(\tilde{\theta}_k)|\to\infty,
\end{equation*}
and using the same form of distribution above, with an appropriate
choice of $k$, on the same bad event
\begin{equation*}
    |\poploss_{P_{(n)}}(\tilde{\theta}_k)-\poploss_{P_{n}}(\tilde{\theta}_k)|\geq 1,
\end{equation*}
so that~\eqref{eqn:uc-l} fails on $\Theta_r$.

\paragraph{\ref{cond:compact-case} $\nRightarrow$~\eqref{eqn:uc-l}.}
Consider the following loss function class
\begin{equation*}
  \loss_{z_1,z_2}(\theta)=z_1+\frac{1}{2}(\theta-z_2)^2,
\end{equation*}
with $z_1\in \R, z_2\in [-1,1]$, and $\Theta=[-1,1]$.
Then condition \ref{cond:compact-case} holds since
$\nabla \loss_{z_1,z_2}(\theta)=\theta-z_2 \in [-2,2]$.
Now for a fixed $n$ let $P_{(n)}$ be a distribution where $Z_2=0$
deterministically and
\begin{align*}
    P_{(n)}(Z_1=0)&=1-\frac{1}{n}
    &
    P_{(n)}(Z_1=n)&=\frac{1}{n}.
\end{align*}
With probability $\left(1-\frac{1}{n}\right)^{n}$ the sample we
observe contains no instances of $n$ and thus for any
$\theta\in \Theta$
\begin{equation*}
    |\poploss_{P_n}(\theta)-\poploss_{P_{(n)}}(\theta)|=1.
\end{equation*}
Therefore, for any compact $\Theta_0\subset \interior\Theta$ and
$n\geq 2$
\begin{equation*}
    \sup_{P}
    \E
    \left[
    \sup_{\theta\in \Theta_0}
    |\poploss_{P_n}(\theta)-\poploss_{P}(\theta)|
    \right]
    \geq
    \left(1-\frac{1}{n}\right)^{n}
    \geq
    \frac{1}{4}.
\end{equation*}

\paragraph{Proving \ref{cond:compact-case}$\Rightarrow$~\eqref{eqn:uc-al}.}
Take a compact subset $\Theta_0\subset \interior \Theta$ and once
again choose $r>0$ small enough that
$\Theta_r=\Theta_0+r\ball_2\subset \interior\Theta$.
Condition \ref{cond:compact-case} now yields that there is a finite $M(\Theta_r)$ such
that all losses $\loss_z$ are uniformly $M(\Theta_r)$-Lipschitz on
$\Theta_r$.
We will now first argue that all anchored losses are uniformly bounded
on all compact subsets of $\interior\Theta$.

Since we assumed $\Theta$ is compact let $D=\diam(\Theta)$ and take
$\theta^\star_z\in \argmin_{\theta'\in \Theta}\loss_z(\theta')$.
We now fix a reference point $\theta\in \Theta_0$ and take the point
$\theta'=\theta+\frac{r}{D}(\theta-\theta^\star_z)$ that is shifted
away from the minimizer.
This shifted point satisfies $\ltwo{\theta'-\theta}\leq r$, so it
belongs to our enlarged set $\Theta_r$, and
\begin{eqnarray}
    \theta
    =
    \frac{1}{1+\frac{r}{D}}\theta'+\frac{r/D}{1+\frac{r}{D}}\theta^\star_z.
\end{eqnarray}
Therefore, by convexity of $\loss_z$
\begin{equation*}
    \loss_z(\theta)
    \leq
    \frac{1}{1+\frac{r}{D}}\loss_z(\theta')
    +
    \frac{r/D}{1+\frac{r}{D}}\loss_z(\theta^\star_z),
\end{equation*}
and thus
\begin{equation*}
    \frac{r}{D}(\loss_z(\theta)-\loss_z(\theta^\star_z))
    \leq
    \loss_z(\theta')-\loss_z(\theta).
\end{equation*}
The leftmost expression corresponds to our scaled anchored losses and
thus by the Lipschitz guarantee afforded by \ref{cond:compact-case}
\begin{equation*}
    \bar{\loss}_z(\theta)\leq M(\Theta_r)D.
\end{equation*}

Since the anchored losses are bounded and uniformly Lipschitz on
$\Theta_r$, and thus on $\Theta_0\subset \R^d$, we pick a finite
$\delta$-covering of $\Theta_0$ with points $\{\theta_i\}_{i=1}^k$.
It now follows that
\begin{equation*}
    |(P_n-P)\bar\loss(\theta)|
    \leq
    \max_{i\in \{1,\cdots,k\}}|(P_n-P)\bar\loss(\theta_i)|+2M(\Theta_0)\delta,
\end{equation*}
and since we argued that $\bar\loss(\theta_i)\in [0, M(\Theta_r)D]$ we
can argue that for any $P$
\begin{equation*}
    \E_P[|(P_n-P)\bar\loss(\theta_i)|]
    \leq
    \frac{M(\Theta_r)D}{\sqrt{n}}.
\end{equation*}
Therefore,
\begin{equation*}
    \E_P\left[\sup_{\theta\in \Theta_0}|(P_n-P)\bar\loss(\theta)|\right]
    \leq \frac{kM(\Theta_r)D}{\sqrt{n}}+2M(\Theta_0)\delta,
\end{equation*}
so taking $n\uparrow\infty$ and then $\delta\downarrow 0$
yields~\eqref{eqn:uc-al}.

\subsection{Proof of Claim~\ref{claim:proper-losses}}
\label{subsec:proof-proper-losses}
We first prove that the class satisfies Condition \ref{cond:compact-case}.
Since the composed losses $\loss_y$ are convex and finite on
$\interior\Theta$ then for any compact
$\Theta_0\subset \interior\Theta$ there exists a finite constant
$M_{y}(\Theta_0)<\infty$ such that $\loss_y$ is
$M_{y}(\Theta_0)$-Lipschitz on $\Theta_0$.
There is then also a uniform
$M(\Theta_0)=\max_{y\in [k]}M_{y}(\Theta_0)<\infty$ since there
are finitely many classes.
This finite constant always exists on compact subsets of the
domain, but the losses need not be globally Lipschitz.

We now prove that the class also satisfies Condition~\ref{cond:unbounded-case}.
Pick any base prediction $\bar{p}\in \relint \Delta_k$ and an
associated $\bar{\theta}\in \interior \Theta$ such that
$q(\bar{\theta})=\bar{p}$, which exists since
$q(\interior \Theta)=\relint\Delta_k$.
For $\eta\in (0,1)$ define the set
\begin{equation*}
  K_\eta \defeq (1-\eta)\Delta_k + \eta \bar{p},
\end{equation*}
which is compact and contained in $\relint \Delta_k$.
By our assumptions, the inverse image $\Theta_\eta=q^{-1}(K_\eta)$
is a compact subset of $\interior \Theta$.
Now for any $\theta\in \Theta$ we can find
$\theta_\eta\in \Theta_\eta$ such that
\begin{equation*}
  q(\theta_\eta)=(1-\eta)q(\theta)+\eta \bar{p},
\end{equation*}
and by convexity of the regular strictly proper losses
\begin{equation*}
  \tilde{\loss}(q(\theta_\eta),y)
  \leq
  (1-\eta)\tilde{\loss}(q(\theta),y)
  +
  \eta\tilde{\loss}(\bar{p},y).
\end{equation*}
By the regularity of the losses there exists some $B<\infty$ such
that $|\tilde{\loss}(\bar{p},y)|\leq B$ for all $y\in \mc{Y}$.
Moreover, since $\bar{p}\in \relint \Delta_k$ and
$\tilde{\loss}(p,y)$ is convex and finite on $\relint \Delta_k$,
there exists at least one subgradient with finite norm at
$\bar{p}$.
This implies that there exists a constant $M_{\bar{p}}<\infty$
such that for any $p\in \Delta_k$
\begin{equation*}
  \tilde{\loss}(p,y)
  \geq
  \tilde{\loss}(\bar{p},y)
  -
  M_{\bar{p}}\ltwo{\bar{p}-p}
  \geq -(B + \sqrt{2}M_{\bar{p}}),
\end{equation*}
and
\begin{equation*}
  \tilde{\loss}(q(\theta_\eta),y)
  \leq
  \tilde{\loss}(q(\theta),y)
  +
  \eta(2B + \sqrt{2}M_{\bar{p}}).
\end{equation*}
Now for any fixed $\epsilon>0$ by choosing
$\eta \leq \frac{\epsilon}{2B+\sqrt{2}M_{\bar{p}}}$ we get
\begin{equation*}
  \inf_{\theta\in \Theta_\eta}\poploss_P(\theta)
  \leq
  \poploss_P^\star(\Theta)
  +
  \epsilon,
\end{equation*}
which is exactly Condition~\ref{cond:unbounded-case}.

\subsection{Proof of Lemma~\ref{lemma:tv-and-hellinger}}
\label{sec:proof-tv-and-hellinger}

Recall the squared Hellinger distance
$\dhel^2(P, Q) = \half \int (\sqrt{dP} - \sqrt{dQ})^2$.
The Hellinger and variation distances satisfy
Le Cam's inequalities~\cite[Lemma 2.3]{Tsybakov09}
\begin{equation*}
  \dhel^2(P, Q) \le \tvnorm{P - Q}
  \le \dhel(P, Q) \sqrt{2 - \dhel^2(P, Q)}
\end{equation*}
for any $P, Q$, and the Hellinger distance tensorizes via $\dhel^2(P^n,
Q^n) = 1 - (1 - \dhel^2(P, Q))^n$.
For the first statement,
\begin{align*}
  \tvnorm{P_0^n - P_1^n}
  & \le \dhel(P_0^n, P_1^n)
  \sqrt{2 - \dhel^2(P_0^n, P_1^n)} \\
  & = \sqrt{1 - (1 - \dhel^2(P_0, P_1))^n}
  \sqrt{1 + (1 - \dhel^2(P_0, P_1))^n} \\
  & \le \sqrt{1 - (1 - \gamma)^n} \sqrt{1 + (1 - \gamma)^n}
  = \sqrt{1 - (1 - \gamma)^{2n}}
\end{align*}
Taking $P = P_{0,n}$ and $Q = P_{1,n}$ and setting $\gamma = \frac{a}{n}$
yields
\begin{equation*}
  \tvnorm{P_{0,n}^n - P_{1,n}^n}
  \le \sqrt{1 - \left(1 - \frac{a}{n}\right)^{2n}}.
\end{equation*}
Take $n \to \infty$.

\subsection{Proof of Lemma~\ref{lemma:lipschitz-consts-monotone}}
\label{sec:proof-lipschitz-consts-monotone}

We prove the results about $\lipconst_\theta^-$.
We know that $\lipconst_\theta^- = \inf_{z \in \mc{Z}} \deriv_-
\loss_z(\theta)$ by Lemma~\ref{lemma:properties-of-derivatives},
part~\eqref{item:subdifferential-as-directional}.
We will demonstrate
that for all $t > 0$, $\lipconst_\theta^- \le \lipconst_{\theta + t}^-$.
If $\lipconst_{\theta + t}^- = -\infty$, then for any $\lipconst <
\infty$, there exists $z \in \mc{Z}$ such that $\deriv_- \loss_z(\theta)
\le - \lipconst$, and the monotonicity of $\theta \mapsto
\deriv_-\loss_z(\theta)$ (Lemma~\ref{lemma:properties-of-derivatives},
part~\eqref{item:continuity-directionals})
implies $\lipconst_\theta^- \le \deriv_- \loss_z(\theta) \le
-\lipconst$.
As $\lipconst$ was arbitrary, we have $\lipconst_\theta^- =
-\infty$.
If $\lipconst_{\theta + t}^- > -\infty$, then for any $\epsilon
> 0$, there exists $z$ such that $\deriv_- \loss_z(\theta + t) \le
\lipconst_{\theta + t}^- + \epsilon$.  For this $z$, $\lipconst_\theta^-
\le \deriv_- \loss_z(\theta) \le \deriv_- \loss_z(\theta + t) \le
\lipconst_{\theta + t}^- + \epsilon$, proving the monotonicity.
The monotonicity of $\lipconst_\theta^+$ is similar.

\subsection{Proof of Lemma~\ref{lemma:discrete-to-borel}}
\label{sec:proof-discrete-to-borel}

We first demonstrate that it is no loss of generality to assume
that the losses $\loss_z$ are nonnegative.
Because point masses $\pointmass_z \in \Pdiscrete(\mc{Z})$ trivially, we see
that Condition~\ref{cond:unbounded-case} implies that
$\loss_z\opt(\Theta) > -\infty$ for all $z \in \mc{Z}$, as
$\loss_z\opt(\Theta_0) > -\infty$ by propriety of the losses.
Now we show that we may reduce to the case that the losses
are non-negative.
Define the rescaled losses
$\wb{\loss}_z = \loss_z - \loss_z\opt(\Theta_0)$,
so that $\wb{\loss}_z \ge 0$ on $\Theta_0$,
$\wb{\loss}_z$ is still $z$-measurable, and
\begin{equation*}
  \wb{\loss}_z(\theta) \ge -\epsilon
  ~~ \mbox{for~all~} \theta \in \Theta, z \in \mc{Z}.
\end{equation*}
Condition~\ref{cond:unbounded-case} includes
Condition~\ref{cond:compact-case}, so for any $\theta_0 \in \interior
\Theta$, there is a Lipschitz constant $\lipconst_0$ such that $\lipconst_0
\ge \sup_{g,z}\{\ltwo{g} \mid g \in \partial \loss_z(\theta_0)\}$.
Thus, assuming w.l.o.g.\ that $\theta_0 \in \Theta_0$ by increasing the
compact set $\Theta_0$ if necessary, for any $g_0 \in \partial
\loss_z(\theta_0)$ we have the lower bound
\begin{equation*}
  \loss_z(\theta)
  \ge \loss_z(\theta_0)
  + \<g_0, \theta - \theta_0\>
  \ge \loss_z(\theta_0) - \lipconst_0 \ltwo{\theta - \theta_0}
\end{equation*}
so
\begin{equation*}
  \loss_z\opt(\Theta_0) \ge \loss_z(\theta_0) - \lipconst_0
  \cdot \diam(\Theta_0).
\end{equation*}
Then $-\infty < \E_P[\loss_Z\opt(\Theta_0)] < \infty$,
and so for any $\theta_0 \in \Theta_0$, $\theta \in \Theta$, and
any $P \in \mc{P}_\loss(\mc{Z})$,
\begin{align*}
  \poploss_P(\theta_0) - \poploss_P(\theta)
  & = \E_P\left[\loss_Z(\theta_0) - \loss_Z\opt(\Theta_0)\right]
  - \E_P\left[\loss_Z(\theta) - \loss_Z\opt(\Theta_0)\right] \\
  & = \E_P\left[\wb{\loss}_Z(\theta_0)\right]
  - \E_P\left[\wb{\loss}_Z(\theta)\right].
\end{align*}
Adding an additional $\epsilon > 0$, we see it is no loss of generality to
assume that the losses are nonnegative, $\loss_z \ge 0$.


\newcommand{\epito}{\stackrel{\textup{epi}}{\rightarrow}}

With the w.l.o.g.\ nonnegativity assumption, we
can show the extension to Borel measures from discrete.
Recall that functions $f_n$ epi-converge to $f$,
denoted $f_n \epito f$, if
whenever $\theta_n \to \theta$,
\begin{equation*}
  \liminf_n f_n(\theta_n) \ge f(\theta)
\end{equation*}
and for each $\theta$, there exists a sequence $\theta_n \to \theta$
such that
\begin{equation*}
  \limsup_{n \to \infty} f_n(\theta_n) \le f(\theta).
\end{equation*}
\citet[Thm.~7.49]{ShapiroDeRu14} show that
$\poploss_{P_n} \epito
\poploss_P$ with probability 1 for all $P \in \mc{P}_\loss$,
and moreover, for any compact $\Theta_0$,
they also show~\cite[Thm.~5.4]{ShapiroDeRu14} that
$\poploss_{P_n}\opt(\Theta_0) \to \poploss_P\opt(\Theta_0)$
with probability 1.
%
%
As Condition~\ref{cond:unbounded-case} holds,
let $\epsilon > 0$ be arbitrary and assume $\Theta_0$ satisfies
$\poploss_Q\opt(\Theta_0) \le \poploss_Q\opt(\Theta) + \epsilon$ for
all discrete $Q$.
Then taking $P_n$ to be the empirical distribution of $\{Z_i\}_{i=1}^n$
drawn i.i.d.\ $P$, which is discrete, we may assume that we have
\begin{equation*}
  \poploss_{P_n} \epito \poploss_P
  ~~ \mbox{and} ~~
  \poploss_{P_n}\opt(\Theta_0) \to
  \poploss_P\opt(\Theta_0)
\end{equation*}
for compact $\Theta_0$.
Because $\poploss_P$ is lower semicontinuous and bounded below
(by the w.l.o.g.\ assumption that $\loss_z \ge 0$),
there exists $\theta \in \Theta$ such that
$\poploss_P(\theta) \le \poploss_P\opt(\Theta) + \epsilon$.
For this $\theta$, epi-convergence implies
there exists a sequence $\theta_n \to \theta$,
$\theta_n \in \Theta$,
for which $\limsup_n \poploss_{P_n}(\theta_n) \le \poploss_P(\theta)$.
Thus
\begin{align*}
  \poploss_P\opt(\Theta_0)
  & = \lim_n \poploss_{P_n}\opt(\Theta_0)
  \stackrel{(i)}{\le} \limsup_n \poploss_{P_n}\opt(\Theta) + \epsilon
  \le \limsup_n \poploss_{P_n}(\theta_n) + \epsilon
  \stackrel{(ii)}{\le} \poploss_P(\theta) + \epsilon,
\end{align*}
where inequality~$(i)$ follows by Condition~\ref{cond:unbounded-case}
and inequality~$(ii)$ by
$\poploss_{P_n} \epito \poploss_P$.
By assumption, $\poploss_P(\theta) \le \poploss_P\opt(\Theta) + \epsilon$,
so
\begin{align*}
  \poploss_P\opt(\Theta_0) \le \poploss_P\opt(\Theta) + 2 \epsilon.
\end{align*}

\subsection{Proof of Lemma~\ref{lemma:achievable-minimizers}}
\label{sec:proof-achievable}

We recall a few results on differentiability and closedness
of stochastic optimization functionals~\cite{Bertsekas73, ShapiroDeRu14}.
When $P \in \mc{P}_\loss$, the probability distributions
making $\poploss_P$ well-defined over $\Theta$,
$\theta \mapsto \poploss_P(\theta)$ is closed convex whenever
$\loss_z(\cdot)$ is closed convex for each $z \in \mc{Z}$,
and moreover, we can interchange directional differentiation
and integration (see, e.g.,~\cite{Bertsekas73}
or \cite[Thm.~7.52]{ShapiroDeRu14}):
\begin{equation*}
  \deriv_+ \poploss_P(\theta)
  = \int \deriv_+ \loss_z(\theta) dP(z)
  ~~ \mbox{and} ~~
  \deriv_- \poploss_P(\theta)
  = \int \deriv_- \loss_z(\theta) dP(z).
\end{equation*}

Recalling the definition~\eqref{eqn:alt-extreme-minimizers} of
$\thetamin$ and $\thetamax$ in terms of directional derivatives, we
have
\begin{equation*}
  \thetamin = \inf\left\{\theta \in \Theta \mid
  \sup_{z \in \mc{Z}} \deriv_+ \loss_z(\theta) > 0
  \right\}
  ~~ \mbox{and} ~~
  \thetamax = \sup\left\{\theta \in \Theta \mid
  \inf_{z \in \mc{Z}} \deriv_- \loss_z(\theta) < 0
  \right\}.
\end{equation*}
Thus,
the swapping of integration and directional differentiation implies
\begin{align*}
  \deriv_+ \poploss_P(\thetamax)
  = \int \deriv_+ \loss_z(\thetamax) dP(z)
  \stackrel{(\star)}{=} \int \lim_{\theta \downarrow \thetamax}
  \underbrace{\deriv_+\loss_z(\theta)}_{\ge 0} dP(z)
  \ge 0,
\end{align*}
where equality~$(\star)$ follows from
Lemma~\ref{lemma:properties-of-derivatives}.
Similarly,
\begin{equation*}
  \deriv_- \poploss_P(\thetamin) \le 0.
\end{equation*}
If $\thetamax = \thetamin$, these equalities
imply that $\theta\opt = \thetamax = \thetamin$ minimizes
$\poploss_P$, and as the choice of $P$ was arbitrary,
this implies part~\eqref{item:singleton-achievable}
in that case.
For $\theta < \thetamin$, we have
$\sup_{z \in \mc{Z}} \deriv_+ \loss_z(\theta) \le 0$,
while $\theta > \thetamax$ implies
$\inf_{z \in \mc{Z}} \deriv_- \loss_z(\theta) \ge 0$.
In particular, $0 \le \deriv_- \loss_z(\theta) \le \deriv_+ \loss_z(\theta)
\le 0$ for all $z$, that is, $\loss_z'(\theta) = 0$ for any $\thetamax <
\theta < \thetamin$, completing the proof of
part~\eqref{item:singleton-achievable} once we recognize that
\begin{equation*}
  \deriv_+ \loss_z(\thetamax)
  \le 0 = \deriv_- \loss_z(\theta)
  = \deriv_+ \loss_z(\theta)
  \le \deriv_- \loss_z(\thetamin)
\end{equation*}
for $\thetamax < \theta < \thetamin$.

Now we consider the cases that $\thetamin \le \thetamax$.
Let $\Theta\opt(P) = \argmin_{\theta \in \Theta} \poploss_P(\theta)$;
because $\poploss_P$ is closed convex, its minimizers are attained (once we
allow extended reals).
We consider the left endpoint $\thetamin$, as the arguments for the
right endpoint are similar, and wish to show
that $[\thetamin, \infty] \cap \Theta\opt(P) \neq \emptyset$.
If $\thetamin = \inf\{\theta \in \Theta\}$, there is nothing to show,
so assume instead that $\thetamin > \inf\{\theta \in \Theta\}$,
whence $-\infty < \poploss_P(\thetamin) < \infty$.
Then for $\theta < \thetamin$,
assuming $\poploss_P(\theta) < \infty$,
we have $\sup_{z \in \mc{Z}} \deriv_+ \loss_z(\theta) \le 0$,
and applying Lemma~\ref{lemma:integrability-convex},
\begin{align*}
  \inf_{\theta \in \Theta}
  \poploss_P(\theta)
  \le \poploss_P(\thetamin)
  = \poploss_P(\theta)
  + \int_{\theta}^{\thetamin} \underbrace{\deriv_+ \poploss_P(t)}_{\le 0} dt
  \le \poploss_P(\theta).
\end{align*}
Performing the same argument for $\thetamax$, we see that
if $\theta \not \in [\thetamin, \thetamax]$, then
\begin{equation*}
  \poploss_P(\theta) \ge
  \inf_{\theta \in [\thetamin, \thetamax]} \poploss_P(\theta),
\end{equation*}
yielding parts~\eqref{item:minimizer-intersection}
and~\eqref{item:project-back} of the lemma.

Finally, we turn to the last
claim~\eqref{item:old-definition-of-achievable}.
Let
\begin{equation*}
  \Theta^*
  \defeq \left\{\theta \in \Theta
  \mid \inf_{z \in \mc{Z}}
  \deriv_- \loss_z(\theta) < 0,
  \sup_{z \in \mc{Z}}
  \deriv_+ \loss_z(\theta) > 0
  \right\}
\end{equation*}
denote the variant of the set on the right side of the
claimed equality~\eqref{eqn:old-achievable-minimizers}
whose closure we have not taken.
When $\thetamin < \thetamax$, for any $\theta \in (\thetamin, \thetamax)$
there exist $z_0, z_1$ such that
$\deriv_+ \loss_{z_0}(\theta) > 0$ and
$\deriv_- \loss_{z_1}(\theta) < 0$,
so $\interior \achievable \subset \Theta^*$.
For the converse, observe that because $\achievable$ has interior,
$\Theta^*$ is non-empty. It is also an interval: if
$\theta_0 < \theta_1$ and for each $i \in \{0, 1\}$, there exists
$z_i^{\pm}$ satisfying $\deriv_- \loss_{z_i^-}(\theta_i) < 0$ and $\deriv_+
\loss_{z_i^+}(\theta_i) > 0$, then for $\theta_0 < \theta < \theta_1$, the
monotonicity of subgradients (Lemma~\ref{lemma:properties-of-derivatives},
part~\eqref{item:monotonicity-subdifferential}) implies that
\begin{equation*}
  \deriv_- \loss_{z_1^-}(\theta)
  \le \deriv_- \loss_{z_1^-}(\theta_1) < 0
  ~~ \mbox{and} ~~
  \deriv_+ \loss_{z_0^+}(\theta)
  \ge \deriv_+ \loss_{z_0^+}(\theta_0) > 0,
\end{equation*}
so that $\theta \in \Theta^*$.
If $\theta \in \Theta^*$, then
clearly $\thetamin \le \theta \le \thetamax$ by construction,
implying $\cl \Theta^* \subset \achievable$.

\subsection{Proof of Lemma~\ref{lemma:achievable-as-all-argmins}}
\label{sec:proof-achievable-as-all-argmins}

Let $C\opt = \cap\{C \in \allargmins(\loss, \mc{Z}, \Theta)\}$
be the set on the right side of the equality
in Lemma~\ref{lemma:achievable-as-all-argmins}.
We first show that $\achievable \subset C\opt$.
Assume w.l.o.g.\ that $\achievable \neq \emptyset$ and
let $\theta \in \achievable$.
If $\achievable = \{\theta\}$, then
part~\eqref{item:singleton-achievable} of
Lemma~\ref{lemma:achievable-minimizers}
gives that $\theta \in C\opt$.
It is straightforward to see that if $\theta \in \interior \achievable$,
then $\theta \in C$ for each $C \in \allargmins(\loss, \mc{Z}, \Theta)$.
Indeed, by definition~\eqref{eqn:achievable-minimizers}, there exist $z^+$
and $z^- \in \mc{Z}$ such that $\deriv_+ \loss_{z^+}(\theta) > 0$ and
$\deriv_- \loss_{z^-}(\theta) < 0$.
Set $p$ to solve
$p \deriv_+ \loss_{z^+}(\theta) + (1 - p) \deriv_- \loss_{z^-}(\theta) = 0$,
and then define the distribution
$P = p \pointmass_{z^+} + (1 - p) \pointmass_{z^-}$.
Evidently, $\theta$ minimizes $\poploss_P$.
Because $C\opt$ is closed, it must therefore contain
$\cl \achievable$.

Suppose $\theta \not \in \achievable = [\thetamin, \thetamax]$,
and w.l.o.g.\ assume $\theta < \thetamax$.
Then by part~\eqref{item:minimizer-intersection}
of Lemma~\ref{lemma:achievable-minimizers},
$\sup \{\theta' \in \argmin \poploss_P\} \ge \thetamin > \theta$
for all distributions $P$.
The closed set $C = [\thetamin, \infty]$ thus belongs
to $\allargmins(\loss, \mc{Z}, \Theta)$, and
$\theta \not \in C$.
We have shown $\achievable^c \subset (C\opt)^c$,
that is, $C\opt \subset \achievable$ as desired.

\subsection{Proof of Lemma~\ref{lemma:discrete-or-borel}}
\label{sec:proof-discrete-or-borel}

If $\mc{P}_\loss$ contains well-defined probability measures, then standard
consistency results for stochastic
approximation~\cite[Thm.~5.4]{ShapiroDeRu14}
show that if $P_n$ denotes the empirical measure
for $Z_i \simiid P$, then
whenever $\sublevel_0(\poploss_P)
= \argmin_{\theta \in \Theta} \poploss_P(\theta)$ is
non-empty and bounded, then
\begin{equation*}
  \poploss_{P_n}\opt(\Theta) \cas \poploss_P\opt(\Theta)
  ~~ \mbox{and} ~~
  \poploss_{P_n}\opt(\Theta \cap H_{v,t})
  \cas \poploss_P\opt(\Theta \cap H_{v,t}).
\end{equation*}
In particular, if $H_{v,t}$ is unconstraining for all finitely
supported measures, then for all $\epsilon > 0$, there exists
a finitely supported measure $Q$ such that
\begin{equation*}
  \left|\poploss_P\opt(H_{v,t} \cap \Theta)
  - \poploss_Q\opt(H_{v, t} \cap \Theta)\right|
  \le \epsilon
  ~~ \mbox{and} ~~
  \left|\poploss_P\opt(\Theta) - \poploss_Q\opt(\Theta)\right|
  \le \epsilon.
\end{equation*}
In particular,
\begin{equation*}
  \left|\poploss_P\opt(H_{v,t} \cap \Theta) - \poploss_P\opt(\Theta)\right|
  \le 2 \epsilon
\end{equation*}
for all $\epsilon > 0$, that is,
$\poploss_P\opt(H_{v,t} \cap \Theta) = \poploss_P\opt(\Theta)$.

\subsection{Proof of Lemma~\ref{lemma:positive-derivative-eliminable}}
\label{sec:proof-positive-derivative-eliminable}

Let $E(\theta_0, \loss) \cap \Theta = \{\theta_0\}$, which
is equivalent to the
condition that
\begin{equation*}
  \Theta \setminus \{\theta_0\} \subset E(\theta_0, \loss)^c
  = \left\{\theta \mid \mbox{there~exists}~ z \in \mc{Z} ~ \mbox{s.t.}~
  \deriv \loss_z(\theta_0; \theta - \theta_0) < 0 \right\}.
\end{equation*}
%
%
For each $\theta_1 \in \Theta \setminus \{\theta_0\}$, there exists
$z \in \mc{Z}$ for which
$\deriv \loss_z(\theta_0; \theta_1 - \theta_0) < 0$.
As a consequence, because for any $v \in \sphere^{d-1}$
we may take $\theta_1 = \theta_0 + t v \in
\Theta$ for some $t > 0$, there exists $z$ for which $\deriv
\loss_z(\theta_0; v)
= \frac{1}{t} \deriv \loss_z(\theta_0; \theta_1 - \theta_0) < 0$.
For this $z$, we obtain
\begin{equation*}
  0 > \deriv \loss_z(\theta_0; v)
  = \sup_{g \in \partial \loss_z(\theta_0)} \<g, v\>
  = -\inf_{g \in \partial \loss_z(\theta_0)} \<g, -v\>
  = -\deriv_- \loss_z(\theta_0; -v),
\end{equation*}
i.e., $\deriv_- \loss_z(\theta_0; -v) > 0$.
So $\deriv \loss_z(\theta_0; -v) \ge \deriv_- \loss_z(\theta_0; -v)
> 0$, and as $v \in \sphere^{d-1}$ was arbitrary (so we may as well
have chosen $-v$), we see that
for any $v \in \sphere^{d-1}$, there exists
$z \in \mc{Z}$ for which $\deriv \loss_z(\theta_0; v) > 0$.
That is, Definition~\ref{definition:achievable-point} holds at $\theta_0$.

For the converse, assume that $\theta \in \interior \Theta$
satisfies Definition~\ref{definition:achievable-point}.
Then for $\theta_0 \in \interior \Theta$,
$\theta_0 \neq \theta$, setting $v = (\theta_0 - \theta) /
\ltwo{\theta - \theta_0}$, there exists $z$ for which
$\deriv \loss_z(\theta; v) > 0$.
The function $h(t) \defeq \loss_z(t \theta_0 + (1 - t) \theta)$
satisfies
\begin{equation*}
  \partial h(t) = \left[\deriv_- \loss_z(t \theta_0 + (1 - t) \theta;
    \theta_0 - \theta),
    \deriv \loss_z(t \theta_0 + (1 - t) \theta;
    \theta_0 - \theta)\right]
\end{equation*}
and $\partial h(t)$ is increasing (or at least non-decreasing) in $t$
(see~\cite[Ch.~VI.2.3]{HiriartUrrutyLe93}).
In particular,
$\deriv_- \loss_z(\theta_0; \theta_0 - \theta) \ge \deriv \loss_z(\theta;
\theta_0 - \theta) > 0$.
Rewriting this,
\begin{equation*}
  0 < \deriv_- \loss_z(\theta_0; \theta_0 - \theta)
  = \inf_{g \in \partial \loss_z(\theta_0)} \<g, \theta_0 - \theta\>
  = -\sup_{g \in \partial \loss_z(\theta_0)} \<g, \theta - \theta_0\>
  = - \deriv \loss_z(\theta_0; \theta - \theta_0).
\end{equation*}
That is, $\deriv \loss_z(\theta_0; \theta - \theta_0) < 0$,
and $\theta \not \in E(\theta_0, \loss)$ for any $\theta_0 \neq \theta,
\theta_0 \in \interior \Theta$.

\subsection{Proof of Lemma~\ref{lemma:dominating-and-achievable}}
\label{sec:proof-dominating-and-achievable}
For each pair $\theta_0, \theta_1 \in \Theta$ with
$\theta_1 \neq \theta_0$, 
there exists $z$ such that $\loss_z(\theta_1) > \loss_z(\theta_0)$.
Written differently, for each $\theta \in \interior \Theta$
and each $v \in \sphere^{d-1}$ and $t > 0$
such that $\theta - tv \in \Theta$,
there exists $z$ such that
\begin{equation*}
  \frac{\loss_z(\theta) - \loss_z(\theta - tv)}{t} > 0.
\end{equation*}
Because $\loss_z$ is subdifferentiable at $\theta \in \interior \Theta$,
$\loss_z(\theta - tv) \ge \loss_z(\theta) - t \<g, v\>$ for each
$g \in \partial \loss_z(\theta)$,
that is, $t \<g, v\> \ge \loss_z(\theta) - \loss_z(\theta - tv)$.
Rearranging the preceding display,
\begin{equation*}
  \<g, v\> \ge \frac{\loss_z(\theta) - \loss_z(\theta - tv)}{t}
  > 0
  ~~ \mbox{for~all~} g \in \partial \loss_z(\theta),
\end{equation*}
and so $\deriv_- \loss_z(\theta; v)
= \inf_{g \in \partial \loss_z(\theta)} \<g, v\> > 0$.

For the converse, let $\theta_0$ satisfy 
Definition~\ref{definition:achievable-point} and
$\theta \neq \theta_0$, $\theta \in \Theta$.
Then for $v = (\theta - \theta_0) / \ltwo{\theta - \theta_0}$,
we take a $z$ for which $\deriv \loss_z(\theta_0; v) > 0$,
and then
\begin{equation*}
  \loss_z(\theta)
  = \loss_z(\theta_0) + \ltwo{\theta - \theta_0}
  \cdot \int_0^1 \underbrace{\deriv \loss_z(\theta_0
    + t (\theta - \theta_0); v)}_{\ge \deriv \loss_z(\theta_0; v) > 0} dt
  \ge \loss_z(\theta_0)
  + \ltwo{\theta_0 - \theta} \deriv \loss_z(\theta_0; v).
\end{equation*}
So $\theta \not \in \dominating(\theta_0)$.


\end{document}